\documentclass[ejs]{imsart}
\RequirePackage{amsthm,amsmath,amsfonts,amssymb}
\RequirePackage[numbers]{natbib}
\RequirePackage[colorlinks,citecolor=blue,urlcolor=blue]{hyperref}
\RequirePackage{graphicx}
\usepackage{bbm}
\usepackage{bookmark}
\usepackage{bm}
\usepackage{algorithm}
\usepackage{algpseudocode}

\startlocaldefs
\theoremstyle{plain}

\newtheorem{theorem}{Theorem}[section]
\newtheorem{proposition}[theorem]{Proposition}
\newtheorem{corollary}[theorem]{Corollary}
\newtheorem{lemma}[theorem]{Lemma}
\theoremstyle{definition}
\newtheorem{definition}[theorem]{Definition}

\newtheorem{remark}{Remark}
\theoremstyle{remark}

\newcommand{\twopiece}[5][0]{% use option [1] if display is desired; default [0] is not in display
    \ifcase#1
        \left\{\begin{array}{ll}{#2}&{\text{ if } #3}\\{#4}&{\text{ if } #5}\end{array}\right.
    \else
        \left\{\begin{array}{ll}{#2}&{\text{ if } #3}\vspace{#1pt}\\{#4}&{\text{ if } #5}\end{array}\right.
\fi}

\newcommand{\E}{{\mathbb E}}

\renewcommand{\P}{{\mathbb P}}

\newcommand{\A}{{\mathcal{A}}}

\newcommand{\one}{{\mathbbm{1}}}

\newcommand{\T}{{\mathcal{T}}}

\newcommand{\lin}{{\mathcal{L}}}
\newcommand{\lf}{{\mathfrak{L}}}
\newcommand{\Ps}{{\mathcal{P}}}

\newcommand{\F}{{\cal F}}
\newcommand{\samp}{{\mathcal{X}}}

\newcommand{\ph}{{\mathfrak{p}}}
\newcommand{\qh}{{\mathfrak{q}}}

\newcommand{\smud}{\mathcal{P}_{\mathrm{SMU}}(d)}
\newcommand{\smumle}{\hat{p}_{n, d}^{\mathrm{SMU}}}

\newcommand{\helconst}{32}

\newcommand{\classunidec}{\mathcal{P}(1)}

\newcommand{\ben}{\begin{enumerate}}
\newcommand{\een}{\end{enumerate}}
\def\qt#1{\qquad\text{#1}}

\def\argmax{\mathop{\rm argmax}}

\endlocaldefs

\begin{document}
\begin{frontmatter}
\title{Convergence rates for estimating multivariate scale mixtures of uniform densities}
%\title{A sample article title with some additional note\thanksref{t1}}
\runtitle{SMU density estimation rates}
%\thankstext{T1}{A sample additional note to the title.}

\begin{aug}
%%%%%%%%%%%%%%%%%%%%%%%%%%%%%%%%%%%%%%%%%%%%%%%
%% Only one address is permitted per author. %%
%% Only division, organization and e-mail is %%
%% included in the address.                  %%
%% Additional information can be included in %%
\author[A]{\fnms{Arlene K. H.}~\snm{Kim}\ead[label=e1]{arlenent@korea.ac.kr}},
\author[B]{\fnms{Gil}~\snm{Kur}\ead[label=e2]{gil.kur@inf.ethz.ch}}
\and
\author[C]{\fnms{Adityanand}~\snm{Guntuboyina}\ead[label=e3]{aditya@stat.berkeley.edu}\orcid{0009-0000-2674-4065}}  
%%%%%%%%%%%%%%%%%%%%%%%%%%%%%%%%%%%%%%%%%%%%%%
%% Addresses                                %%
%%%%%%%%%%%%%%%%%%%%%%%%%%%%%%%%%%%%%%%%%%%%%%
\address[A]{Department of Statistics,
Korea University\printead[presep={,\ }]{e1}}

\address[B]{Institute for Machine Learning,
ETH Z\"urich\printead[presep={,\ }]{e2}}

\address[C]{Department of Statistics,
University of California at Berkeley\printead[presep={,\ }]{e3}}
\runauthor{A. K. H. Kim, G. Kur and A. Guntuboyina}
\end{aug}

\begin{abstract}
    The Grenander estimator is a well-studied
  procedure for univariate nonparametric density estimation. It is
  usually defined as the Maximum Likelihood Estimator (MLE) over the class
  of all non-increasing densities on the positive real line. It can
  also be seen as the MLE over the class of
  all scale mixtures of uniform densities. Using the latter viewpoint,
  Pavlides and Wellner~\cite{pavlides2012nonparametric} proposed a
  multivariate extension of the  
  Grenander estimator as the nonparametric MLE over the class of all
  multivariate scale mixtures of uniform densities. We prove that this
  multivariate estimator achieves the univariate cube root rate of
  convergence with only a logarithmic multiplicative factor that
  depends on the dimension. The usual curse of dimensionality is 
  therefore avoided to some extent for this multivariate
  estimator. This result positively resolves a conjecture of
  Pavlides and Wellner~\cite{pavlides2012nonparametric} under an
  additional lower bound   assumption. Our proof proceeds via a
  general accuracy result for the   Hellinger accuracy of MLEs over
  convex classes of densities. We also provide algorithms for
  computing the estimator, and illustrate performance on real and
  simulated datasets. 
\end{abstract}



\end{frontmatter}
%%%%%%%%%%%%%%%%%%%%%%%%%%%%%%%%%%%%%%%%%%%%%%
%% Please use \tableofcontents for articles %%
%% with 50 pages and more                   %%
%%%%%%%%%%%%%%%%%%%%%%%%%%%%%%%%%%%%%%%%%%%%%%
% \tableofcontents

\section{Introduction}
The Grenander estimator \cite{grenander1956theory}  is a popular procedure
for univariate nonparametric density estimation. Given positive observations
$x_1, \dots, x_n$ for some $n \geq 2$, the Grenander estimator
$\hat{p}_n$ is defined as the Maximum Likelihood Estimator (MLE) over
the class of all  nonincreasing densities on $(0, \infty)$. More precisely
\begin{equation*}
  \hat{p}_n := \argmax_{p \in \classunidec} \frac{1}{n} \sum_{i=1}^n \log p(x_i)
\end{equation*}
where $\classunidec$ is the class of all univariate density functions on the
 positive real line $(0, \infty)$
which are nonincreasing. Basic properties of the Grenander estimator
(including existence, uniqueness, efficient computation as well as
applications) can be found in the books
\cite{groeneboom2014nonparametric} and \cite{BBBB72}.

The Grenander estimator can also be seen as the MLE over the class of
scale mixtures of uniform densities. More specifically, consider the
class $\mathcal{P}_{\mathrm{SMU}}(1)$  consisting of all densities $p$ on $(0, \infty)$
that can be written, for every $u > 0$,  as
\begin{equation}\label{smu1}
  p(u) := \int_0^{\infty} p_{\mathrm{Unif(0, \theta]}}(u) dG(\theta)
  = \int_0^{\infty} \frac{\one\{u \leq \theta\}}{\theta} dG(\theta)
\end{equation}
for some probability measure $G$ on $(0, \infty)$. Here
$p_{\mathrm{Unif(0, \theta]}}(u) := \theta^{-1}\one\{u \leq \theta\}$ is
the uniform density on $(0, \theta]$. A density of the form
\eqref{smu1} is referred to as a scale mixture of uniform densities
because the mixture is over the scale parameter $\theta$ (the
subscript SMU in $\mathcal{P}_{\mathrm{SMU}}(1)$ refers to ``Scale
Mixture of Uniform''). The 
Grenander estimator maximizes likelihood over $\mathcal{P}_{\mathrm{SMU}}(1)$ 
because $\mathcal{P}_{\mathrm{SMU}}(1)$ and
$\classunidec$ are essentially the same density function
class. Indeed, it is easy to see that $\mathcal{P}_{\mathrm{SMU}}(1) \subseteq
\classunidec$ and, conversely, every density in $\classunidec$ that is
also upper semi-continuous belongs to $\mathcal{P}_{\mathrm{SMU}}(1)$ (see
\cite{williamson1956multiply} for a proof). 

Many authors have studied theoretical convergence properties of
$\hat{p}_n$ under the assumption that the observations $x_1, \dots, x_n$ are realizations of
independent random variables $X_1,\dots, X_n$ having a common density
$p_0 \in \mathcal{P}_{\mathrm{SMU}}(1)$. In this case,
$\hat{p}_n$ is a decently accurate estimator of $p_0$, especially when $n$ is
large. More precisely, it is well-known that the risk of $\hat{p}_n$
under the squared Hellinger loss function, defined for two densities
$p$ and $q$ as:
\begin{equation}
  \label{helldef}
h^2(p, q) := \int (\sqrt{p} - \sqrt{q})^2
\end{equation}
converges to zero at the
rate $n^{-2/3}$ under mild additional assumptions on $p_0$ (see e.g.,
\cite[Theorem 7.12]{VandegeerBook}; these mild  
additional assumptions will be satisfied if, for example,
$p_0$ is bounded from above and has compact support). Similar results
exist for the total variation loss function (see e.g.,
\cite{Birge89}):
\[
TV(p,q) := \int |p-q|,
\]
as well as for the convergence of $\hat{p}_n(x_0)$
to $p_0(x_0)$ for fixed points $x_0$ (see e.g., \cite[Chapter
3]{groeneboom2014nonparametric}). The rate $n^{-2/3}$ cannot be
improved in a minimax sense (see e.g., \cite{birge1987estimating,
  Birge87, G85}) although when $p_0 \in \mathcal{P}_{\mathrm{SMU}}(1)$ is piecewise
constant with a finite number of constant pieces, the rate of
convergence of $\hat{p}_n$ to $p_0$ is parametric (i.e., $n^{-1}$)
upto logarithmic factors in the squared Hellinger distance (see
\cite[Page 113]{VandegeerBook}; analogous results for the total
variation distance can be found in \cite{Birge89}).

Our paper studies rates of convergence for a multivariate extension of
the Grenander estimator that was originally proposed and studied by
Pavlides and Wellner \cite{pavlides2012nonparametric} (henceforth, we
shall use PW to refer to the paper
\cite{pavlides2012nonparametric}). For a fixed $d \geq 1$,  
PW defined the class $\smud$ consisting
of  all densities $p$ on $(0, \infty)^d$ that can be written, for
every $u_1, \dots, u_d > 0$, as  
\begin{equation}\label{pwsmudef}
  \begin{split}
  p(u_1, \dots, u_d) = \int_0^{\infty}  \dots
  \int_0^{\infty} p_{\mathrm{Unif(0, 
      \theta_1]}}(u_1) \dots
  p_{\mathrm{Unif(0, \theta_d]}}(u_d) dG(\theta_1,  \dots,
    \theta_d)
  \end{split}  
\end{equation}
for some probability measure $G$ on $(0,
\infty)^d$. PW argued that $\smud$  is a natural multivariate analog of the
univariate class $\mathcal{P}_{\mathrm{SMU}}(1)$. For the multivariate
density estimation problem where the goal is to fit a density to
observations $x_1, \dots, x_n$ in $(0, \infty)^d$, PW 
studied the MLE over $\smud$ : 
\begin{equation}\label{smumle_def}
\smumle := \argmax_{p \in \smud} \frac{1}{n}
  \sum_{i=1}^n \log p(x_i). 
\end{equation}
PW  proved several important properties
of $\smumle$ including 
existence, almost sure uniqueness, and characterizations. Under the 
standard modeling 
assumption that the data points $x_1, \dots, x_n$ are realizations of
random variables 
\begin{equation*}
  X_1, \dots, X_n \overset{\text{i.i.d}}{\sim} p_0 \qt{with $p_0 \in
    \smud$}, 
\end{equation*}
PW also studied the performance of $\smumle$
as an estimator for $p_0$. Among other results, they proved that
$\smumle$ is a strongly
consistent estimator of $p_0$ in both the total variation and
Hellinger loss functions.

PW also made an interesting but unproved observation on the rate of
convergence of $\smumle$ to $p_0$ in the Hellinger distance. The main
motivation for the present paper is to rigorously prove this
conjecture which appeared as Conjecture 2 in \cite[Section
5]{pavlides2012nonparametric}, and states the following: suppose $p_0
\in \smud$ is bounded from above by a constant and is concentrated on
$[0,   M]^d$ for some constant $M$, then  
\begin{equation}\label{pwconj}
  h^2(\smumle, p_0)  = O_p(n^{-2/3} (\log n)^{\gamma_d})  
\end{equation}
for some $\gamma_d$ depending on the dimension $d$ alone. 
%Here $h^2(\cdot,
%\cdot)$ stands for the squared Hellinger distance defined as
%\begin{equation}
%  \label{helldef}
%h^2(p, q) := \int (\sqrt{p} - \sqrt{q})^2
%\end{equation}
%for two densities $p$ and $q$. 
The same conjecture \eqref{pwconj} was
also stated in \cite[Section 5.3]{gao2013global}.   

Assertion \eqref{pwconj} is interesting mainly because the rate
$n^{-2/3} (\log n)^{\gamma_d}$  is quite close to the univariate rate of
$n^{-2/3}$ achieved by the Grenander estimator. Indeed, it is only
inferior by the logarithmic multiplier $(\log n)^{\gamma_d}$. The
curse of dimensionality which plagues most multidimensional estimation
procedures is therefore much milder for the multivariate extension
$\smumle$ of the Grenander estimator. Alternative multivariate
extensions of the Grenander estimator such as the MLE over ``block
decreasing'' densities over $(0,
\infty)^d$ admit convergence rates that are adversely affected by the curse of
dimensionality. Indeed, the minimax rate over ``block decreasing'' densities
was shown in \cite{biau2003risk} to be $n^{-2/(d+2)}$ in the squared total
variation distance and this rate is clearly much slower than the right
hand side of \eqref{pwconj} for $d \geq 2$.

Insight into the fast convergence rate in \eqref{pwconj} can be
obtained by noting the fact that the number of 
constraints imposed by the class $\smud$ on its member densities
increases significantly with the dimension $d$. More precisely, it can
be shown (using, for example, \cite[Theorem
  2.3]{pavlides2012nonparametric}) that, in order to belong to the
  class $\smud$, a smooth density $p$ on $(0, \infty)^d$ needs to
  satisfy the constraints:  
\begin{equation}\label{smupd}
(-1)^{|S|}  \frac{\partial^{|S|}p}{\prod_{i\in S} \partial x_i} \geq 0
\qt{for every $\emptyset \neq S \subseteq \{0, 1\}^d$}, 
\end{equation}
where $|S|$ denotes the cardinality of the subset $S$. Thus partial derivatives of up to order $d$ are
constrained by the class $\smud$ and, moreover, the number of
constraints is increasing exponentially in $d$. This is intuitively
the reason why the convergence rates for $\smumle$ do not suffer from
the usual curse of dimensionality. For comparison, note that the class of 
block-decreasing densities  (\cite{robertson1967estimating,
  sager1982nonparametric,   polonik1998silhouette,   biau2003risk})
imposes only the significantly weaker conditions
\begin{equation}\label{bdpd} 
\frac{\partial p}{\partial x_i}   \geq 0 \qt{for every $i = 1, \dots,
  d$}.  
\end{equation}
The constraint in \eqref{smupd} is similar to the notion of Entire
Monotonicity \cite{aistleitner2015functions, hobson1921theory,
  leonov1996total, young1924discontinuties} which has been used as a
shape constraint for nonparametric regression in
\cite{fang2021multivariate}.  More generally, $L_p$ norm constraints
on mixed derivatives similar to those appearing in \eqref{smupd}  have
been used for nonparametric regression by many authors (see e.g.,
\cite{donoho2000high, lin2000tensor, fang2021multivariate,
  ki2021mars, benkeser2016highly}) and these procedures often achieve
rates  similar to \eqref{pwconj} avoiding the usual curse of
dimensionality. On
the other hand, nonparametric regression with monotonicity constraints
similar to \eqref{bdpd} has been studied in \cite{han2019isotonic}.  

\color{black}
We now describe our results. Our main result is Corollary \ref{conjlowbou} which  proves \eqref{pwconj} with
$\gamma_d = 4d-2$ for $d \geq 2$, under the following
assumptions:
\begin{enumerate}
\item \textbf{Compact Support (CS)}: $p_0 \in \smud$ is concentrated on
  $[0, M]^d$ for a positive constant $M$, 
\item \textbf{Upper Bound (UB)}: $p_0$ is bounded from above on $[0, M]^d$
  by a positive constant $B$, 
\item \textbf{Lower Bound (LB)}: $p_0$ is bounded from below on $[0, M]^d$
  by a positive constant $b$. 
\end{enumerate}
The first two assumptions were also made by PW while stating their
conjecture \eqref{pwconj}. The third assumption is an additional one
that is needed for our proof of \eqref{pwconj}. Although we are unable to remove the LB assumption completely, we have been able to prove results which weaken it to some extent by relaxing it to hold on subrectangles of the full domain $[0, M]^d$ and also by replacing it with conditions on the
$L_{\qh}$ norm of $p_0^{-1}$ for a fixed $\qh \in (1, \infty)$ (see Theorem \ref{jrec}, Corollary \ref{conjecture}, and Corollary \ref{finitelq}). 
 
\color{black}

Our proofs proceed via a new result,
Theorem \ref{hellexp}, which gives Hellinger distance bounds for 
the MLE over an arbitrary convex class of densities $\Ps$. It reduces
the problem of obtaining Hellinger rates for the MLE to that of
obtaining upper bounds for the function:    
\begin{equation} \label{expsup}
t \mapsto  \E \sup_{p \in \mathcal{P} : h(p, p_0) \leq t} \int
\frac{4p_0}{p_0 + p} d(P_0 - P_n),  
\end{equation}
where $P_0$ is the probability distribution with density $p_0$ and
$P_n$ is the empirical distribution of the samples $X_1, \dots,
X_n$. Theorem \ref{hellexp} appears to be new and can be seen as a
maximum likelihood analogue of the result of Chatterjee
\cite{chatterjee2014new} 
for least squares estimators under convex constraints. While our focus
is on the case $\Ps = \smud$, Theorem \ref{hellexp} is applicable for
any convex class of densities $\Ps$. In order to obtain upper bounds
for \eqref{expsup} when $\Ps = \smud$, we use available bracketing
entropy bounds for distribution functions of nonnegative measures from
Gao \cite{gao2013book}. The connection between densities in $\smud$
and distribution functions of nonnegative measures is explained in
Section \ref{sec:bracken} (see \eqref{altrep}). 

We also provide a minimax lower bound (Theorem \ref{lowerbound}) which
proves that the logarithmic 
factor in \eqref{pwconj} cannot be removed
completely. Specifically, we prove that the minimax risk in squared
Hellinger distance over the class of densities in $\smud$ that are
bounded (from above by $B$ and below by $b$) and are supported on $[0,
M]^d$ is at least by a constant multiple of $n^{-2/3} (\log
n)^{(d-1)/3}$ (as long as $B$ and $M$ are large enough
constants and $b$ is a small enough constant). This obviously implies
that $\gamma_d$ in \eqref{pwconj} 
has to be at least $(d-1)/3$ (on the other hand, the upper bound on
$\gamma_d$ from our Theorem \ref{conjlowbou} is
$4d-2$). 

In Theorem \ref{adaptation}, we also prove that the rate of
convergence of $\smumle$ to $p_0 \in 
\smud$ can be much faster than \eqref{pwconj} when $p_0$ is piecewise 
constant over a finite set of rectangles in
$(0,\infty)^d$. Specifically, if the support of $p_0$ can be 
decomposed into $m$ rectangles that are nearly disjoint (in the sense
that their pairwise intersections have zero volume) such that $p_0$ is
constant on each rectangle, then
\begin{equation}\label{adapresult}
    h^2(\smumle, p_0)  = O_p\left(\frac{m}{n} (\log n)^{\tilde \gamma_d}
    \right),   
  \end{equation}
where $\tilde \gamma_d = 8(2d-1)/3$ 
which implies that the   rate of
convergence of $\smumle$ to $p_0$ is faster than the worst case upper
bound given by \eqref{pwconj} when $m$ is of smaller order than
$n^{1/3}$. In the univariate case (i.e., for the Grenander estimator),
such results can be found in \cite[Page 113]{VandegeerBook} and
\cite{Birge89}). 

We also discuss algorithms for computing $\smumle$. In Section \ref{sec:comp}, we discuss an exact algorithm (see Algorithm \ref{alg:smu_mle}) for computing $\smumle$, and also an approximate algorithm (see Algorithm \ref{alg:smu_mle_approx}) which is more computationally efficient. We illustrate the performance of the estimator on one simulated dataset and one real dataset involving bivariate $p$-values. 

The rest of the paper is organized as follows. Our general
result connecting the Hellinger accuracy of an MLE over a convex class of
densities to the expected supremum in \eqref{expsup} is stated in
Section \ref{sec:general}. This
result is crucially used with $\Ps = \smud$ to prove our Hellinger accuracy
results for $\smumle$. In Section \ref{sec:bracken}, we state
bracketing entropy results for subclasses of $\smud$ that are
necessary for proving our Hellinger accuracy results for
$\smumle$. Our main results are given in Section \ref{sec:smuresults}. Section \ref{sumdisc} has additional
discussion of issues relevant to our main results. Section \ref{sec:comp} discusses computational details. 
The proofs of the main
results are in Section \ref{proofs}, while Section \ref{additechres}
contains additional technical results and proofs.  

\section{Hellinger Accuracy of MLEs over convex classes of
  densities}\label{sec:general}
This section describes a general result for the Hellinger accuracy of
the MLE over a convex class of densities. Let $\Ps$ be a convex class
of densities on some common domain. Given $X_1, \dots, X_n$  generated
according to a true density $p_0 \in \Ps$, consider any MLE over $\Ps$
defined as  
\begin{equation*}
  \hat{p}_n \in \argmax_{p \in \Ps} \frac{1}{n} \sum_{i=1}^n \log
  p(X_i). 
\end{equation*}
We assume that $\hat{p}_n$ exists. The following result gives upper
bounds for the squared Hellinger distance $h^2(\hat{p}_n,
p_0)$. It will be used with $\Ps = \smud$ to prove
our Hellinger rate results for $\smumle$. 

\begin{theorem}\label{hellexp}
 Consider the setting described above. For $t \geq 0$, let
  \begin{equation}\label{gdef}
    G(t) := \sup_{p \in \Ps : h(p_0, p) \leq t} \int
    \frac{4p_0}{p_0 + p} d(P_0 - P_n)
  \end{equation}
  where $P_0$ is the probability measure corresponding to the true
  density $p_0$ and $P_n$ is the empirical distribution of $X_1,
  \dots, X_n$. All expectations below are with respect to
  $P_0$. Suppose there exist two real numbers $t_0 > 0$ and $0 < 
  \eta \leq 1$, and a function $\bar{G}: [0, \infty)
  \rightarrow [0, \infty)$ such that  
  \begin{enumerate}
  \item $\E G(t) \leq \bar{G}(t)$ for every $t \geq t_0$, 
  \item  $\bar{G}(t_0) \leq t_0^2$, and
  \item  $t \mapsto \frac{\bar{G}(t)}{t^{2 - \eta}}$ is non-increasing
    on $[t_0, \infty)$.  
  \end{enumerate}
  Then 
  \begin{equation}\label{hellexp.eq}
    \P \left\{h(\hat{p}_n, p_0) \geq t_0 + x \right\} \leq \exp
    \left(\frac{-n \eta^2 x^2}{\helconst} \right) \qt{for every
      $x > 0$}
  \end{equation}
  and
  \begin{align}\label{hellexp.exp}
    \E h^2(\hat{p}_n, p_0) \leq 2 t_0^2 + \frac{\helconst}{n \eta^2}. 
  \end{align}
\end{theorem}

In order to apply Theorem \ref{hellexp}, we need to bound the
expectation of \eqref{gdef} from above. For this, our main tool will
be the following standard 
bound from \cite[Theorem 19.36]{vaart98book} on the expected supremum
of an empirical process. \color{black} This result uses the definition of bracketing
numbers (see e.g. \cite[Definition 2.1.6]{vaart98book}). \color{black}

%\color{black}
%\begin{definition}[Bracketing numbers]\label{bracketdef}
%Let $\F$  be a class of functions on some space $\samp$ and
%let $\rho$ be a pseudometric on $\F$ (in the result below, $\rho$ will
%be the $L_2$ metric with respect to a probability  measure $P_0$ on
%$\samp$). The $\epsilon$-bracketing number of $\F$  with respect to
%the pseudometric $\rho$ will be denoted by $N_{[]}(\epsilon, \F,
%\rho)$ and is defined as the smallest positive integer $M$ for which
%there exist $M$ pairs of functions $(f_{L, 1}, f_{U, 1}), \dots,
%(f_{L, M}, f_{U, M})$ such that $\rho(f_{L, j}, f_{U, j}) \leq \epsilon$ for each
%$j = 1, \dots, M$  and such that for every $f \in \F$, there exists $j
%:= j(f) \in \{1, \dots, M\}$ with $f_{L, j}(x) \leq f(x) \leq f_{U,
%  j}(x)$  for every $x \in \samp$. We shall refer to the logarithm of
%$N_{[]}(\epsilon, \F, \rho)$ as the $\epsilon$-bracketing entropy of
%$\F$ with respect to the pseudometric $\rho$. 
%\end{definition}
%\color{black}

%\color{black}
%DP pollard paper citation? Ossiander bound
%\color{black}
\begin{theorem}[\cite{Ossiander87bracket} and Theorem 19.36 of
  \cite{vaart98book}]\label{expcontrol}
Let $X_1, \dots, X_n$ be i.i.d taking values in a space $\samp$ with
distribution $P_0$.  Suppose $\F$ is a class of functions on $\samp$
that are uniformly bounded by $M$ and such that $\sup_{f \in \F} \E
f^2(X_1) \leq \delta^2$ for some fixed $\delta > 0$. Let
\begin{equation}\label{braentint}
  J(\delta) := \int_0^{\delta} \sqrt{\log N_{[]}(\epsilon, \F,
    L_2(P_0))} d\epsilon. 
\end{equation}
Then
\begin{equation*}
  \E \sup_{f \in \F} |P_n f - P_0 f| \leq \frac{C}{\sqrt{n}} J(\delta)
  \left(1 + \frac{M J(\delta)}{\delta^2 \sqrt{n}} \right)
\end{equation*}
for a universal constant $C$. 
\end{theorem}

Theorem \ref{hellexp} appears to be new although it is quite
similar to existing results such as \cite[Theorem
7.6]{VandegeerBook}. The main difference is that Theorem \ref{hellexp}
characterizes the key quantity $t_0$ (which controls $h(\hat{p}_n,
p_0)$) via the condition:  
\begin{equation}
  \label{rate.char}
 \E \sup_{p \in \Ps : h(p_0, p) \leq t} \int
    \frac{4p_0}{p_0 + p} d(P_0 - P_n) \leq t^2 \qt{for all $t \geq
      t_0$}. 
  \end{equation}
On the other hand, van de Geer \cite[Theorem 7.6]{VandegeerBook} characterizes
the rate $t_0$ by the inequality obtained by replacing the left hand
side of \eqref{rate.char} by the bracketing entropy integral (as in \eqref{braentint}) of the function class 
\begin{equation}\label{brackentclass}
  \left\{\frac{4p_0}{p+p_0} : p \in \Ps, h(p, p_0) \leq t \right\} 
\end{equation}
under the $L_2(P_0)$ metric. Even though bracketing entropy integrals
are important for bounding
expected suprema of empirical processes, the expected supremum in
\eqref{rate.char} is more directly 
connected to the Hellinger accuracy of $\hat{p}_n$. Working with the expected
supremum as in \eqref{rate.char} is more convenient compared to
working with the bracketing entropy
integral because the bracketing entropy
of the whole class \eqref{brackentclass} is usually not available so
one would need to decompose it into smaller subclasses whose entropy
can be bounded; it is easier to carry out such a decomposition in
terms of the expected supremum. In some cases, one can use simpler
bounds on the expected supremum without recourse to bracketing entropy
integrals (see, for example, the bound \eqref{trivialb} below); it is
not clear how such bounds can be used in conjunction with  \cite[Theorem
7.6]{VandegeerBook}. Also, for obtaining accuracy results
for the least squares estimator in nonparametric regression with
convex constraints, the current popular approach is based on bounding
expected suprema similar to \eqref{rate.char} via the results of
Chatterjee \cite{chatterjee2014new}. Our Theorem \ref{hellexp} can be seen as an
analogue of the upper bound part of \cite[Theorem
1.1]{chatterjee2014new} for density estimation. Note however that \cite[Theorem
1.1]{chatterjee2014new} also provides a lower bound on the accuracy of
convex least squares estimators in terms of expected suprema while our
result, Theorem \ref{hellexp}, only gives upper bounds.

\section{Bracketing Entropy Bounds for subclasses of
  $\smud$}\label{sec:bracken}
Subsection \ref{init_steps} gives an outline of how Theorem \ref{hellexp} is used to prove Hellinger accuracy bounds for $\smumle$. The key here is to prove bracketing entropy numbers for the following function class: 
\begin{align}\label{Fdest}
  \F = \left\{\frac{p_0 - p}{p_0 + p} \one(R_i) : p \in \smud \text{ and
  }  h(p_0, p) \leq t\right\}. 
\end{align}
The following lemma  allows working with the SMU densities $p$
directly instead of the transformed functions $\frac{p_0 - p}{p_0 +
  p}$. Specifically, it bounds the bracketing numbers of $\F$
via those of  
\begin{align}\label{Pdest}
  \left\{p I_{R} : p \in \smud \text{ and } h(p_0, p) \leq t
  \right\}. 
\end{align}

\begin{lemma}\label{trivial}
  Fix $\mathfrak{q} \in (1, \infty]$ and let $\mathfrak{p}$ be such that
  $1/\mathfrak{p} + 1/\mathfrak{q} = 1$. Then for every $\epsilon >
  0$, we have 
  \begin{equation}\label{brack.lqbound}
    \begin{split}
 &   N_{[]}\left(\epsilon, \left\{\frac{p_0 - p}{p_0 + p} \one(R) : p
          \in \smud, h(p_0, p) \leq t \right\}, L_2(P_0) \right) \\ &\leq
    N_{[]} \left(\frac{\epsilon}{2
    \sqrt{\|p_0^{-1}\|_{L_{\mathfrak{q}}(R)}}}, \{p 
                                \one(R) : p \in \smud, h(p_0, p) \leq
                                                                      t
                                                                      \}
                                                                      ,
                                                                      L_{2
                                                                      \mathfrak{p}}(R)
                                                                      \right)  
   \end{split}                                                                      
 \end{equation}
 where $L_{2 \mathfrak{p}}(R)$ is the usual $L_{2 \mathfrak{p}}$
 metric with respect to the Lebesgue measure on $R$, and
 \[ \|p_0^{-1}\|_{L_{\mathfrak{q}}(R)} :=
  \begin{cases}
    \left(\int_{R} \frac{1}{p_0^{\mathfrak{q}}} \right)^{1/\mathfrak{q}}       & \quad \text{for } \mathfrak{q} \in (1, \infty)\\
   \left(\min_{x \in R} p_0(x) \right)^{-1}  & \quad \text{for } \mathfrak{q} = \infty.
  \end{cases}
\]
%In particular for $\mathfrak{q} = \infty$, we obtain
%\begin{equation}\label{brack.l2minbound}
%  \begin{split}
%   &   N_{[]}\left(\epsilon, \left\{\frac{p_0 - p}{p_0 + p} \one(R) : p
%          \in \smud, h(p_0, p) \leq t \right\}, L_2(P_0) \right) \\ &\leq N_{[]}
%    \left(\frac{\epsilon}{2} \sqrt{\min_{x \in R} p_0(x)}, \left\{p
%        \one(R) : p \in \smud, h(p_0, p) \leq t \right\},
%    L_2(R) \right)
%    \end{split}
%  \end{equation}
\end{lemma}

The above lemma bounds the $L_2(P_0)$ bracketing
entropy number of $$\left\{\frac{p_0 - p}{p_0 + p} \one(R) : p
          \in \mathfrak{P} \right\}$$ in terms of the
bracketing entropy number of $\left\{p
  \one(R) : p \in \mathfrak{P} \right\}$ for the $L_{2 \ph}$
metric. Because of the presence of $L_{\qh}(R)$ norm of $p_0^{-1}$,
these bounds are useful only when $p_0$ is not too small at any point
in $R$. This term is ultimately the reason for the lower bound
restrictions in our Hellinger rate results for $\smumle$ .

We shall apply Lemma \ref{trivial} with $\mathfrak{P} = \left\{p \in
  \smud : h(p_0, p) \leq t \right\}$ for $t > 0$ and this will lead to
upper bounds on the $L_2(P_0)$ bracketing entropy of \eqref{Fdest} in
terms of the bracketing entropy of \eqref{Pdest}. The next step is
therefore to bound the bracketing entropy numbers of SMU densities
over rectangles $R$ under Hellinger constraints of the form $h(p, p_0)
\leq t$. Dealing with such Hellinger constraints directly is a bit
tricky so we convert them into upper and lower bounds for $p$ on the
set $R$ via the following result. 

\begin{lemma}\label{bds}
Suppose $p$ and $p_0$ are coordinatewise non-increasing functions on
  $(0, \infty)^d$ such that
  \begin{align*}
    \int_{{  R }}  \left(\sqrt{p} - \sqrt{p_0} \right)^2 \leq
    t^2 
  \end{align*}
  for some $t > 0$ where $R \subseteq (0, \infty)^d$ is a $d$ dimensional rectangle. Then for every $x \in R$, we have 
  \begin{equation*}
    p(x) \leq U_{p_0}(x, t) := \inf_{  \alpha \leq x, \alpha \in R}
    \left(\sqrt{p_0(\alpha)} + \frac{t}{\sqrt{(x_1 - \alpha_1) \dots (x_d -
          \alpha_d)}} \right)^2
  \end{equation*}
  and
  \begin{equation*}
    p(x) \geq L_{p_0}(x, t) := \sup_{\beta \geq x, \beta \in R} \left(\sqrt{p_0(\beta)} -
        \frac{t}{\sqrt{(\beta_1 - x_1) \dots (\beta_d - x_d)}}
    \right)_+^2
  \end{equation*}
  where, in the second inequality above, $u^2_+ := \left[\max(u,0)
  \right]^2$.  The inequalities $\leq$ and  $\geq$ appearing in the
  infimum and supremum above respectively should be interpreted in
  the pointwise   sense.
\end{lemma}
The above result is stated for coordinatewise non-increasing functions
on $(0, \infty)^d$ and it automatically applies to densities in
$\smud$ as they are always coordinatewise non-increasing (this follows
directly from \eqref{pwsmudef}). 

The main task now is to control the bracketing entropy numbers of
bounded densities in $\smud$ over rectangles $R$ (with respect to the
$L_2$ and Hellinger metrics). More precisely, for a fixed compact rectangle
\begin{align}\label{Rspec}
  R := [a_1, b_1] \times \dots \times [a_d, b_d]
\end{align}
for $0 \leq a_i < b_i < \infty, 1 \leq i \leq d$ and $0 \leq \alpha <
\beta < \infty$, let 
\begin{equation}\label{falbe}
  \F(R, \alpha, \beta) := \left\{g : R \rightarrow [\alpha, \beta]
    \text{ such that } g = p \middle|_R \text{ for some } p \in
    \smud  \right\}
\end{equation}
where $g = p|_R$ means that $g(x) = p(x)$ for $x \in R$. Note that 
functions in $\F(R, \alpha, \beta)$ are bounded on $R$ by $\alpha$ (from 
below) and $\beta$ (from above). The following result gives upper
bounds on the bracketing entropy of $\F(R, \alpha, \beta)$ under the
$L_r(R)$ metric (here $L_r(R)$ stands for $L_r$ metric with respect to 
the Lebesgue measure on $R$) for fixed $r \in [1, \infty)$. 

\begin{lemma}\label{brackd}
For every $\epsilon > 0$ and $r \in [1, \infty)$, we have 
\begin{equation}\label{L2ent.eq}
  \begin{split}
  &    \log N_{[]}(\epsilon, \F(R, \alpha, \beta),  L_r(R)
      \color{black}) \\ &\leq \frac{C_{d, r} (\beta - \alpha)
       |R|^{1/r}}{\epsilon} 
      \left(\log \frac{(\beta - \alpha) |R|^{1/r}}{\epsilon}
      \right)^{2(d-1)} \one \left(\epsilon \leq (\beta - \alpha)
    |R|^{1/r} \right)
  \end{split}    
    \end{equation}
    where $|R| := (b_1 - a_1) \dots (b_d - a_d)$ is the volume of
    $R$, and $C_{d, r}$ is a constant depending on $d$ and $r$. 
\end{lemma}

Lemma \ref{brackd} (proved in  Section \ref{additechres}\color{black})
is a consequence of the following result due to
Gao \cite{gao2013book} on  bracketing entropy numbers of distribution
functions of subprobability measures on $[0, 1]^d$ with respect to the
$L_2[0, 1]^d$ metric (a subprobability measure $G$ on $[0, 1]^d$ is a
nonnegative measure satisfying $G[0, 1]^d \leq 1$).  

\begin{theorem}[Theorem 1.1 of \cite{gao2013book}]\label{gaoL2}
Let $\A_d$ denote the class of all distribution functions of
subprobability measures on $[0, 1]^d$ i.e., $\A_d$ contains functions
of the form 
\begin{equation*}
  (x_1, \dots, x_d) \mapsto G \left([0, x_1] \times \dots \times [0, x_d] \right)
\end{equation*}
as $G$ varies over the class of all nonnegatives measures on $[0,
1]^d$ with $G[0, 1]^d \leq 1$. Then for every $\epsilon > 0$ and $r
\in [1, \infty)$, we have
  \begin{equation}\label{gaohell.eq}
    \log N_{[]} \left(\epsilon, \A_d, { L_r([0,1]^d)} \right) \leq
    \frac{C_{d, r}}{\epsilon} \left(\log \frac{1}{\epsilon}
    \right)^{2(d-1)} \one \left(\epsilon \leq 1 \right)
  \end{equation}
  for a constant $C_{d, r}$ depending on $d$ and $r$. 
\end{theorem}

The reason why Theorem \ref{gaoL2} implies Lemma \ref{brackd} is that
functions in $\smud$ are quite closely connected to distribution
functions of measures. To see this, note that, by definition, every
density $p \in \smud$ is of the form 
\begin{equation*}
  p(u_1, \dots, u_d) = \int \frac{\one\{u_1 \leq \theta_1, \dots, u_d
    \leq \theta_d\}}{\theta_1 \dots \theta_d} dG(\theta_1, \dots,
  \theta_d)
\end{equation*}
for some probability measure $G$ on $(0, \infty)^d$. The above can be
alternatively written as
\begin{equation}\label{altrep}
  p(u_1, \dots, u_d) = \tilde{G} \left([u_1, \infty) \times \dots
    \times [u_d, \infty) \right). 
\end{equation}
where $\tilde{G}$ is the measure on $(0, \infty)^d$ defined by 
\begin{equation*}
  d \tilde{G}(\theta_1, \dots, \theta_d) := \frac{dG(\theta_1, \dots,
    \theta_d)}{\theta_1 \dots \theta_d}. 
\end{equation*}
The right hand side of \eqref{altrep} has obvious connections to the
distribution function of a measure. 

Lemma \ref{brackd} can be used, in conjunction with  inequality
\eqref{brack.lqbound} for $q=\infty$ as well as Lemma \ref{bds}, to prove
bracketing entropy bounds for \eqref{Fdest}. These entropy bounds are then used with Theorem \ref{hellexp} (following the approach
outlined in Subsection \ref{init_steps}) to yield bounds
on the Hellinger accuracy for $\smumle$.

\section{Hellinger accuracy of $\smumle$} \label{sec:smuresults}

We now present our results on the Hellinger‐distance accuracy of $\smumle$ relative to the true density $p_0 \in \smud$. The centerpiece is Theorem \ref{jrec}, from which several natural consequences follow—most notably Corollary \ref{conjlowbou}, which confirms the PW conjecture under the three assumptions (CS, UB, and LB) introduced in the introduction.

We use here the following notation. For a closed and
bounded rectangle $R \subseteq [0, \infty)^d$ and $\qh \in (1, \infty)$
\begin{align*}
  W(R, p_0, \qh) := \max \left(1, |R|^{1/(4\ph)}
    \|p_0^{-1}\|_{L_{\qh}(R)}^{1/4} \sqrt{\max_{x \in R} p_0(x)}
  \right)
\end{align*}
where $\ph$ is such that $1/\ph + 1/\qh = 1$. It is helpful to note
that if $R$ is of the form $[a_1, b_1] \times \dots \times [a_d,
b_d]$, then $\max_{x \in R} p_0(x) = p_0(a_1, \dots, a_d)$ because
$p_0 \in \smud$. 

\begin{theorem}\label{jrec}
  Suppose there exists a set of rectangles $R_j, j = 1, \dots, J$ with
  disjoint interiors such that
  \begin{equation}
    \label{jrec.con}
    \max_{1 \leq j \leq J} W(R_j, p_0, \qh) < \infty ~~ \text{ and }
    ~~ P_0 \left(\cup_{j=1}^J R_j \right) \geq 1 -  J^2\frac{(\log
      n)^{4d-2}}{n^{2/3}}.
  \end{equation}
  Then there exist positive constants $C_{d, \qh}$ and $C_d$ such that
  \begin{equation}\label{jrec.eq}
    \E h^2 \left(p_0, \smumle \right)  \leq C_{d, \qh} J^2 \frac{(\log
  n)^{4d-2}}{n^{2/3}} W^2 \max \left((\log W)^{2d-2}, 1 \right) + 
 C_d \frac{J}{n} W^4 
\end{equation}
where $W = \max_{1 \leq j \leq J} W(R_j, p_0, \qh)$. 
\end{theorem}

It is clear from \eqref{jrec.eq} that if the support of $p_0$ can be
partitioned into a logarithmic number of subrectangles $R_j$ (along
with a residual subset of small 
$p_0$ probability) for which $W(R_j, p_0, \qh) < \infty$, then
$h^2(p_0, \smumle)$ converges at the rate  $n^{-2/3} (\log
n)^{\gamma_d}$.

When $J=1$, we have the following Corollary.
\begin{corollary}\label{conjecture}
Fix $d \geq 2$ and $n \geq 2$. Suppose $p_0 \in \smud$ is concentrated
on a rectangle $R = [0,M]^d \subseteq [0, \infty)^d$ for some $M < \infty$. Assume that
\begin{equation}
  \label{wass}
 W =  W(R, p_0, \qh)  < \infty, 
\end{equation}
for a fixed $\qh \in (1, \infty]$. Then there exist positive constants $C_{d, \qh}$ and $C_d$ such that 
\begin{align}\label{conjecture.eq}
  \E h^2 \left(p_0, \smumle \right) \leq C_{d, \qh} \frac{(\log
  n)^{4d-2}}{n^{2/3}} W^2 \max \left((\log W)^{2d-2}, 1 \right) + 
  C_d\frac{W^4}{n}. 
\end{align}
\end{corollary}

Observe that $W(R, p_0, \qh) < \infty$ is equivalent to the three
conditions $|R| < \infty$ (CS assumption),  $\max_{x \in R} p_0(x) <
\infty$ (UB assumption) and $\|p_0^{-1}\|_{L_{\qh}(R)}
< \infty$ (assumption \eqref{lqass}). Therefore Corollary
\ref{conjecture} is equivalent to the following Corollary \ref{finitelq}.
Note that Corollary \ref{conjecture} is a more explicit form of Corollary \ref{finitelq} where
the dependence of the constant $C_{d, B, M, \qh, T}$ on the
$p_0$-dependent quantities $B, M, \qh, T$ is made more explicit. 
If \eqref{lqass}  is violated, then $W(R, p_0, \qh) = +\infty$. For such
$p_0$, it might sometimes be possible to obtain smaller subrectangles
$R_1, \dots, R_J$ inside the full domain $R$ for which $W(R_j, p_0,
\qh) < \infty$. If the number of such rectangles $J$ is at most
logarithmic in $n$, then one still gets the $n^{-2/3} (\log
n)^{\gamma_d}$ rate as proved in Theorem \ref{jrec}.

\begin{corollary}\label{finitelq}
  Fix $d \geq 2$ and $n \geq 2$. Suppose $p_0 \in \smud$ is
  concentrated on $[0, M]^d$ for some $M > 0$, and is bounded from
  above by $B$ on $[0, M]^d$. Suppose further that for some fixed $\qh \in (0,
      \infty)$,
  \begin{equation}\label{lqass}
    T := \|p_0^{-1}\|_{L_{\qh}([0, M]^d)} := \left(\int_{[0, M]^d} p_0^{-\qh}
    \right)^{1/\qh} < \infty. %\qt{for some fixed $\qh \in (0,\infty)$}. 
  \end{equation}
  Then there exists $C_{d, B, M, \qh, T} \in (0, \infty)$ such that
  \begin{equation*}
    \E h^2 \left(p_0, \smumle \right) \leq C_{d. B, M, \qh, T}
    n^{-2/3} (\log n)^{4d-2}. 
  \end{equation*}
\end{corollary}

On the compact domain $[0, M]^d$, it is clear that the LB assumption
implies \eqref{lqass} for every $\qh$. This leads to the following corollary. Note that \eqref{lqass} is a weaker assumption compared to LB because there exist
many densities $p_0 \in \smud$ which satisfy \eqref{lqass} for a fixed
finite $\qh \in (1, \infty)$ but which violate the LB assumption. 

\begin{corollary}\label{conjlowbou}
  Fix $d\geq 2$ and $n \geq 2$. Suppose $p_0 \in \smud$  is
  concentrated on $[0, M]^d$ for some $M > 0$, and is bounded from
  above by $B$ and below by $b > 0$ on $[0, M]^d$. Then there exists
  $C_{d, B, M, b} \in (0, \infty)$ (depending on $d, B, M, b$) such that
  \begin{equation}\label{conjlowbou.eq}
\mathbb{E} h^2(p_0,  \smumle) \leq C_{d, B, M, b} n^{-2/3} (\log
  n)^{4d-2}. 
\end{equation}
\end{corollary}

% In Corollary \ref{conjlowbou}, we used the three assumptions CS, UB and LB (these names for the assumptions were introduced in the introduction) to prove the PW conjecture \eqref{pwconj} with $\gamma_d = 4d -2$.  The CS and UB assumptions were also made by PW while stating their conjecture. We have included the LB assumption for technical reasons because we are unable to prove \eqref{pwconj} without it. We are able to weaken it to some degree (as will be clear from the following results) but are unable to remove it completely. Corollary \ref{finitelq} replaces the LB assumption by a weaker finite $L_{\qh}$ norm assumption on $p_0^{-1}$ for a fixed $\qh \in (1, \infty)$.

\color{black}

Theorem \ref{jrec} can be used to remove the LB assumption in some cases. The following result shows that, when $p_0 \in
\smud$ is \color{black} close to 
a product density, $h^2(p_0, \smumle)$ has the $n^{-2/3} (\log n)^{4d-2}$ rate without
any lower bound assumption on $p_0$. Only assumptions needed are compact
support and boundedness from above for each marginal of $p_0$ used in the lower and the upper bound.  Note
that even though $p_0$ is assumed to  be close to a product measure in
Proposition \ref{prodsmugen}, the estimator $\smumle$ is the MLE over all
densities in $\smud$. 
Proposition \ref{prodsmugen} ia proved by
explicitly constructing a partition of $[0, M]^d$ with $J \leq C_{d,
  A, M} ( \log \log n)^{d}$ satisfying the 
conditions of Theorem \ref{jrec}.

\color{black}

\begin{proposition}\label{prodsmugen}
  Suppose $p_0 \in \smud$ is such that there exist
  univariate right-continuous nonincreasing densities $p_{01}, \dots,
  p_{0d}$  such that 
  \begin{equation}\label{prodsmugencon}
a~p_{01}(x_1) \dots p_{0d}(x_d) \leq  p_0(x_1, \dots, x_d) \leq A~
p_{01}(x_1) \dots p_{0d}(x_d) 
\end{equation}
for two positive $a$ and $A$. Further assume that each
$p_{0j}$ is concentrated on $[0, 
  M]$ with $\sup_{x_j \in [0,
    M]} p_{0j}(x_j) \leq B$. Then there 
  exists $C_{d, B, M, a, A} \in (0, \infty)$ such that
  \begin{equation*}
    \E h^2 \left(p_0, \smumle \right) \leq C_{d, B, M, a, A} (
    \log \log n)^{2d} n^{-2/3}
    (\log n)^{4d - 2}.  
  \end{equation*}
  for all $n \geq 3$. 
\end{proposition}

\begin{remark}
If $a=A=1$, $p_0 \in \smud$ is a product probability density of the form
  \begin{equation*}
    p_0(x_1, \dots, x_d) = p_{01}(x_1) \dots p_{0d}(x_d)
  \end{equation*}
  where each $p_{0j}$ is a nonincreasing right continuous univariate
  density on $[0, 
  M]$ with $\sup_{x_j \in [0, M]} p_{0j}(x_j) \leq B$. Thus there
  exists $C_{d, B, M} \in (0, \infty)$ such that
  \begin{equation*}
    \E h^2 \left(p_0, \smumle \right) \leq C_{d, B, M} n^{-2/3} (
    \log \log n )^{2d} (\log
    n)^{4d - 2}. 
  \end{equation*}
\end{remark}

In the next result, we prove a minimax lower bound which proves that
the rate  given by Corollary \ref{conjlowbou} cannot be significantly
improved. Specifically, we prove that the minimax risk in squared
Hellinger distance under the assumptions of Corollary \ref{conjlowbou}
is bounded from below by $n^{-2/3} (\log n)^{(d-1)/3}$. This shows
that the bound \eqref{conjlowbou.eq} is optimal up to a logarithmic
factor of  $(\log n)^{(11d-5)/3}$. 

\begin{theorem}[Minimax lower bound]\label{lowerbound}
Let $\mathcal{P}_{\mathrm{SMU}}([0,M]^d,b, B)$ be the class of scale
mixtures of uniform densities that are supported on $[0,M]^d$ and that
are bounded above by $B$ and bounded below by $b$. There exists a
positive constant $c_d$ such that 
\begin{equation}\label{lowerbound.eq}
\inf_{\tilde p_n} \sup_{p_0 \in \mathcal{P}_{\mathrm{SMU}}([0,M]^d,b, B)
} \E h^2(p_0, \tilde p_n) \geq c_d n^{-2/3} (\log n)^{(d-1)/3}. 
\end{equation}
whenever $B, 1/b, M$ are all larger than $c_d$.
\end{theorem}

In the next result, we prove that the rate of convergence of $\smumle$
can be  faster when $p_0$ is piecewise constant on a finite number of
bounded rectangles. This reveals adaptive risk properties of
$\smumle$. 

\begin{theorem}\label{adaptation}
Suppose $p_0 \in \smud$ is of the form
\[
p_0(x) = \sum_{j=1}^m p_j \one\{\bm{x} \in R_j\}
\]
where $R_j$ is a d-dimensional rectangle of the form $R_j = [a_{j1},
b_{j1}] \times \ldots \times [a_{jd},b_{jd}] \in \mathbb{R}^d$ for
$j=1, \ldots, m$. Also suppose that $|R_j \cap R_{j'}| =0$ for $j \neq
j'$. Then there exists $C_{d} \in (0,
\infty)$ depending only on $d$ such that  
\[
\mathbb{E} h^2(p_0, \smumle) \leq C_{d} \frac{m}{n} (\log n)^{(8/3)(2d-1)}.
\]
\end{theorem}
Clearly when $m$ is of constant order, the rate given by Theorem
\ref{adaptation} is much faster than the minimax lower bound $n^{-2/3}
(\log n)^{(d-1)/3}$. Observe that no lower
bound assumption on values $p_1, \dots, p_m$ of $p_0$ on the $m$
rectangles is needed for Theorem \ref{adaptation}.

We would like to emphasize that all our Hellinger rate results apply to the estimator $\smumle$ which is the MLE
over the entire class $\smud$. In other words, even though we make
some assumptions on $p_0$ (such as compact support and boundedness in
Corollary \ref{conjlowbou}, and rectangular piecewise constant in
Theorem \ref{adaptation}), 
the estimator analyzed is still the MLE over all the densities in
$\smud$. This makes the proofs 
of these results nontrivial. 

%We follow the strategy outlined near the end of Section \ref{sec:general} which require bounding expected suprema of an empirical process over sets of the form $\{p \in \smud: h(p_0, p) \leq t\}$. The boundedness assumptions on $p_0$ imply that the densities in this set are also bounded in parts of the domain that are away from the boundary. This allows the use of bracketing entropy bounds in these parts of the domain together with the simpler bound \eqref{trivialb} near the boundary. 

%\section{Improvements of Theorem \ref{conjlowbou}}\label{improv} 

\section{Summary and Discussion}\label{sumdisc}

\subsection{Summary of main results and the key proof idea}
In this paper, we proved Hellinger risk results for the nonparametric
maximum likelihood estimator $\smumle$ over the class of SMU densities
$\smud$. Our main result (Corollary \ref{conjlowbou}) proves the rate
$n^{-2/3} (\log n)^{\gamma_d}$ for $h^2(p_0, \smumle)$ provided the true
density $p_0 \in \smud$ satisfies the three assumptions: CS, UB and
LB. The LB assumption can be relaxed to an $L_{\qh}$ assumption on
$p_0^{-1}$ (see Corollaries \ref{conjecture} and \ref{finitelq}). We also
proved a more 
abstract result (Theorem \ref{jrec}) which requires boundedness of
$p_0$ over smaller subrectangles instead of the full domain. We
demonstrated in Proposition \ref{prodsmugen} how
this abstract result can be used in the absence of 
the lower bound restriction for densities $p_0$ which are not far
from product densities in the sense of \eqref{prodsmugencon}. We also
proved a minimax lower bound (Theorem \ref{lowerbound}) which matches
the rate in Corollary \ref{conjlowbou} up to logarithmic factors, and an
adaptation result (Theorem \ref{adaptation}) which proves near
parametric rates for piecewise constant densities $p_0$ in $\smud$.

Our bounds for $h^2(p_0, \smumle)$ are all based on upper bounds for (see e.g., Subsection \ref{init_steps}):
\begin{equation*}
  \E \sup_{p \in \smud: h(p_0, p) \leq t} \int \frac{p_0 - p}{p_0 + p}
  d(P_0 - P_n).  
\end{equation*}
Bounding the above expected supremum requires bracketing entropy
bounds on the functions $(p_0 - p)/(p_0 + p)$. In order to modify
available bracketing entropy bounds for distribution functions
\cite{gao2013book}, we convert distances between these transformed
functions $(p_0 - p)/(p_0 + p)$ to distances in terms of the original
densities $p$. The following inequality (from the proof of Lemma
\ref{trivial}) is our main tool here: 
\begin{equation}\label{discl2}
  \begin{split}
&\int_{R} \left( \frac{p_0-p_L}{p_0+p_L}-\frac{p_0-p_U}{p_0+p_U}
  \right)^2 p_0 \\ &= 4\int_{R} \frac{p_0^3
                  (p_U-p_L)^2}{(p_0+p_L)^2(p_0+p_U)^2} 
\leq 4 \left\{\int_R \left(p_U -
                     p_L \right)^{2\ph} \right\}^{1/\ph} \|p_0^{-1}\|_{L_{\qh}(R)}.
\end{split}                     
\end{equation}
This inequality involves $\|p_0^{-1}\|_{L_{\qh}(R)} <
\infty$, and this is the reason for the presence of this term in Theorem \ref{jrec} and 
Corollaries \ref{conjecture}, \ref{finitelq} and \ref{conjlowbou}.

\subsection{Hellinger bracketing numbers}

Let us discuss here a possible alternative approach to bounding the
left hand side of \eqref{discl2}. This involves the inequality:
\begin{equation}\label{dischell}
  \begin{split}
&  \int_{R} \left( \frac{p_0-p_L}{p_0+p_L}-\frac{p_0-p_U}{p_0+p_U}
  \right)^2 p_0 \\ &= 4\int_{R} \frac{p_0^3
                  (p_U-p_L)^2}{(p_0+p_L)^2(p_0+p_U)^2} 
    \leq 16 \int_R \left(\sqrt{p_U} - \sqrt{p_L} \right)^2.
  \end{split}    
\end{equation}
Unlike \eqref{discl2}, the right hand side of \eqref{dischell} does
not involve any norm on $p_0^{-1}$. Instead, it involves the Hellinger
distance between $p_L$ and $p_U$ on the set $R$. If we replace
\eqref{discl2} by \eqref{dischell} in the proof of 
Lemma \ref{trivial}, we would obtain the following bound instead of
\eqref{brack.lqbound}:
\begin{align}\label{brack.hellbound}
  \begin{split}
 &     N_{[]}\left(\epsilon, \left\{\frac{p_0 - p}{p_0 + p} \one(R) : p
          \in \smud, h(p_0, p) \leq t \right\}, L_2(P_0) \right) \\ &\leq N_{[]}
      \left(\frac{\epsilon}{4}, \left\{p
    \one(R) : p \in \smud, h(p_0, p) \leq t \right\}, h \right).
  \end{split}
  \end{align}
%where the right hand side is the bracketing number in the Hellinger
%distance. 
Unfortunately, we are unable to use the inequality
\eqref{brack.hellbound} because we do not quite know how to bound this
Hellinger bracketing number. The key challenge here is to prove
an analogue of Theorem \ref{gaoL2} for Hellinger bracketing. Hellinger
distance for distribution functions of subprobability measures on $[0,
1]^d$ is larger (up to a factor of $1/2$) than the $L_2$ distance
because:
\begin{align*}
h^2(F_1, F_2) &:= \int_{[0, 1]^d} \left(\sqrt{F_1(x)} - \sqrt{F_2(x)}
  \right)^2 dx \\ &= \int_{[0, 1]^d} \frac{(F_1(x) - F_2(x))^2}{\left(\sqrt{F_1(x)} + \sqrt{F_2(x)}
  \right)^2} dx \geq \frac{1}{4} \int_{[0, 1]^d} \left(F_1(x) - F_2(x)
  \right)^2 dx, 
\end{align*}
where we used the fact that $F_1$ and $F_2$ are nonnegative functions
that are upper bounded by  
$1$ on $[0, 1]^d$. Because of this, it is not clear if Theorem
\ref{gaoL2} will continue to hold if the $L_2$ metric is replaced by
the Hellinger metric. However if this stronger result can be proved,
then no condition on the size of $p_0^{-1}$ will be necessary, and
this will allow one to establish the PW conjecture without additional
assumptions on the size of $p_0^{-1}$.

\subsection{Minimax rates}
A related issue that we have not resolved in this paper concerns the
minimax rate. Corollary \ref{conjlowbou} and Theorem \ref{lowerbound} together
show that the minimax rate 
(in squared Hellinger distance) for the class $\mathcal{P}_{\mathrm{SMU}}([0,M]^d,b, B)$
(consisting of all densities in $\smud$ that are supported on $[0,
M]^d$ and are bounded from above by $B$ and below by $b$) is of the
order $n^{-2/3}$ with a multiplicative factor that lies between $(\log
n)^{(d-1)/3}$ and $(\log n)^{4d-2}$. It is natural to ask here for the
minimax rate without the lower bound constraint; in other words, what
is the minimax rate for $\mathcal{P}_{\mathrm{SMU}}([0,M]^d,b = 0, B)$.

It is obvious that the minimax rate for $\mathcal{P}_{\mathrm{SMU}}([0,M]^d,b = 0, B)$ will be larger than the rate for $\mathcal{P}_{\mathrm{SMU}}([0,M]^d,b, B)$ for $b > 0$. It is unclear however if the former minimax rate will be of a strictly larger order than the latter rate. If the former minimax rate is also  $n^{-2/3}$ with logarithmic factors, it would give a strong
indication that the MLE $\smumle$ will achieve the $n^{-2/3}(\log
n)^{\gamma_d}$ rate without any additional lower bound
assumptions on $p_0$. On the other hand, if the minimax rate were to become
significantly slower, then obviously conditions on $p_0^{-1}$ are necessary for the
PW conjecture to hold. We highlight the determination of the minimax rate
for $\mathcal{P}_{\mathrm{SMU}}([0,M]^d,b = 0, B)$ as an open question. 

\subsection{The univariate ($d = 1$) case}
Specialized to $d = 1$, our results lead to superfluous logarithmic
factors multiplying the expected $n^{-2/3}$ rate for
$\smumle$ (which coincides with the Grenander estimator). For example,
specializing Proposition 
\ref{prodsmugen} to $d = 1$, we get that the rate of convergence of
$\smumle$ equals $n^{-2/3} (\log n)^2 (\log \log n)^2$ when $p_0$ is
concentrated on $[0, M]$ and bounded from above by $B$ for two positive
constants $B$ and $M$ (note that \eqref{prodsmugencon} is
automatically satisfied as $d = 1$). It turns out that our argument
(specifically Proposition \ref{mainhrbound}) 
can be sharpened for $d = 1$ which eliminates the additional $(\log
n)^2$ factor leading to the following result (proved in Subsection
\ref{proofs_1d}).  

\begin{proposition}\label{1dimproved}
    Suppose $p_0 \in \mathcal{P}_{\text{SMU}}(1)$ is concentrated on
    $[0, M]$ and bounded from above by $B$. Then there exists $C_{B,
      M} \in (0, \infty)$ such that for every $n \geq 3$
    \begin{equation}\label{1dimproved_eq}
        \mathbb{E} h^2\left(p_0, \hat{p}_{n, 1}^{\text{SMU}} \right)
        \leq C_{B, M} (\log \log n)^{2} n^{-2/3}.  
    \end{equation}
  \end{proposition}
It appears that it may not be possible to remove the $(\log \log n)^2$
factor using our proof techniques. \cite[Example 7.4.2]{VandegeerBook}
uses a different technique to achieve the univariate rate $n^{-2/3}$ without any
redundant logarithmic factors. For completeness, we state this result
below and include a proof in Subsection \ref{proofs_1d}. Note that
\cite[Theorem 7.12]{VandegeerBook} is an even stronger result that
replaces the boundedness and compact support assumptions by moment
restrictions on $p_0$. 

\begin{proposition}[van de Geer]\label{1d_vandegeer}
    Suppose $p_0 \in \mathcal{P}_{\text{SMU}}(1)$ is concentrated on $[0, M]$ and bounded from above by $B$. Then there exists $C_{B, M} \in (0, \infty)$ such that
    \begin{equation}\label{1dvandegeer_eq}
        \mathbb{E} h^2\left(p_0, \hat{p}_{n, 1}^{\text{SMU}} \right) \leq C_{B, M} n^{-2/3}. 
    \end{equation}
\end{proposition}
Although our proof of Proposition
\ref{1d_vandegeer}  differs slightly
from that in \cite[Example 7.4.2]{VandegeerBook}---being based on
Theorem \ref{hellexp}---it employs the same central idea: exploiting
the fact that the function 
\begin{equation}\label{keymon}
   x \mapsto \frac{p(x) p_0(x)}{p(x) + p_0(x)}
\end{equation}
is non-increasing on $[0, M]$ for every $p, p_0 \in
\mathcal{P}_{\text{SMU}}(1)$. This monotonicity follows since the
right-hand side of \eqref{keymon} equals $(p_0^{-1}(x) +
p^{-1}(x))^{-1}$, where both $p^{-1}(x)$ and $p_0^{-1}(x)$ are
non-decreasing in $x$. This key observation allows control of the
$L_2(P_0)$ bracketing numbers of \eqref{brackentclass} (for
$\mathcal{P} = \mathcal{P}_{\text{SMU}}(1)$) directly using bracketing
numbers for non-increasing functions. 

However, when $d \geq 2$, while the function in \eqref{keymon} remains
coordinate-wise non-increasing, this property alone is insufficient to
achieve our desired rate for $\smumle$  which is $n^{-2/3} (\log
n)^{O(d)}$  because the bracketing numbers of coordinate-wise
non-increasing functions are quite large (see e.g.,
\cite{GW07}). Obtaining this rate requires additional structural
constraints on \eqref{keymon} (analogous to \eqref{smupd}), but it
remains unclear how such constraints might be derived. Consequently,
the univariate approach from \cite{VandegeerBook} does not readily
extend to the multivariate case $d \geq 2$.

To summarize, our argument bounds the bracketing entropy numbers of
\eqref{brackentclass} in terms of bracketing entropy numbers of SMU
densities $p$. On the other hand, in the univariate case, it is
possible to control the entropy numbers of \eqref{brackentclass}
directly. This direct method seems infeasible for $d \geq 2$. 

\subsection{When the domain is $[0, M_1] \times \dots \times [0, M_d]$}\label{domain}
In our Hellinger rate results: Corollaries \ref{conjecture}, \ref{finitelq} and \ref{conjlowbou}, we assumed that the domain of $p_0$ is $[0, M]^d$. As described below, these results also hold when the domain of $p_0$ is $[0,M_1] \times \ldots \times [0,M_d]$ with an appropriate modification of the underlying constants. We focus on Corollary \ref{conjlowbou} for simplicity. Suppose $p_0$ is
supported on $[0, M_1] \times \dots \times [0, M_d]$ and we are given
i.i.d observations $X_1, \dots, X_n$ from $p_0$. For each $i$, define
\begin{equation*}
  \tilde{X}_i := \left(X_{i1}/M_1, \dots, X_{id}/M_d \right)
\end{equation*}
where $X_{i1}, \dots, X_{id}$ denote the coordinates of $X_i$. Then it
is clear that
\begin{equation*}
  \tilde{X}_{1}, \dots, \tilde{X}_n \overset{\text{i.i.d}}{\sim}
  \tilde{p}_0  \qt{where $\tilde{p}_0(u_1, \dots, u_d) = p_0(u_1M_1,
    \dots, u_dM_d) M_1\dots M_d$}. 
\end{equation*}
$\tilde{p}_0$ is clearly concentrated on $[0, 1]^d$ and it is an SMU density (see the proof in Subsection \ref{domain_claims}). Also, $\tilde{p}_0$ is bounded
from above by $B M_1 \dots M_d$ and from below by $b M_1 \dots
M_d$. Let $\tilde{p}^{\text{SMU}}_{n, d}$ denote the SMU MLE for the
data $\tilde{X}_1, \dots, \tilde{X}_n$. Then Corollary 4.4 implies
that 
\begin{equation}\label{tilbound}
  \E h^2 \left(\tilde{p}_0, \tilde{p}^{\text{SMU}}_{n, d} \right) \leq
  C_{d, B, M_1, \dots, M_d, b} n^{-2/3} (\log 
n)^{4d-2}. 
\end{equation}
Next we observe that
\begin{equation}\label{mle_scaling}
  \tilde{p}^{\text{SMU}}_{n, d} (u_1, \dots, u_d) =
  \hat{p}^{\text{SMU}}_{n, d}(u_1 M_1, \dots, u_d M_d) M_1 \dots M_d
\end{equation}
where $\hat{p}^{\text{SMU}}_{n, d}$ is the SMU MLE based on the
original data $X_1, \dots, X_n$. The proof of \eqref{mle_scaling} is given in Subsection \ref{domain_claims}. Then
it follows by scale invariance of the Hellinger distance that
\begin{align*}
  h^2  \left(\tilde{p}_0, \tilde{p}^{\text{SMU}}_{n, d} \right)& =
  h^2\left(p_0(\cdot M_1, 
    \dots, \cdot M_d) M_1\dots M_d, \hat{p}^{\text{SMU}}_{n, d}(\cdot
    M_1, \dots, \cdot M_d) M_1 \dots M_d \right)\\
    &= h^2 (p_0,
  \hat{p}^{\text{SMU}}_{n, d}) 
\end{align*}
which proves that $h^2 (p_0,
  \hat{p}^{\text{SMU}}_{n, d}) $ is also bounded by the right hand
  side of \eqref{tilbound}. 

The minimax lower bound in Theorem \ref{lowerbound} can also be similarly extended to the case with domain $[0, M_1] \times \dots \times [0, M_d]$. The construction of $f_\alpha(\bm x)$ for $\bm x \in
[0,1]^d$ and $b \leq f_\alpha \leq B$ in the proof of Theorem \ref{lowerbound} is modified by
considering $\tilde f_\alpha(\bm x) =\frac{1}{M_1 \ldots M_d} f_\alpha(x_1/M_1, x_2/M_2, \ldots, x_d/M_d)$ where
$\bm x \in [0,M_1] \times [0,M_2] \times \ldots \times [0,M_d]$. Because $\frac{1}{M_1 \ldots M_d} b\leq \tilde f_\alpha(\bm x) \leq
\frac{1}{M_1 \ldots M_d} B$, the minimax lower bound we obtain still holds with a constant $c$ which depends on $d,b,B,M_1, \ldots, M_d$. Thus without loss of generality, we can let $M_1 = \ldots = M_d = 1$, $b=1/2$, and let $B=3/2$ as used in Theorem \ref{lowerbound}.

\section{Computational details and numerical experiments}\label{sec:comp}

We focused mainly on the theoretical convergence rates of the MLE $\smumle$ in this paper. Computational details, which also do not seem to have been studied previously in the literature, are discussed in this section. Using the expression \eqref{pwsmudef} for $p$ in \eqref{smumle_def}, it follows that the optimization problem underlying the MLE involves maximization of: 
\begin{equation}\label{G_opt}
   \frac{1}{n} \sum_{i=1}^n \log \left(\int_0^{\infty}  \dots
  \int_0^{\infty} p_{\mathrm{Unif(0, 
      \theta_1]}}(x_{i1}) \dots
  p_{\mathrm{Unif(0, \theta_d]}}(x_{id})  dG(\theta_1,  \dots,
    \theta_d) \right)
\end{equation}
over all probability measures $G$ on $(0, \infty)^d$ (here $x_{i1}, \dots, x_{id}$ denote the coordinates of the $i^{th}$ data point $x_i$). 

This maximization is a convex optimization problem as the objective function \eqref{G_opt} is concave in $G$ and the constraint set is the space of all probability measures on $(0,\infty)^d$ which is a convex class. However it is an \textit{infinite-dimensional} optimization problem 
as the optimization variable $G$ takes values in the infinite-dimensional set of all probability measures on $(0, \infty)^d$. To solve it,
we need to reduce $G$ to be a probability measure in a finite-dimensional space. This can be done using results of \cite[Section
3]{pavlides2012nonparametric} which ensure that $G$ can be restricted
to be a discrete probability measure that is supported on the
\textit{rectangular grid} generated by the data $x_1, \dots, x_n$. The rectangular grid generated by the data points $x_i := (x_{i1},
\dots, x_{id})$ for $i = 1, \dots, n$ is given by
\begin{equation}\label{rect_grid}
  A := \left\{(x_{(i_1), 1}, x_{(i_2), 2}, \dots, x_{(i_d), d}) : i_1,
  \dots, i_d \in \{1, \dots, n\}\right\}
\end{equation}
where $x_{(i),j}$ denotes the $i^{th}$ smallest element among $x_{1j},
\dots, x_{nj}$ for $1 \leq i \leq n, 1 \le j \le d$. Restricting $G$ in \eqref{G_opt} to be supported on $A$, we get
\begin{equation*}
\argmax_{\{w_{\theta}, \theta \in A\}: w_{\theta} \geq 0, \sum_{\theta
  \in A} w_{\theta} = 1} \frac{1}{n}\sum_{i=1}^n \log
\left(\sum_{\theta = (\theta_1, \dots, \theta_d) \in A} \frac{I\{x_{i1} \leq \theta_1, \dots, x_{id} \leq
    \theta_d\}}{\theta_1 \dots \theta_d} w_{\theta} \right)
\end{equation*}
This is a finite-dimensional optimization problem that can be solved using standard software for convex optimization. Here is the overall algorithm: 
%\begin{enumerate}
%\item Input is $n$ points $x_1, \dots, x_n$ in $(0, \infty)^d$.
%\item Create the rectangular grid \eqref{rect_grid} generated %by $x_1, \dots, x_n$. Suppose the
%elements of $A$ are given by $\theta^{(1)}, \dots, \theta^{(N)}$ for
%some $N \geq 1$.
%\item We now obtain weights $\hat{w}_1, \dots, \hat{w}_N$ (where
 % $\hat{w}_j$ corresponds to $\theta^{(j)}$) by solving the %convex
 % optimization problem: 
%\begin{equation*}
%\argmax_{w_j \geq 0, \sum_{j=1}^N w_j = 1} \frac{1}{n}\sum_{i=1}^n \log
%\left(\sum_{j = 1}^N w_j \frac{I\{x_{i1} \leq \theta_{j1}, \dots, x_{id} \leq
%    \theta_{jd}\}}{\theta_{j1} \dots \theta_{jd}}  \right)
%\end{equation*}
%where $(\theta_{j1}, \dots, \theta_{jd})$ denote the elements of the
%vector $\theta^{(j)}$.
%\item The MLE is given by
%  \begin{equation*}
%    \smumle(x) := \sum_{j = 1}^N \hat{w}_j \frac{I\{x_{1} \leq \theta_{j1}, \dots, x_{d} \leq
%    \theta_{jd}\}}{\theta_{j1} \dots \theta_{jd}}. 
%  \end{equation*}
%\end{enumerate}

\begin{algorithm}[htbp]
\caption{SMU MLE Exact Algorithm}
\label{alg:smu_mle}
\begin{algorithmic}[1]    % “1” switches on line-numbers
  \Require $n$ data points $x_1,\dots,x_n \in (0,\infty)^d$
  \Ensure SMU MLE $\hat f_{\mathrm{SMU}}(x)$

  \State Construct the rectangular grid \eqref{rect_grid} generated by $\{x_i\}_{i=1}^n$; denote its vertices by $\theta^{(1)},\dots,\theta^{(N)}$.
  \State Obtain weights $\hat w_1,\dots,\hat w_N$ by solving
        \[
          \max_{\substack{w_j\ge 0\\ \sum_{j=1}^N w_j = 1}}
          \frac1n\sum_{i=1}^n
          \log\!\Bigl(
            \sum_{j=1}^{N} w_j\,
            \frac{\mathbf 1\{x_{i1}\le\theta_{j1},\dots,
                         x_{id}\le\theta_{jd}\}}
                 {\theta_{j1}\cdots\theta_{jd}}
          \Bigr),
        \]
        where $\theta^{(j)}=(\theta_{j1},\ldots,\theta_{jd})$.
  \State \textbf{Return}
        \[
          \hat f_{\mathrm{SMU}}(x)\;=\;
          \sum_{j=1}^N
            \hat w_j\,
            \frac{\mathbf 1\{x_1\le\theta_{j1},\dots,
                            x_d\le\theta_{jd}\}}
                 {\theta_{j1}\cdots\theta_{jd}}.
        \]
\end{algorithmic}
\end{algorithm}

The computational cost of Algorithm \ref{alg:smu_mle} grows with the size of 
$N$. In the worst-case scenario, $N$ can be as large as $n^d$. When the dimension is modest—say $d = 2$—and 
$n$ remains moderate, an exact implementation is still practical. As either $d$ or $n$ increases, however, the workload rises sharply, and an exact computation soon becomes infeasible.

\begin{figure}[htbp]
  \centering
  % Adjust width as needed (e.g., 0.8\linewidth, 10cm, etc.)
  \includegraphics[width=\linewidth]{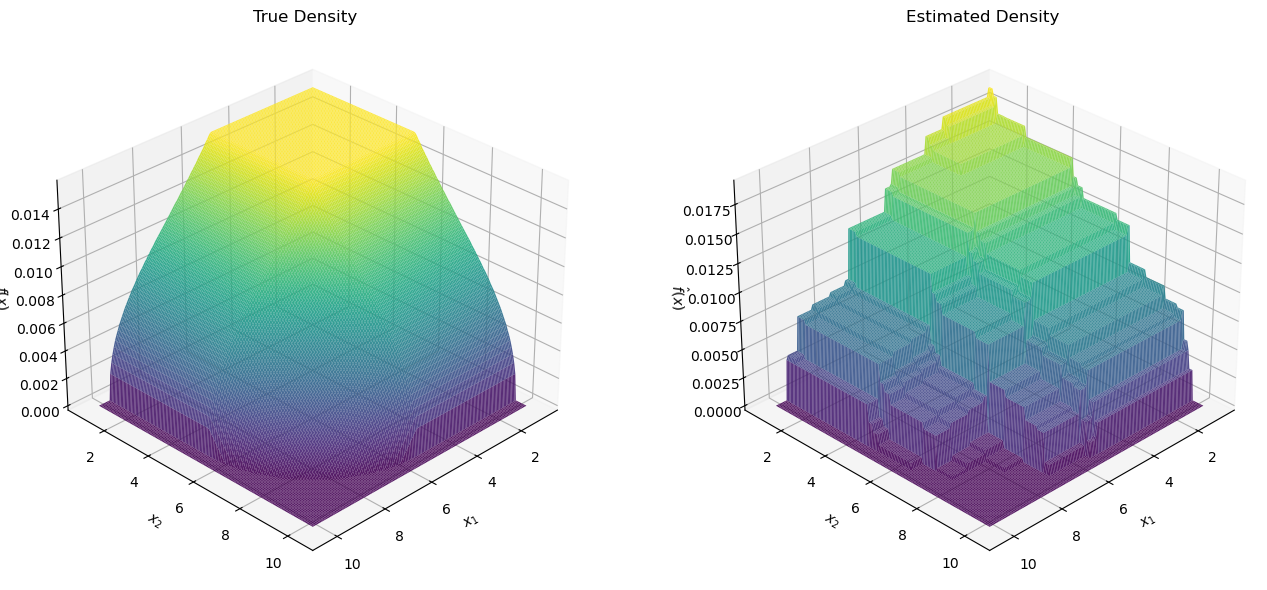}
  \caption{True density (left panel) $p_0$ and Estimated density computed using Algorithm \ref{alg:smu_mle} from $n = 400$ points drawn from $p_0$. Here $p_0$ is given by \eqref{pwsmudef} with $G$ taken to be the discrete uniform distribution on $\{\vartheta_1,\dots, \vartheta_n\}$ where $\vartheta_j = (5 + 5 \cos(\pi j/(2n)), 5 + 5 \sin (\pi j /(2n)))$}
  \label{fig:plot:data_estimate_fullmethod_n400}
\end{figure}

Figure \ref{fig:plot:data_estimate_fullmethod_n400} shows a true density in $\smud$ (left panel) along with $\smumle$ (right panel) computed, using Algorithm \ref{alg:smu_mle}, from $n = 400$ data points drawn from the true density. 

When $n$ is larger (even with $d = 2$), Algorithm \ref{alg:smu_mle} becomes computationally infeasible because of the large size of the rectangular grid. In such cases, it is natural to take a smaller set of points $\theta^{(1)}, \dots, \theta^{(N)}$ in Step 2. Motivated by the exemplar algorithm from Gaussian location density mixture fitting (see e.g., \cite{lashkari2007convex, bohning1999computer, soloff2025multivariate}), it is natural to take $\theta^{(j)} = x_j$ for $j = 1, \dots, n$. In other words, we take the $\theta$-vectors to be just the data points. This significantly reduces the number of $\theta$-vectors (from $N$, which can be as large as $n^d$, to $n$) and allows computation. However, the resulting density estimate will only be an approximate MLE. This algorithm is summarized below.

\begin{algorithm}[htbp]
\caption{SMU MLE Approximate Algorithm}
\label{alg:smu_mle_approx}
\begin{algorithmic}[1]    % “1” switches on line-numbers
  \Require $n$ data points $x_1,\dots,x_n \in (0,\infty)^d$
  \Ensure An approximate SMU MLE $\hat f_{\mathrm{SMU-APPROX}}(x)$

  \State Take $\theta^{(i)} = x_i$ for $i = 1, \dots, n$. 
  \State Obtain weights $\hat w_1,\dots,\hat w_n$ by solving
        \[
          \max_{\substack{w_j\ge 0\\ \sum_{j=1}^n w_j = 1}}
          \frac1n\sum_{i=1}^n
          \log\!\Bigl(
            \sum_{j=1}^{n} w_j\,
            \frac{\mathbf 1\{x_{i1}\le\theta_{j1},\dots,
                         x_{id}\le\theta_{jd}\}}
                 {\theta_{j1}\cdots\theta_{jd}}
          \Bigr),
        \]
        where $\theta^{(j)}=(\theta_{j1},\ldots,\theta_{jd})$.
  \State \textbf{Return}
        \[
          \hat f_{\mathrm{SMU-APPROX}}(x)\;=\;
          \sum_{j=1}^n
            \hat w_j\,
            \frac{\mathbf 1\{x_1\le\theta_{j1},\dots,
                            x_d\le\theta_{jd}\}}
                 {\theta_{j1}\cdots\theta_{jd}}.
        \]
\end{algorithmic}
\end{algorithm}
Figure \ref{fig:plot:data_estimate_approximate_n2000} shows the result of applying Algorithm \ref{alg:smu_mle_approx} to $n = 2000$ data points drawn from the same true density $p_0$ as in Figure \ref{fig:plot:data_estimate_fullmethod_n400}. 
\begin{figure}[htbp]
  \centering
  % Adjust width as needed (e.g., 0.8\linewidth, 10cm, etc.)
  \includegraphics[width=\linewidth]{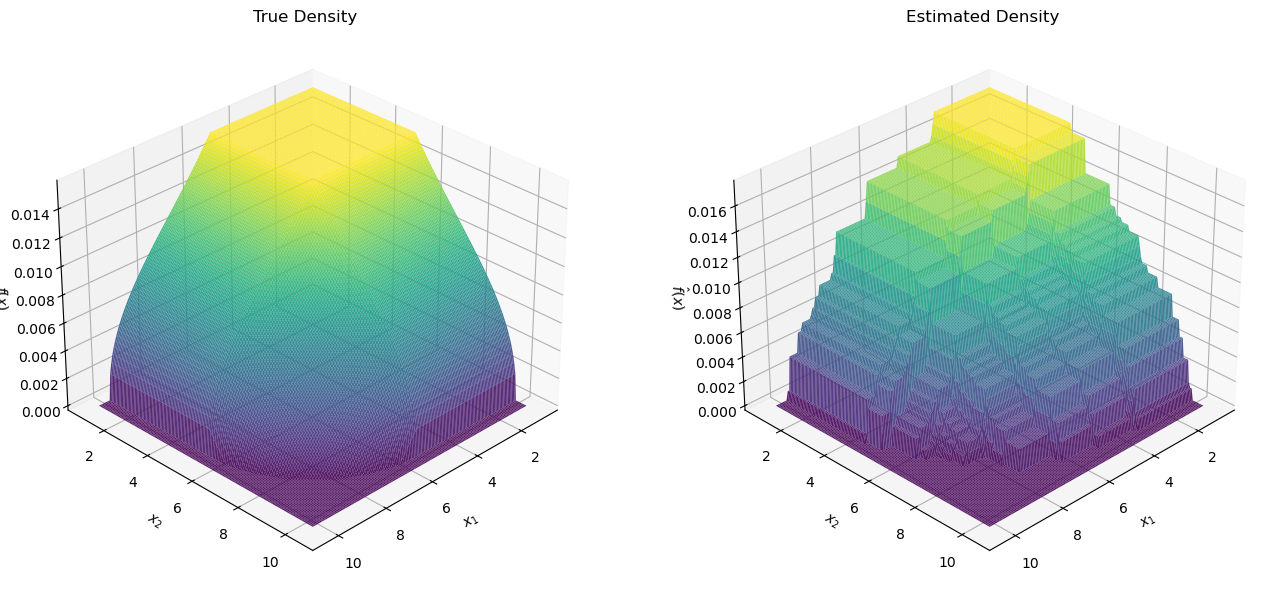}
  \caption{True density (left panel) $p_0$ (same as in Figure \ref{fig:plot:data_estimate_fullmethod_n400}) and estimated density computed using Algorithm \ref{alg:smu_mle_approx} from $n = 2000$ points drawn from $p_0$}
  \label{fig:plot:data_estimate_approximate_n2000}
\end{figure}

\paragraph{Application to leukaemia gene‐expression data}

We applied our method to the paired $p$-values obtained from the seminal
leukaemia micro-array study of Golub \emph{et al.}\,\cite{golub1999molecular}.
Their experiment profiled mRNA abundance in childhood leukaemia samples
using the Affymetrix \texttt{HGU‐6800} array, which reports expression
levels for 7\,129 probe sets (roughly one per gene).  
The data are publicly available on OpenML (data‐id~1104) and contain
expression measurements for 72 patients, each annotated with two
clinically relevant labels:

\begin{enumerate}
  \item \textbf{Disease subtype:} acute lymphoblastic leukaemia (ALL, 47 samples) vs.\ acute myeloid leukaemia (AML, 25 samples);
  \item \textbf{Sampling tissue:} bone marrow (BM, 24 samples) vs.\ peripheral blood (PB, 48 samples).
\end{enumerate}

For every gene we performed \emph{two} two‐sided Welch $t$‐tests—
splitting the samples according to each label—to obtain the following
hypothesis tests and their corresponding $p$-values:
\begin{enumerate}
  \item \textbf{ALL vs.\ AML:} compares the mean log‐expression of a gene across the two disease subtypes;
  \item \textbf{BM vs.\ PB:} compares the same gene across the two sampling tissues.
\end{enumerate}
Thus each of the 7\,129 genes contributes a pair of $p$-values.  These
pairs are shown in Figure~\ref{fig:pvalues}, and we fit an SMU density to
this bivariate $p$-value samples. Here is an argument for why the SMU model is suitable here. Under a true null hypothesis, the $p$-value is uniformly distributed on
$(0,1]$.  Under an alternative, it tends to be small, and a
$\mathrm{Unif}(0,\theta]$ distribution with $\theta\ll1$ is a reasonable
approximation \cite{strimmer2008unified}.  For each gene $i$ we therefore
associate a latent vector $\theta_i=(\theta_{i1},\theta_{i2})$ and model
\begin{equation}
  x_{ij}\,\stackrel{\mathrm{ind}}{\sim}\,
  \mathrm{Unif}\bigl(0,\theta_{ij}\bigr],\qquad j=1,2,
  \label{eq:unif_model}
\end{equation}
where $\theta_{ij}=1$ if gene~$i$ is null for hypothesis~$j$ and
$\theta_{ij}\ll1$ otherwise.  Assuming the genes are a priori
exchangeable, we posit
\begin{equation}
  \theta_1,\dots,\theta_n
  \,\stackrel{\text{i.i.d.}}{\sim}\,G
  \label{eq:theta_prior}
\end{equation}
for some mixing measure~$G$.  Combining
\eqref{eq:unif_model}–\eqref{eq:theta_prior} yields the piecewise‐SMU
model in~\eqref{pwsmudef}.

\begin{figure}[htbp]
  \centering
  \includegraphics[width=0.8\linewidth]{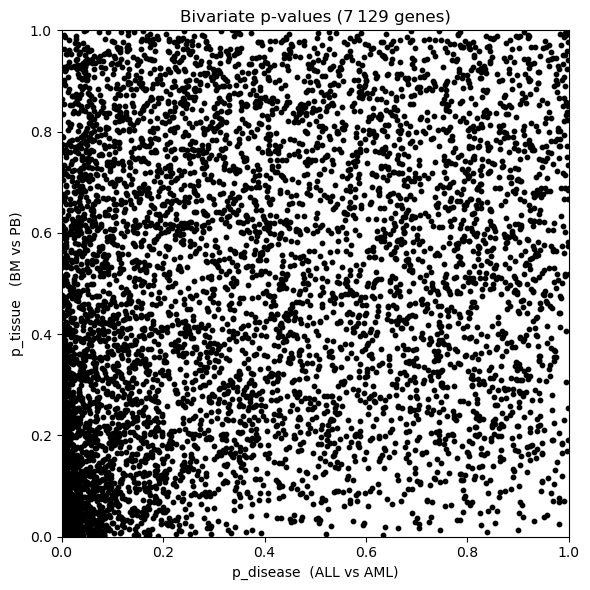}
  \caption{7129 $p$-values obtained from the micro-array dataset of \cite{golub1999molecular}. Each point in this plot corresponds to one gene, and represents a pair of $p$-values.}
  \label{fig:pvalues}
\end{figure}

We used Algorithm~\ref{alg:smu_mle_approx} to compute the approximate
SMU maximum‐likelihood estimator (MLE).  The resulting density estimate
is shown in Figure~\ref{fig:smu_pvalues}.

Figure~\ref{fig:smu_pvalues} reveals a mixture of a uniform component
and a component concentrating near the origin, the latter capturing the
non‐null $p$-values.

We would like to emphasize that the density estimates shown in Figures \ref{fig:plot:data_estimate_fullmethod_n400}, \ref{fig:plot:data_estimate_approximate_n2000} and \ref{fig:smu_pvalues} do not involve any tuning parameters (unlike other density estimation techniques such as those based on kernel density estimation). This is an important advantage of methods such as $\smumle$ which are based on shape constraints. 

In Figures \ref{fig:plot:data_estimate_fullmethod_n400}, \ref{fig:plot:data_estimate_approximate_n2000} and \ref{fig:smu_pvalues}, we actually do not plot the densities for points $(x_1, x_2)$ for which either $x_1$ or $x_2$ is too close to zero. At such points, $\smumle$ has a tendency to overfit leading to very large values for the fitted density. This behavior has been observed previously in other shape-constrained estimation problems \cite{woodroofe1993penalized, sun1996adaptive, kulikov2006behavior, lim2012consistency, mazumder2019computational, liao2024convex, liao2024overfitting}. 

\begin{figure}[htbp]
  \centering
  \includegraphics[width=\linewidth]{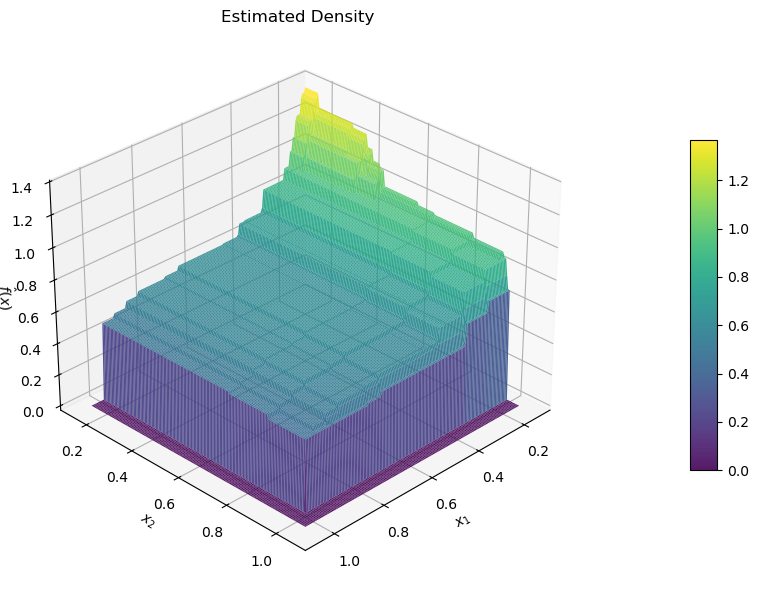}
  \caption{An approximate SMU MLE fitted to the data in Figure \ref{fig:pvalues} using Algorithm \ref{alg:smu_mle_approx}.} 
  \label{fig:smu_pvalues}
\end{figure}

\section{Proofs of main results}\label{proofs}

This section contains proofs of Theorems
\ref{hellexp}, \ref{jrec}, \ref{adaptation}, \ref{lowerbound}, 
Proposition \ref{prodsmugen}, Proposition \ref{1dimproved} and Proposition \ref{1d_vandegeer}. 
Note that Corollary \ref{conjecture} is
the special case of Theorem \ref{jrec} (for $J = 1$), Corollary
\ref{finitelq} is simply a restatement of Corollary \ref{conjecture},
Corollary \ref{conjlowbou} is a consequence of Corollary
\ref{finitelq} because $L_{\qh}$ norms for finite $\qh$ on a compact
rectangle can be bounded from above using the $L_{\infty}$ norm.
Due
to these reasons, we do not need to provide proofs for Corollaries
\ref{conjecture}, \ref{finitelq}, \ref{conjlowbou}. We also note 
that the lemmas stated in Section \ref{sec:bracken} are proved in
Section \ref{additechres}.

\subsection{Proof of Theorem \ref{hellexp}}
\color{black}
First we use convexity arguments to prove
that
\begin{equation*}
  s^2 \leq G(s) \qt{for all $0 \leq s \leq h(\hat{p}_n, p_0)$}
\end{equation*}
where $G(\cdot)$ is as defined in \eqref{gdef}. From here, the proof
is completed by use of the Bousquet concentration inequality for
suprema of empirical processes (see, for example, \cite[Theorem
12.5]{boucheron2013concentration}).   
\color{black}

\begin{proof}[Proof of Theorem \ref{hellexp}]
%\section{Proof of Theorem \ref{hellexp}} 
Because $\hat{p}_n$ is the MLE over $\Ps$, the function
\begin{equation*}
  g(\alpha) := \frac{1}{n} \sum_{i=1}^n \log \left((1 - \alpha)
    \hat{p}_n(X_i) + \alpha p(X_i) \right)
\end{equation*}
for $\alpha \in [0, 1]$ is maximized at $\alpha = 0$ for every $p \in
\Ps$. This implies 
that $g'(0+) \leq 0$ which gives 
\begin{equation*}
  \frac{1}{n} \sum_{i=1}^n \frac{p(X_i)}{\hat{p}_n(X_i)} \leq 1
  \qt{for every $p \in \Ps$}. 
\end{equation*}
The above inequality is equivalent to 
\begin{equation*}
\frac{1}{n} \sum_{i=1}^n \left(\frac{1}{2}
  \frac{p(X_i)}{\hat{p}_n(X_i)} + \frac{1}{2}
  \frac{p(X_i)}{p(X_i)}\right) \leq \frac{1}{2} + \frac{1}{2} = 1. 
\end{equation*}
Using convexity of the map $u \mapsto p(X_i)/u$, we obtain
\begin{equation}\label{kkt}
  \frac{1}{n} \sum_{i=1}^n \frac{2p(X_i)}{p(X_i) + \hat{p}_n(X_i)} \leq 1
  \qt{for every $p \in \Ps$}. 
\end{equation}
Specializing the above inequality to $p = p_0$, we get (below $P_0$ is
the probability measure having density $p_0$ and $P_n$ is the
empirical distribution) 
\begin{equation*}
  1 \geq \int \frac{2p_0}{p_0 + \hat{p}_n} dP_n = \int \frac{2p_0}{p_0 +
    \hat{p}_n} dP_0 + \int\frac{2p_0}{p_0 + \hat{p}_n} d(P_n - P_0). 
\end{equation*}
This gives
\begin{equation}\label{temp11}
  \int\frac{2p_0}{p_0 +
    \hat{p}_n} dP_0 -1\leq \int
  \frac{2p_0}{p_0 + \hat{p}_n} d(P_0 - P_n). 
\end{equation}
Also note that for any pair of densities $p$ and $q$:
\begin{align*}
  h^2(p, q) &= \int \left( \sqrt{p} - \sqrt{q} \right)^2  \\
            &= \int \frac{(p - q)^2}{(\sqrt{p} + \sqrt{q})^2} \\
  &\leq \int \frac{(p - q)^2}{p + q} = \int \frac{4 p^2 + (p+q)^2 - 
                                          4 p (p + q)}{p + q} = 2
    \left(\int \frac{2p^2}{p + q} - 1 \right). 
\end{align*}
With $p = p_0$ and $q = \hat{p}_n$, we get
\begin{align*}
h^2(\hat{p}_n, p_0) \leq  2\left(\int\frac{2p_0}{p_0 +
    \hat{p}_n} dP_0 -1 \right).
\end{align*}
Combining the above inequality with \eqref{temp11}, we get 
\begin{equation}\label{biloss}
  \hat{t}^2 \leq G(\hat{t}) \qt{where $\hat{t} := h(\hat{p}_n,
    p_0)$}. 
\end{equation}
Here the function $G(\cdot)$ is as defined in \eqref{gdef}. We now
claim that the above inequality is actually true for all $s \in [0, 
\hat{t}]$ i.e.,
\begin{equation}\label{strong.biloss}
  s^2 \leq G(s) \qt{for all $0 \leq s \leq \hat{t}$}. 
\end{equation}
To prove \eqref{strong.biloss}, assume, if possible, that $G(s) < s^2$
for some $0 < s < \hat{t}$. Suppose $\alpha_s \in (0, 1)$ is such that
\begin{equation}\label{exeq}
h(p_0, (1 - \alpha_s) p_0 + \alpha_s \hat{p}_n) = s. 
\end{equation}
Such an $\alpha_s \in (0, 1)$ exists because the function
\begin{equation*}
  \alpha \mapsto h(p_0, (1 - \alpha) p_0 + \alpha \hat{p}_n)
\end{equation*}
is continuous in $\alpha$, takes the value $0$ at $\alpha = 0$ and
$\hat{t}$ at $\alpha = 1$. We then get
\begin{equation*}
  \int \frac{4p_0}{p_0 + (1 - \alpha_s) p_0 + \alpha_s \hat{p}_n}
  d(P_0 - P_n) \leq
  G(s) < s^2 
\end{equation*}
which is equivalent to
\begin{equation*}
  h^2(p_0, (1 - \alpha_s) p_0 + \alpha_s \hat{p}_n) - s^2 + 2 -
  \int \frac{4p_0}{p_0 + (1 - \alpha_s) p_0 + \alpha_s \hat{p}_n} 
  dP_n  < 0. 
\end{equation*}
Because of \eqref{exeq}, the above is same as
\begin{equation*}
  \int \frac{2p_0}{p_0 + (1 - \alpha_s) p_0 + \alpha_s \hat{p}_n} 
  dP_n > 1. 
\end{equation*}
Using convexity of $x \mapsto 1/x$, we get
\begin{align*}
  1 &< \int \frac{2p_0}{p_0 + (1 - \alpha_s) p_0 + \alpha_s \hat{p}_n} 
s  dP_n \\ &= \int \frac{2p_0}{(1 - \alpha_s) (2p_0) + \alpha_s (p_0 + \hat{p}_n)} 
  dP_n \leq (1 - \alpha_s) + \alpha_s \int \frac{2p_0}{p_0 +
    \hat{p}_n} dP_n. 
\end{align*}
This gives
\begin{equation*}
  \int \frac{2p_0}{p_0 + \hat{p}_n} dP_n > 1
\end{equation*}
which contradicts \eqref{kkt}. This proves \eqref{strong.biloss}. 

Using \eqref{strong.biloss}, the probability on the left hand side of
\eqref{hellexp.eq} can be bounded as follows.
\begin{align*}
  \P \left\{h(p_0, \hat{p}_n) \geq t_0 + x \right\} &= \P
  \left\{\hat{t} \geq t_0 + x \right\} \\ &\leq \P \left\{G(t_0 + x) \geq
                                            (t_0 + x)^2 \right\} \\
  &\leq \P \left\{G(t_0 + x) - \E G(t_0 + x) \geq
                                            (t_0 + x)^2 - \E G(t_0 +
    x) \right\} \\
  &\leq    \P \left\{G(t_0 + x) - \E G(t_0 + x) \geq
                                            (t_0 + x)^2 - \bar{G}(t_0 +
    x) \right\}. 
\end{align*}
Because we assumed $\bar{G}(t)/t^{2 - \eta}$ is nonincreasing on $[t_0,
\infty)$ and $\bar{G}(t_0) \leq t_0^2$, we get
\begin{equation*}
  \frac{\bar{G}(t_0 + x)}{(t_0 + x)^{2-\eta}} \leq
  \frac{\bar{G}(t_0)}{t_0^{2 - \eta}} \leq t_0^{\eta}
\end{equation*}
so that 
\begin{equation}\label{bgb}
  \bar{G}(t_0 + x) \leq t_0^{\eta} (t_0 + x)^{2 - \eta}. 
\end{equation}
As a result
\begin{equation}\label{tei}
  \begin{split}
&  \P \left\{h(p_0, \hat{p}_n) \geq t_0 + x \right\} \\ &\leq \P \left\{G(t_0 + x) - \E G(t_0 + x) \geq
(t_0 + x)^{2-\eta} \left((t_0 + x)^{\eta} - t_0^{\eta}
    \right)\right\}.
  \end{split}  
\end{equation}
To bound the probability above, we use Bousquet's concentration
inequality for the supremum of an empirical process (see, for example,
\cite[Theorem 12.5]{boucheron2013concentration}) which gives 
\begin{equation}\label{bouscon}
  \P \left\{G(t) \geq \E G(t) + u \right\} \leq \exp \left(\frac{- n
      u^2}{16(\E G(t) + t^2 + \frac{u}{6})} \right) 
\end{equation}
for every $t > 0$ and $u \geq 0$. To see how \eqref{bouscon} is
obtained from Bousquet's inequality in 
the form stated in \cite[Theorem 12.5]{boucheron2013concentration},
just take the index set $\T := \left\{p \in \Ps : h(p_0, p) \le t \right\}$
and 
\begin{equation*}
  X_{i, p} := \int \frac{p_0}{p_0 + p} dP_0 - \frac{p_0(X_i)}{p_0(X_i)
  + p(X_i)}, 
\end{equation*}
so that $\sup_{p \in \mathcal{P}: h(p,p_0)\leq t}
\frac{1}{n}\sum_{i=1}^n X_{i,p} = G(t)$ and
\begin{align*}
&\sup_{s \in \T} \sum_{i=1}^n \mathrm{var}(X_{i, s})
  \\ &
  \leq n \sup_{p \in \Ps : h(p_0, p) \leq t}  \int
    \left(\frac{p_0}{p_0 + p} - \frac{1}{2} \right)^2 p_0 \\
  &=\frac{n}{4}\sup_{p \in \Ps : h(p_0, p) \leq t} \int
    \left(\frac{p-p_0}{p+p_0} \right)^2 p_0 \\ &\leq 
    \frac{n}{4} \sup_{p \in \Ps : h(p_0, p) \leq t} \int \frac{(p - p_0)^2}{p + p_0} \leq \frac{n}{2}
 \sup_{p \in \Ps : h(p_0, p) \leq t}   h^2(p_0, p) \leq \frac{n t^2}{2}. 
\end{align*} 
Applying \eqref{bouscon} to $t = t_0 + x$ and $u =  
(t_0 + x)^{2-\eta} \left((t_0 + x)^{\eta} - t_0^{\eta} \right)$, we
get 
(via \eqref{tei})
\begin{align*}
&  \P \left\{h(p_0, \hat{p}_n) \geq t_0 + x \right\} \\ &\leq \exp
  \left(\frac{-n (t_0 + x)^{4-2\eta} \left((t_0 + x)^{\eta} - t_0^{\eta} \right)^2}{16 \left( \E G(t_0 + x) + {(t_0
  + x)^2} + \frac{(t_0 + x)^{2-\eta} \left((t_0 + x)^{\eta} - t_0^{\eta} \right)}{6}\right)} \right).
\end{align*}
Using $\E G(t_0 + x) \leq \bar{G}(t_0 + x)$ and the bound \eqref{bgb}
on $\bar{G}(t_0 + x)$, we obtain
\begin{align*}
&  \P \left\{h(p_0, \hat{p}_n) \geq t_0 + x \right\} \\ &\leq \exp
  \left(\frac{-n (t_0 + x)^{4-2\eta} \left((t_0 + x)^{\eta} -
        t_0^{\eta} \right)^2}{16 \left(t_0^{\eta} (t_0 + x)^{2 - \eta} + {(t_0
  + x)^2} + \frac{(t_0 + x)^{2-\eta} \left((t_0 + x)^{\eta} -
    t_0^{\eta} \right)}{6}\right)} \right). 
\end{align*}
Because
\begin{align*}
&  t_0^{\eta} (t_0 + x)^{2 - \eta} + {(t_0
  + x)^2} + \frac{(t_0 + x)^{2-\eta} \left((t_0 + x)^{\eta} -
    t_0^{\eta} \right)}{6} \\ &= \frac{5}{6} t_0^{\eta} (t_0 + x)^{2 -
  \eta} + \frac{7}{6} (t_0 + x)^2 \leq 2(t_0 + x)^2, 
\end{align*}
we get
\begin{equation*}
  \P \left\{h(p_0, \hat{p}_n) \geq t_0 + x \right\} \leq \exp
  \left(\frac{-n (t_0 + x)^{2-2\eta} \left((t_0 + x)^{\eta} -
        t_0^{\eta} \right)^2}{\helconst} \right). 
\end{equation*}
We now use the elementary inequality (the first equality below holds
for some $\tilde x \in [0, x]$ by the mean value theorem): 
\begin{equation*}
  (t_0 + x)^{\eta} - t_0^{\eta} = \frac{\eta x}{(t_0 + \tilde{x})^{1- \eta}}
  \geq \frac{\eta x}{(t_0 + x)^{1 - \eta}}
\end{equation*}
which proves \eqref{hellexp.eq}. To prove
\eqref{hellexp.exp}, just mulitply both sides of \eqref{hellexp.eq} by
$x$ and integrate from $x = 0$ to $x = \infty$ to get
\begin{align*}
  \E \left(h(p_0, \hat{p}_n) - t_0 \right)_+^2 \leq \frac{16}{n
  \eta^2} 
\end{align*}
and then use $a^2 \leq 2 (a - b)_+^2 + 2 b^2$ for $a, b \geq 0$. This
completes the proof of Theorem \ref{hellexp}. 
\end{proof}

\color{black}

\subsection{Initial steps in appplying Theorem \ref{hellexp} with $\Ps = \smud$} \label{init_steps}
Theorem \ref{hellexp} (with $\Ps = \smud$) is our starting point 
for proving the Hellinger accuracy results for $\smumle$.  The next
step is to prove upper bounds for   
\begin{align}
      \E G(t) &= \E \sup_{p \in \smud : h(p_0, p) \leq t} \int 
                \frac{4p_0}{p_0 + p} d(P_0 - P_n) \nonumber \\ &=\E \sup_{p \in
                                                       \smud : h(p_0,
                                                       p) \leq t} \int \left[ 
                \frac{4p_0}{p_0 + p} - 2 \right]d(P_0 - P_n) \nonumber  \\& =
                                                       2 \E \sup_{p \in \smud : h(p_0,
      p) \leq t} \int  \frac{p_0 - p}{p_0 + p} d(P_0 - P_n).
  \label{altexpEGt}
  \end{align}
For this, we decompose the support of $P_0$ into a finite
collection of rectangles $R_1, \dots, R_J$ whose pairwise
intersections have zero volume, and then use the bound:
\begin{align}
  \E G(t) &= 2 \E \sup_{p \in \smud : h(p_0,
            p) \leq t} \sum_{i=1}^J \int \frac{p_0 - p}{p_0 + p} \one(R_i)
            d(P_0 - P_n) \nonumber \\ &\leq 2 \sum_{i=1}^J \E \sup_{p \in \smud : h(p_0,
            p) \leq t} \int \frac{p_0 - p}{p_0 + p} \one(R_i)
            d(P_0 - P_n). \label{eledeco}
\end{align}
Here $\one(R)$ denotes the indicator function for the set $R$. The
$i^{th}$ term in the above sum is 
\begin{align}\label{secupbo}
H(t, R_i) :=  \E \sup_{p \in \smud : h(p_0,
            p) \leq t} \int \frac{p_0 - p}{p_0 + p} \one(R_i)
            d(P_0 - P_n)
\end{align}
and we employ two upper bounds for the above quantity. The first upper
bound is the trivial one  
obtained by replacing $\frac{p_0 - p}{p_0 + p}$ by 1:
\begin{align}\label{trivialb}
H(t, R_i) \leq   \E \sup_{p \in \smud : h(p_0,
            p) \leq t} \int_{R_i} d(P_0 + P_n) = 2 P_0(R_i), 
\end{align}
and this bound will be useful when $P_0(R_i)$ is small. The second upper
bound on \eqref{secupbo} is obtained from the use of Theorem \ref{expcontrol} with $\F$ defined in \eqref{Fdest}. 

This bound involves bracketing entropy numbers of $\F$ under
the $L_2(P_0)$ metric. Results on these bracketing entropy numbers
are provided in Section \ref{sec:bracken}.
\color{black}

\subsection{Proofs of Theorem \ref{jrec} and Theorem
  \ref{adaptation}}
The proofs of Theorem \ref{jrec} and Theorem \ref{adaptation}
will both be based on the following result which provides an upper
bound on an expected supremum. 

\begin{lemma}\label{ESupResult}
  Consider the rectangle $R := [a_1, b_1] \times \dots \times [a_d,
  b_d]$ with $0 \leq a_j < b_j$ for each $j = 1, \dots,
  d$. For $t > 0$, let
\begin{equation*}
H(t, R) := \E \sup_{p \in \smud : h(p_0, p) \leq t} \int
    \frac{p_0 - p}{p_0 + p} \one(R) d(P_0 - P_n). 
  \end{equation*}
Then, for every $\qh \in (1, \infty]$, the quantity $H(t, R)$ is
bounded from above by: 
  \begin{equation}\label{alphabetaclaim}
    \begin{split}
&C_{d, \qh} \sqrt{\frac{t}{n}} \sqrt{\beta - \alpha}
      |R|^{\frac{1}{4\ph}} \|p_0^{-1}\|^{1/4}_{L_{\qh}(R)} \left[\log
    \left(e + \frac{2e (\beta - \alpha) |R|^{\frac{1}{2\ph}}
               \|p_0^{-1}\|_{L_{\qh}(R)}^{1/2}}{t \sqrt{2}} 
               \right)\right]^{d-1} \\ &+ 
  \frac{C_{d, \qh}}{nt} (\beta - \alpha) |R|^{\frac{1}{2\ph}}
                                         \|p_0^{-1}\|_{L_{\qh}(R)}^{1/2}
                                         \left[\log \left(e + \frac{2
                                         e (\beta 
        - \alpha) |R|^{\frac{1}{2\ph}} \|p_0^{-1}\|_{L_{\qh}(R)}^{1/2}}{t
                                         \sqrt{2}} 
    \right) \right]^{2(d-1)} 
    \end{split}
  \end{equation}
  where $\alpha$ and $\beta$ are given by
 \begin{equation*}
\alpha = L(R) := \inf_{\substack{p \in \smud \\ h(p, p_0) \leq t}}
   \inf_{x \in R} p(x) ~~~ \text{ and } ~~~ \beta = U(R) := \sup_{\substack{p \in \smud \\ h(p, p_0) \leq t}}
   \sup_{x \in R} p(x), 
 \end{equation*}  
Also, in \eqref{alphabetaclaim}, $C_{d, \qh}$ is a constant that depends on
$d$ and $\qh$ alone, and $\ph$ is such that $1/\ph + 1/\qh = 1$. 
\end{lemma}

\begin{proof}[Proof of Lemma \ref{ESupResult}]
We write 
  \begin{equation*}
H(t, R)  = \E \sup_{f \in \F}
 \left(P_0 f - P_n f \right) 
  \end{equation*}
  where
  \begin{equation*}
\F := \left\{\frac{p_0 - p}{p_0 +
          p} \one(R) : p \in \smud, h(p_0, p) \leq t \right\},    
  \end{equation*}
and apply Theorem \ref{expcontrol} to bound the right hand side
  above. The quantity $\delta$ appearing in Theorem
  \ref{expcontrol} can be taken to be equal to $t \sqrt{2}$ because
  for every $p \in \mathfrak{P}(\alpha, \beta)$ and $f := \frac{p_0 - p}{p_0 + p}
  \one(R)$, we have (below $X_1 \sim P_0$),
  \begin{align*}
    \E f^2(X_1) &= \int_R \frac{(p_0 - p)^2}{(p_0 + p)^2}p_0 \\ & \leq
    \int \frac{(p_0 - p)^2}{(p_0 + p)^2}p_0 \\ & = \int \left(\sqrt{p_0} -
    \sqrt{p} \right)^2 \frac{(\sqrt{p_0} + \sqrt{p})^2 p_0}{(p_0 +
    p)^2} \\ &\leq \int \left(\sqrt{p_0} -
    \sqrt{p} \right)^2 \frac{2(p_0 + p) p_0}{(p_0 +
    p)^2} \leq 2 h^2(p_0, p) \leq 2 t^2. 
  \end{align*}
  The quantity $M$ appearing in Theorem \ref{expcontrol} can be taken
  to be one because $\frac{p_0 - p}{p_0 + p} \leq 1$. Theorem
  \ref{expcontrol} then implies
  \begin{align}
\E \sup_{p \in \smud: h(p_0, p) \leq t} &\int
    \frac{p_0 - p}{p_0 + p} \one(R) d(P_0 - P_n) \nonumber \\ 
    &\leq
    \frac{C}{\sqrt{n}} J(t \sqrt{2}) + \frac{C}{n t^2} J^2(t \sqrt{2}) \label{temp.jb}
  \end{align}
  where $J(\delta)$ is defined in Equation \eqref{braentint}.
%  \begin{align}\label{jdel}
%    J(\delta) := \int_0^{\delta} \sqrt{\log N_{[]}(\epsilon, \F,
%    L_2(P_0))} d\epsilon. 
%  \end{align}
To bound $J(\delta)$, we first use inequality \eqref{brack.lqbound} in
Lemma \ref{trivial} to get 
\begin{align*}
&  \log N_{[]}(\epsilon, \F, L_2(P_0)) \\ &\leq \log
  N_{[]}\left(\frac{\epsilon}{2\|p_0^{-1}\|_{L_{\qh}(R)}^{1/2}},
  \left\{p \one(R) : p \in \smud, h(p_0, p) \leq t \right\},
  L_{2\ph}(R) \right), 
\end{align*}
followed by Lemma \ref{brackd} to obtain
 \begin{align*}
   &\log N_{[]}(\epsilon, \F, L_2(P_0)) \\
   &\leq C_{d, \qh} \frac{\left(\beta - \alpha \right)
     |R|^{1/(2 \ph)} \|p_0^{-1}\|^{1/2}_{L_{\qh}(R)}}{\epsilon}
     \left(\log \frac{2\left(\beta - \alpha \right)
     |R|^{1/(2 \ph)} \|p_0^{-1}\|^{1/2}_{L_{\qh}(R)}}{\epsilon} \right)^{2(d-1)} 
 \end{align*}
 provided $\epsilon \leq 2 \left(\beta - \alpha \right)
 |R|^{1/(2 \ph)} \|p_0^{-1}\|^{1/2}_{L_{\qh}(R)}$.

 Plugging this bound in \eqref{braentint} and then applying Lemma
\ref{simpleintegral} leads to the following upper bound for
$J(\delta)$: 
\begin{equation*}
 C_{d, \qh} \sqrt{\delta} \sqrt{\beta - \alpha}
  |R|^{\frac{1}{4\ph}} \|p_0^{-1}\|_{L_{\qh}(R)}^{1/4} \left[\log \left(e +
      \frac{2 e (\beta - 
      \alpha) |R|^{\frac{1}{2\ph}} \|p_0^{-1}\|_{L_{\qh}(R)}^{1/2} }{\delta}
  \right) \right]^{d-1}.  
\end{equation*}
Combining this bound on $J(\delta)$ with \eqref{temp.jb}
leads to \eqref{alphabetaclaim} which completes the proof of Lemma
\ref{ESupResult}. 
\end{proof}

\subsubsection{Proof of Theorem \ref{jrec}}

The proof of Theorem \ref{jrec} will be based on the following
result.

\begin{proposition}\label{mainhrbound}
  Fix $n \geq 2$ and $\qh \in (1, \infty]$ with $\ph$ such that
  $1/(\ph) + 1/(\qh) = 1$. Suppose $R \subseteq [0, \infty)^d$ is the
  rectangle 
  given by $R = 
  [a_1, b_1] \times \dots \times [a_d, b_d]$. Then $H(t,R)$ (for the definition, see \eqref{secupbo})
%  \begin{align*}
%    H(t, R) := \E  \sup_{p \in \smud: h(p, p_0) \leq t} \int 
%    \frac{p_0 - p}{p_0 + p} \one(R) d(P_0 - P_n) 
%  \end{align*}
  satisfies the following bound for every $t \geq n^{-1/3}$: 
\begin{align*}
  H(t, R) &\leq C_{d, \qh} (\log n)^{2d - 1} \sqrt{\frac{t}{n}} \left(1 +
            n^{1/6}\sqrt{t} \right) W \max \left((\log W)^{d-1}, 1
            \right) \\
  &+ \frac{C_{d, \qh}}{nt} (\log n)^{3d-2} \left(1 + n^{1/6} \sqrt{t}
    \right)^2 W^2 \max\left((\log W)^{2d-2}, 1
            \right) \\ & + \frac{2d}{n}W^4 + 2(\log n)^d
                         \frac{t}{n^{1/3}} W
\end{align*}
where $W = W(R, p_0, \qh)$. 
\end{proposition}

\begin{proof}[Proof of Proposition \ref{mainhrbound}]
Fix $u = 1/n$ and let $I := \log_2(1/u) = \log_2
n$. Let $u_0 = 0$ and $u_i = 2^{i-1} u$ for $i = 1, \dots, I+1$. Note
then that $u_{I+1} = 1$. Consider the rectangles
\begin{equation*}
  R_{i_1, \dots, i_d} := \prod_{j=1}^d \left[a_j + u_{i_j} (b_j - a_j), a_j
  + u_{i_j+1} (b_j - a_j) \right]
\end{equation*}
for $0 \leq i_j \leq I$ and $j = 1, \dots, d$. All together, there are
$(I+1)^d$ rectangles $R_{i_1, \dots, i_d}$ as each $i_j$ ranges over
$0, 1, \dots, I$ for $j = 1, \dots, d$. Also all these rectangles have
disjoint interiors. We therefore have
  \begin{align*}
H(t, R) &:= \E  \sup_{p \in \smud: h(p, p_0) \leq t} \int 
    \frac{p_0 - p}{p_0 + p} \one(R) d(P_0 - P_n) \\ &\leq \sum_{i_1, \dots,
                                              i_d \in \{0, 1, \dots,
    I\}} H_{i_1, \dots, i_d}(t)
  \end{align*}
  where
\begin{align*}  
H_{i_1, \dots, i_d}(t) :=  \E \sup_{p \in \smud: h(p, p_0) \leq t}
  \int 
    \frac{p_0 - p}{p_0 + p} \one(R_{i_1, \dots, i_d}) d(P_0 - P_n) 
\end{align*}
We now apply Lemma \ref{ESupResult} to bound the above. Note that by
Lemma \ref{bds},  
\begin{align*}
\beta &=  U(R_{i_1, \dots, i_d}) = \sup_{p \in \smud: h(p, p_0) \leq t}
                           \sup_{x \in R_{i_1, \dots i_d}} p(x) \\
  &\leq  \sup_{p \in \smud: h(p, p_0) \leq t} p \left(a_1 + u_{i_1}(b_1 -
    a_1), \dots, a_d + u_{i_d}(b_d -
    a_d) \right) \\
  &\leq \left(\sqrt{p_0(a_1, \dots, a_d)} +
    \frac{t}{\sqrt{\prod_{j=1}^d u_{i_j} (b_j - a_j)}} \right)^2 \\ &=
    \left(\sqrt{p_0(\mathbf{a})} + \frac{t}{\sqrt{|R|} \sqrt{u_{i_1}
    \dots u_{i_d}}} \right)^2
\end{align*}
where $\mathbf{a} := (a_1, \dots, a_d)$ and $|R| = (b_1 - a_1) \dots
(b_d - a_d)$. Observe that $u_{i_j}$ can equal 0 (when $i_j = 0$) in
which case the right hand side above will equal $+\infty$. Applying
\eqref{alphabetaclaim} with $\beta$ replaced by the
right hand side above, $\alpha = 0$, we obtain the
following bound in which we use the notation
\begin{align*}
  \Upsilon_{\mathbf{i}} := |R_{i_1, \dots, i_d}|^{1/(2 \mathfrak{p})}
  \|p_0^{-1} \|^{1/2}_{L_{\mathfrak{q}}(R_{i_1, \dots, i_d})}. 
\end{align*}
Our upper bound on $H_{i_1, \dots, i_d}(t)$ is given by
\begin{equation}\label{firbo}
  \begin{split}
&C_{d, \qh} \sqrt{\frac{t \Upsilon_{\mathbf{i}}}{n}} \left(\sqrt{p_0(\mathbf{a})} +
    \frac{t}{\sqrt{|R|} \sqrt{u_{i_1} \dots u_{i_d}}} \right)\times \\
& \left[\log \left(e + 
    \frac{2e\Upsilon_{\mathbf{i}}}{t \sqrt{2}}
    \left(\sqrt{p_0(\mathbf{a})} + \frac{t}{\sqrt{|R|} \sqrt{u_{i_1}
    \dots u_{i_d}}} \right)^2 \right)
    \right]^{d-1} \\
  &+ \frac{C_{d, \qh} \Upsilon_{\mathbf{i}}}{nt} \left(\sqrt{p_0(\mathbf{a})} + \frac{t}{\sqrt{|R|} \sqrt{u_{i_1}
    \dots u_{i_d}}} \right)^2 \times \\
    &\left[\log \left(e + 
    \frac{2e\Upsilon_{\mathbf{i}}}{t \sqrt{2}}
    \left(\sqrt{p_0(\mathbf{a})} + \frac{t}{\sqrt{|R|} \sqrt{u_{i_1}
    \dots u_{i_d}}} \right)^2 \right)
    \right]^{2(d-1)}
 \end{split}
\end{equation}
We can trivially bound $\Upsilon_{\mathbf{i}}$ by
\begin{align}\label{upsilon}
  \Upsilon_{\mathbf{i}} = |R_{i_1, \dots, i_d}|^{1/(2 \mathfrak{p})}
  \|p_0^{-1} \|^{1/2}_{L_{\mathfrak{q}}(R_{i_1, \dots, i_d})} \leq
  |R|^{1/(2 \mathfrak{p})} 
  \|p_0^{-1} \|^{1/2}_{L_{\mathfrak{q}}(R)} =: \Upsilon. 
\end{align}
which leads to the similar bound to \eqref{firbo} where $\Upsilon_{\mathbf{i}}$ is replaced by $\Upsilon$.
%\begin{equation}\label{firbo}
%  \begin{split}
%  & H_{i_1, \dots, i_d}(t) \\
%  &\leq C_{d, \qh} \sqrt{\frac{t \Upsilon}{n}} \left(\sqrt{p_0(\mathbf{a})} +
%    \frac{t}{\sqrt{|R|} \sqrt{u_{i_1} \dots u_{i_d}}} \right)
%\left[\log \left(e + 
%    \frac{2e\Upsilon}{t \sqrt{2}}
%    \left(\sqrt{p_0(\mathbf{a})} + \frac{t}{\sqrt{|R|} \sqrt{u_{i_1}
%    \dots u_{i_d}}} \right)^2 \right)
%    \right]^{d-1} \\
%  &+ \frac{C_{d, \qh} \Upsilon}{nt} \left(\sqrt{p_0(\mathbf{a})} + \frac{t}{\sqrt{|R|} \sqrt{u_{i_1}
%    \dots u_{i_d}}} \right)^2 \left[\log \left(e + 
%    \frac{2e\Upsilon}{t \sqrt{2}}
%    \left(\sqrt{p_0(\mathbf{a})} + \frac{t}{\sqrt{|R|} \sqrt{u_{i_1}
%    \dots u_{i_d}}} \right)^2 \right)
%    \right]^{2(d-1)}
% \end{split}
%\end{equation}
Observe that when one of the $i_j$'s equals zero, the bound above
becomes infinite (because $u_{0} = 0$). For such cases, we use the
following simpler upper bound (see \eqref{trivialb}): 
\begin{equation}\label{trib}
  \begin{split}
  H_{i_1, \dots, i_d}(t) &\leq 2 P_0(R_{i_1, \dots, i_d}) \\ &\leq 2
  p_0(\mathbf{a}) |R_{i_1, \dots, i_d}| = 2 p_0(\mathbf{a}) |R|
  (u_{i_1+1} - u_{i_1}) \dots (u_{i_d + 1} 
    - u_{i_d}).
  \end{split}  
\end{equation}
Now we fix $\eta \in (0, 1)$ and write
\begin{align*}
  H(t, R) &\leq \sum_{i_1, \dots, i_d \in \{0, 1, \dots, I\}}
         H_{i_1,\dots, i_d}(t) = A(t, \eta) + B(t, \eta)
\end{align*}
where
\begin{align*}
  A(t, \eta) &:= \sum_{i_1, \dots, i_d: u_{i_1} \dots u_{i_d} > \eta}
         H_{i_1,\dots, i_d}(t)\\
         B(t, \eta) &:=
  \sum_{i_1, \dots, i_d: u_{i_1} \dots u_{i_d} \leq \eta} 
         H_{i_1,\dots, i_d}(t).  
\end{align*}
For the terms in $A(t, \eta)$, we shall use the upper bound \eqref{firbo} with $\Upsilon$ to get
\begin{align*}
&  A(t, \eta) \\ &\leq C_{d, \qh} (I+1)^d \sqrt{\frac{t \Upsilon}{n}} \left(\sqrt{p_0(\mathbf{a})} +
    \frac{t}{\sqrt{\eta |R|}} \right) \times \\
&\left[\log \left(e + 
    \frac{2e\Upsilon}{t \sqrt{2}}
    \left(\sqrt{p_0(\mathbf{a})} + \frac{t}{\sqrt{\eta|R|}} \right)^2 \right)
    \right]^{d-1} \\
  &+ (I+1)^d\frac{C_{d, \qh} \Upsilon}{nt} \left(\sqrt{p_0(\mathbf{a})} +
    \frac{t}{\sqrt{\eta |R|}} \right)^2 \times \\
    &\left[\log \left(e + 
    \frac{2e\Upsilon}{t \sqrt{2}}
    \left(\sqrt{p_0(\mathbf{a})} + \frac{t}{\sqrt{\eta|R|}} \right)^2 \right)
    \right]^{2(d-1)}
\end{align*}
where the term $(I+1)^d$ appears because the number of indices $i_1, \dots,
i_d$ with $u_{i_1} \dots u_{i_d} > \eta$ is trivially bounded from
above by the total number of $i_1, \dots, i_d \in \{0, 1, \dots, I\}$
which is $(I+1)^d$. This term can be further bounded by considering
$(\log_2(2/u))^d = (\log_2(2n))^d$. 

For bounding $B(t, \eta)$, we use the trivial bound \eqref{trib} after
further breaking up $B(t, \eta)$ as follows
\begin{align*}
  B(t, \eta) =   \sum_{i_1, \dots, i_d: u_{i_1} \dots u_{i_d} \leq \eta} 
               H_{i_1,\dots, i_d}(t) = C(t, \eta) + D(t, \eta)
\end{align*}
where
\begin{align*}
C(t, \eta) &:= \sum_{i_1, \dots, i_d: i_j = 0 \text{ for some } j} 
               H_{i_1,\dots, i_d}(t) \\
               D(t, \eta) &:=
  \sum_{\substack{i_1, \dots, i_d: 
    i_j \geq 1 \text{ for all } j \\ u_{i_1} \dots u_{i_d} \leq \eta}} 
               H_{i_1,\dots, i_d}(t)
\end{align*}
Note that $u_{i_1} \dots u_{i_d} = 0$ when any $i_j = 0$ which is why
we did not include the clause $u_{i_1} \dots u_{i_d} \leq \eta$ in the
definition of $C(t, \eta)$. Now
\begin{align*}
  C(t, \eta) &= \sum_{i_1, \dots, i_d: i_j = 0 \text{ for some } j} 
               H_{i_1,\dots, i_d}(t) \\ &\leq 2  \sum_{i_1, \dots, i_d:
               i_j = 0 \text{ for some } j} P_0 
    \left(R_{i_1, \dots, i_d} \right) = 2 P_0 (\bigcup_{\substack{i_1, \dots,
    i_d \\ i_j = 0 \text{ for some } j}} R_{i_1, \dots, i_d} ). 
\end{align*}
It is easy to check that the union above equals $R \setminus
\prod_{j=1}^d [a_j + u(b_j - a_j), b_j]$ so that 
\begin{align*}
  C(t, \eta) &\leq 2 P_0 \left(R \setminus
               \prod_{j=1}^d [a_j + u(b_j - a_j), b_j] \right) \\
  &\leq  2 p_0(\mathbf{a}) \left(\text{volume of} ~R \setminus
    \prod_{j=1}^d [a_j + u(b_j - a_j), b_j] \right) \\
&= 2 p_0(\mathbf{a}) |R| \left(1 - (1 - u)^d \right) \leq 2
  p_0(\mathbf{a}) |R| d u = \frac{2 p_0(\mathbf{a}) |R| d}{n}. 
\end{align*}
For $D(t, \eta)$, we have 
$u_{i_j+1} - u_{i_j} = u_{i_j}$ because $i_j \geq 1$ and thus, by
\eqref{trib}, we get 
\begin{align*}
  D(t, \eta) \leq 2 I^d p_0(\mathbf{a}) |R| \eta. 
\end{align*}
where $I^d$ appears because the number of $i_1, \dots, i_d$ with
$\min_j i_j \geq 1$ equals $I^d$.

Putting bounds for $A(t, \eta)$, $C(t, \eta)$ and $D(t, 
\eta)$ together, we obtain
\begin{align*}
&  H(t, R)\\ &\leq C_{d, \qh} (I+1)^d \sqrt{\frac{t \Upsilon}{n}} \left(\sqrt{p_0(\mathbf{a})} +
    \frac{t}{\sqrt{\eta |R|}} \right) \times \\
&\left[\log \left(e + 
    \frac{2e\Upsilon}{t \sqrt{2}}
    \left(\sqrt{p_0(\mathbf{a})} + \frac{t}{\sqrt{\eta|R|}} \right)^2 \right)
    \right]^{d-1} \\
  &+ (I+1)^d\frac{C_{d, \qh} \Upsilon}{nt} \left(\sqrt{p_0(\mathbf{a})} +
    \frac{t}{\sqrt{\eta |R|}} \right)^2 \times \\
    &\left[\log \left(e + 
    \frac{2e\Upsilon}{t \sqrt{2}}
    \left(\sqrt{p_0(\mathbf{a})} + \frac{t}{\sqrt{\eta|R|}} \right)^2 \right)
    \right]^{2(d-1)} \\
    &+ \frac{2 p_0(\mathbf{a}) |R| d}{n} + 2 I^d p_0(\mathbf{a}) |R| \eta. 
\end{align*}
We set
\begin{align*}
  \eta = \frac{t \Upsilon^{1/3}}{n^{1/3}(p_0(\mathbf{a}))^{2/3} |R|} 
\end{align*}
so that
\begin{align*}
  \sqrt{\Upsilon} \left(\sqrt{p_0(\mathbf{a})} + \frac{t}{\sqrt{\eta
  |R|}} \right) &= \left(\mathfrak{M}^{1/2} +
                  \sqrt{t} n^{1/6} \mathfrak{M}^{1/3}
                  \right) \qt{where $\mathfrak{M} := \Upsilon
                  p_0(\mathbf{a})$}
\end{align*}
We check that
\begin{align*}
  \mathfrak{M}^{1/2} = \left(\Upsilon p_0(\mathbf{a}) \right)^{1/2} =
  |R|^{1/(4 \mathfrak{p})}
  \|p_0^{-1}\|_{L_{\mathfrak{q}}(R)}^{1/4}\sqrt{p_0(\mathbf{a})} \leq
  W
\end{align*}
and also $\mathfrak{M}^{1/3} \leq \max\left(\mathfrak{M}^{1/2}, 1
\right) \leq W$. Here $W = W(R, p_0, \qh)$. We thus get
\begin{align*}
  \sqrt{\Upsilon} \left(\sqrt{p_0(\mathbf{a})} + \frac{t}{\sqrt{\eta
  |R|}} \right) \leq W 
    \left(1 + \sqrt{t} n^{1/6} \right). 
\end{align*}
This gives (below we also use $(I + 1)^d \leq C_d (\log n)^d$)
\begin{align*}
  & H(t, R) \\
  &\leq C_{d, \qh} (\log n)^d \sqrt{\frac{t}{n}} W \left(1 +
            \sqrt{t}n^{1/6} \right) \left[ \log \left(e + \frac{2e}{t
            \sqrt{2}} W^2 (1 + \sqrt{t} n^{1/6})^2 \right)
            \right]^{d-1}  \\
  &+ \frac{C_{d, \qh}}{nt} (\log n)^d W^2 \left(1 +
            \sqrt{t}n^{1/6} \right)^2 \left[ \log \left(e + \frac{2e}{t
            \sqrt{2}} W^2 (1 + \sqrt{t} n^{1/6})^2 \right)
    \right]^{2(d-1)} \\
  &+ \frac{2 p_0(\mathbf{a}) |R| d}{n} + 2 (\log n)^d
    \mathfrak{M}^{1/3} \frac{t}{n^{1/3}}.  
\end{align*}
In the last term on the right hand side above, we again use
$\mathfrak{M}^{1/3} \leq W$. In the penultimate term, we use (note
that $p_0(\mathbf{a}) \geq p_0(\mathbf{x})$ for all $\mathbf{x} \in R$)
\begin{align*}
  p_0(\mathbf{a}) |R| &= \frac{(p_0(\mathbf{a}))^2}{p_0(\mathbf{a})}
                        |R| \leq (p_0(\mathbf{a}))^2 \|p_0^{-1}
                        \|_{L_{\mathfrak{q}}(R)} |R|^{1/\mathfrak{p}}
                        = (p_0(\mathbf{a}))^2 \Upsilon^2 =
                        \mathfrak{M}^2 \leq W^4. 
\end{align*}
The bound for $H(t, R)$ then becomes
\begin{align*}
&  H(t, R)\\
&\leq C_{d, \qh} (\log n)^d \sqrt{\frac{t}{n}} W \left(1 +
            \sqrt{t}n^{1/6} \right) \left[ \log \left(e + \frac{2e}{t
            \sqrt{2}} W^2 (1 + \sqrt{t} n^{1/6})^2 \right)
            \right]^{d-1}  \\
  &+ \frac{C_{d, \qh}}{nt} (\log n)^d W^2 \left(1 +
            \sqrt{t}n^{1/6} \right)^2 \left[ \log \left(e + \frac{2e}{t
            \sqrt{2}} W^2 (1 + \sqrt{t} n^{1/6})^2 \right)
    \right]^{2(d-1)} \\
  &+ \frac{2d}{n}W^4 + 2(\log n)^d W \frac{t}{n^{1/3}}. 
\end{align*}

Suppose now that $t \geq n^{-1/3}$. Then because $t^{-1}(1 +
n^{1/6}\sqrt{t})^2$ is decreasing in $t$, we have
\begin{align*}
  t^{-1}(1 + n^{1/6}\sqrt{t})^2 \leq 4 n^{1/3} \qt{for $t \geq
  n^{-1/3}$}. 
\end{align*}
Thus the  log term in the above bound for $H(t, R)$ can
be bounded, for $t \geq n^{-1/3}$ and $n \geq 2$, as: 
\begin{align*}
  \log \left(e + \frac{2e}{t
            \sqrt{2}} W^2 (1 + \sqrt{t} n^{1/6})^2 \right) & \leq \log
  \left(e + 4 \sqrt{2} e W^2  n^{1/3} \right) \\ &\leq C_d (\log n)
  \max(\log W, 1). 
\end{align*}
We thus get
\begin{align*}
  H(t, R) &\leq C_{d, \qh} (\log n)^{2d - 1} \sqrt{\frac{t}{n}} \left(1 +
            n^{1/6}\sqrt{t} \right) W \max \left((\log W)^{d-1}, 1
            \right) \\
  &+ \frac{C_{d, \qh}}{nt} (\log n)^{3d-2} \left(1 + n^{1/6} \sqrt{t}
    \right)^2 W^2 \max\left((\log W)^{2d-2}, 1
            \right) \\ & + \frac{2d}{n}W^4 + 2(\log n)^d W \frac{t}{n^{1/3}}. 
\end{align*}
for $t \geq n^{-1/3}$. This completes the proof of Proposition
\ref{mainhrbound}. 
\end{proof}

We are now ready to prove Theorem \ref{jrec}.

\begin{proof}[Proof of Theorem \ref{jrec}]
We use Theorem \ref{hellexp} along with the bound given by Proposition
\ref{mainhrbound}. Using \eqref{altexpEGt}, \eqref{eledeco} and
\eqref{trivialb}, we can write
\begin{equation*}
  \E G(t) \leq 2 \sum_{i=1}^J H(t, R_i) + 4 P_0 \left[ \left(\cup_{j=1}^J R_j
  \right)^c \right]. 
\end{equation*}
Using Proposition \ref{mainhrbound} for each $R_i$, and the assumed
condition for $P_0 (\cup_{j} R_j)$, we get
\begin{align*}
  \E G(t) &\leq C_{d, \qh}J (\log n)^{2d - 1} \sqrt{\frac{t}{n}} \left(1 +
            n^{1/6}\sqrt{t} \right) W \max \left((\log W)^{d-1}, 1
            \right) \\
  &+ J \frac{C_{d, \qh}}{nt} (\log n)^{3d-2} \left(1 + n^{1/6} \sqrt{t}
    \right)^2 W^2 \max\left((\log W)^{2d-2}, 1
            \right) \\ & + 2J\frac{2d}{n}W^4 + 4J(\log n)^d W
            \frac{t}{n^{1/3}} + 4J^2 \frac{(\log n)^{4d-2}}{n^{2/3}}. 
\end{align*}
for all $t \geq n^{-1/3}$, where $W = \max_{1 \leq j \leq J} W(R_j,
p_0, \qh)$.

We now compare each term on the right hand side above to $t^2/7$. For
the first term,
\begin{equation*}
  C_{d, \qh} J (\log n)^{2d-1} \sqrt{\frac{t}{n}} W \max((\log
  W)^{d-1}, 1) \leq t^2/7 
\end{equation*}
provided
\begin{equation}\label{t1}
t \geq 7^{2/3} C_{d, \qh}^{2/3}
    J^{2/3} \frac{(\log n)^{2(2d-1)/3}}{n^{1/3}} \left(W \max((\log
  W)^{d-1}, 1) \right)^{2/3}.   
\end{equation}
For the second term, 
\begin{equation*}
  C_{d, \qh} J (\log n)^{2d-1} n^{-1/3} t W \max((\log
  W)^{d-1}, 1) \leq t^2/7 
\end{equation*}
provided
\begin{equation}\label{t2}
  t \geq 7 C_{d, \qh} J (\log n)^{2d-1} n^{-1/3} W \max((\log
  W)^{d-1}, 1). 
\end{equation}
For the third term,
\begin{equation*}
J \frac{C_{d, \qh}}{nt} (\log n)^{3d-2} W^2 \max\left((\log W)^{2d-2},
  1  \right) \leq t^2/7
\end{equation*}
provided
\begin{equation}
  \label{t3}
  t \geq 7^{1/3} J^{1/3} C_{d, \qh}^{1/3} \frac{(\log
    n)^{(3d-2)/3}}{n^{1/3}} \left[W^2 \max\left((\log W)^{2d-2},
  1  \right) \right]^{1/3}. 
\end{equation}
For the fourth term, 
\begin{equation*}
  J \frac{C_{d, \qh}}{n^{2/3}} (\log n)^{3d-2} W^2 \max\left((\log W)^{2d-2},
  1  \right) \leq t^2/7
\end{equation*}
provided
\begin{equation}
  \label{t4}
  t \geq 7^{1/2} J^{1/2} C_{d, \qh}^{1/2} \frac{(\log
    n)^{(3d-2)/2}}{n^{1/3}}\left[W^2 \max\left((\log W)^{2d-2},
  1  \right) \right]^{1/2}. 
\end{equation}
For the last three terms, we have
\begin{equation}\label{t5}
  2J \frac{2d}{n} W^4 \leq t^2/7 \text{ provided } t \geq 14^{1/2}
  J^{1/2} W^2 \sqrt{\frac{2d}{n}}, 
\end{equation}
\begin{equation}\label{t6}
  4 J (\log n)^d W \frac{t}{n^{1/3}} \leq t^2/7 \text{ provided } t \geq
  28 J (\log n)^d W n^{-1/3}, 
\end{equation}
\begin{equation}\label{t7}
  4J^2 \frac{(\log n)^{4d-2}}{n^{2/3}} \leq t^2/7 \text{ provided } t
  \geq \sqrt{28} J (\log n)^{2d-1} n^{-1/3}. 
\end{equation}
The lower bounds on $t$ in \eqref{t1}, \eqref{t2},
\eqref{t3}, \eqref{t4},  \eqref{t6}, \eqref{t7} are all of
order \color{black} up to $JW\max\left((\log W)^{d-1},1) \right) n^{-1/3}$ up to logarithmic factors on $n$. \color{black} The largest logarithmic
factor is in \eqref{t2} and \eqref{t7} which is $(\log n)^{2d-1}$. On
the other hand, the lower bound in \eqref{t5} is of the order \color{black} up to
$J^{1/2} W^{2} n^{-1/2}$.\color{black} Combining these, it is clear that $\E G(t) \leq t^2$ 
provided
\begin{equation*}
  t \geq t_0 := \max \left(C_{d, \qh} J (\log n)^{2d-1} n^{-\frac{1}{3}} W
    \max((\log W)^{d-1}, 1), 7^{\frac{1}{2}} \color{black}J^{1/2} \color{black} W^2 \sqrt{\frac{2d}{n}}
  \right).  
\end{equation*}
Theorem \ref{jrec} then follows from Theorem \ref{hellexp}. 
\end{proof}

\subsubsection{Proof of Theorem \ref{adaptation}}

The proof of Theorem \ref{adaptation} is based on the following result.
%\begin{lemma}\label{adaptiveres}
%Assume the same setting in Theorem \ref{adaptation} and let
%\[
%A_j(t) := \sup_{p \in \mathcal{P}, h(p,p_0) \leq t} \int_{R_j} \frac{p_0-p}{p+p_0}d(P_0-P_n)
%\]
%Then for each $j \in \{1, \ldots, m\}$, we have
%\begin{align}
%  \mathbb{E} A_j(t) &\leq C_{d,B} \left( \frac{t}{\sqrt{n}} p_j|R_j| + \left(\log_2 \frac{1}{n^{-1/2}t} \right)^{d} \left( \frac{t}{\sqrt{n}}+
%  \frac{1}{n}   \right) \right)\label{prove1}
%  %&\leq C_{d,B,\kappa}\left(\log_2 \frac{1}{n^{-1/2}t} \right)^{d} \left( \frac{t}{\sqrt{n}}+
%  %\frac{1}{n}   \right)  
%\end{align}
%where $C_{d}$ is a constant depending only on $d$ but not $n$.
%\end{lemma}
%\color{black}

\begin{proposition}\label{adaptiveres}
Suppose $R \subseteq [0,\infty)^d$ is the rectangle given by $R = [a_1,b_1] \times \ldots \times [a_d, b_d]$ and let $p_0$ take a constant value in the interior of $R$. Then $H(t,R)$ (for the definition, see \eqref{secupbo}) satisfies 
%\[
%H(t, R) := \mathbb{E}\sup_{p \in \smud: h(p.p_0) \leq t} \int \frac{p_0-p}{p_0+p} \one (R) d(P_0-P_n),
%\]
%then we have 
\begin{align*}
 H(t, R) &\leq C_{d} \left[\left(\log \frac{1}{n^{-1/2}t}
 \right)^{d}(n^{-1/2}t + n^{-3/8} t^{5/4}) \left( \log
     n\right)^{d-1} \right. \\ &+ \left. \left( n^{-3/4}t^{1/2} + \frac{1}{n} \right)
   \left( \log n\right)^{2d-2}  \right] 
\end{align*}
where $C_{d}$ is a constant depending on $d$ but not $n$.
\end{proposition}
\begin{remark}\label{remark_adap} Proposition \ref{adaptiveres} is
  different from Proposition \ref{mainhrbound} since it assumes that
  $p_0$ is a constant on $R$. This stronger assumption leads to a
  better bound in the sense that the main term $n^{-1/2}t$ in
  Proposition \ref{adaptiveres} is smaller than $n^{-1/3}t$ in
  Proposition \ref{mainhrbound}.     
\end{remark}

\begin{proof}[Proof of Proposition \ref{adaptiveres}]
Suppose $p_0(x) = B$ when $x$ is in the interior of $R$.  Without loss of generality, we assume $|R| \geq n^{-1/2}t B^{-1}$ since otherwise we can bound $H(t, R) \leq 2 B |R| \leq  2 n^{-1/2} t$. Also, we suppose $t \leq 2$ since the supremum over the set $h(p,p_0) \leq t$ does not change whenever $t \geq 2$. 
Also note that $C,C_u,C_l$ can represent different constants in different lines.

Fix  $u :=n^{-1/2}t$ and let $I:= \log_2(1/(2u))$. Let $u_0 = 0$ and $u_i = 2^{i-1}u$ for $i=1, \ldots, I+1$. 
Consider rectangles for $s_j \in \{i_j, \tilde i_j\}$
\begin{align*}
R_{s_1,\ldots,s_d} := \prod_{j=1}^d I_{s_j}
\end{align*}
where $I_{i_j} = [a_j+u_{i_j}(b_j-a_j), a_j+u_{i_j+1}(b_j-a_j)]$ and $I_{\tilde i_j} = [b_j-u_{i_j+1}(b_j-a_j), b_j-u_{i_j}(b_j-a_j)]$ for $0 \leq i_j \leq I$ and $j=1, \ldots, d$.

%\begin{align*}
%R_{s_1,\ldots,s_d} := \prod_{j=1}^d [a_j+u_{i_j}(b_j-a_j), a_j+u_{i_j+1}(b_j-a_j)]
%\end{align*}
%for $0\leq i_j \leq I$ and $j=1, \ldots, d$.
%Similarly, we let for $1\leq \tilde i_j \leq I$,
%\begin{align*}
%R_{\tilde i_1,\ldots,\tilde i_d} := \prod_{j=1}^d [b_j-u_{i_j}(b_j-a_j), b_j-u_{i_j-1}(b_j-a_j)]
%\end{align*}

All together, there are $(2(I+1))^d$ rectangles $R_{s_1, \ldots, s_d}$ as each $s_j \in \{i_j, \tilde i_j\}$ ranges over $0,1,\ldots, I$ for $j=1, \ldots, d$. These rectangles have disjoint interiors. Similar to the proof of Proposition \ref{mainhrbound}, we have 
\begin{align*}
H(t, R) &:= \mathbb{E}\sup_{p \in \smud: h(p, p_0) \leq t} \int
          \frac{p_0-p}{p_0+p} \one(R) d(P_0-P_n) \\ &\leq \sum_{i_j,
          \tilde i_j \in \{0,1,\ldots, I\}}\sum_{s_1, \ldots, s_d \in
          \{i_j, \tilde i_j\}}  H_{s_1, \ldots, s_d} (t) 
\end{align*}
where
\[
H_{s_1, \ldots, s_d} (t) := \mathbb{E} \sup_{p \in \smud: h(p,p_0) \leq t} \int \frac{p_0-p}{p_0+p} \one(R_{s_1, \ldots, s_d}) d(P_0-P_n). 
\]
We now apply Lemma \ref{ESupResult} to bound the above. 

Without loss of generality, for the subset $\mathcal{H} \subset \{1,\ldots,d\}$, we consider $s_j = i_j$ for $j \in \mathcal{H}$ and $s_k = \tilde i_k$ for $k \in \{1,\ldots,d\} \setminus \mathcal{H}$. 
%suppose we have $i_1, \ldots, i_\ell$ and $\tilde i_{\ell+1}, \ldots, \tilde i_d$ in $R_{s_1,\ldots, s_d}$. 
That is, we let 
\begin{align*} 
&R_{s_1, \ldots, s_d} = \prod_{j \in \mathcal{H}} [a_{j} +u_{i_j} (b_j-a_j), a_{j} +u_{i_j+1}(b_j-a_j)] \times\\
&\prod_{k \notin \mathcal{H}} [b_{k}-u_{ i_{k+1}}(b_{k}-a_{k}),  b_{k}- u_{ i_{k}}(b_{k}-a_{k}) ]
%\times \ldots \times [a_{\ell}+u_{i_\ell} (b_\ell-a_\ell) ,a_{\ell}+ u_{i_\ell+1}(b_\ell-a_\ell)] \times \\
%&[b_{\ell+1}-u_{ i_\ell}(b_{\ell+1}-a_{\ell+1}),  b_{\ell+1}- u_{ i_\ell-1}(b_{\ell+1}-a_{\ell+1}) ] \times \ldots \times [b_{d}-u_{ i_d}(b_{d}-a_{d}),  b_{d}-u_{ i_d-1}(b_d-a_d)].
\end{align*}
Note that 
\begin{equation}\label{volR}
|R_{s_1, \ldots, s_d}|= u_{i_1} \ldots u_{i_d} |R| =  2^{\sum_{j=1}^d i_j-d} u^d |R|.
\end{equation}

Note that by Lemma \ref{bds},
\begin{align}\label{betabound}
\beta &= U(R_{s_1, \ldots, s_d}) = \sup_{p \in \smud: h(p, p_0) \leq t}
                           \sup_{x \in R_{s_1, \dots s_d}} p(x)  \nonumber \\
%  &\leq  \sup_{p \in \smud: h(p, p_0) \leq t} p \left(a_1 + u_{i_1}(b_1 -
%    a_1), \dots, a_\ell + u_{i_\ell}(b_\ell -
%    a_\ell), b_{\ell+1}-u_{ i_{\ell+1}} 
%    (b_{\ell+1}-a_{\ell+1}), \dots b_{d}-u_{ i_d-1} (b_{d}-a_{d})) \right) \\
 &\leq \left(\sqrt{B} + \frac{t}{\sqrt{\prod_{j \in \mathcal{H}} u_{i_j}(b_j-a_j) \prod_{j\notin \mathcal{H}} (1-u_{ i_j})(b_j-a_j)}} \right)^2\nonumber \\
&\leq \left(\sqrt{B} + C_u \frac{t}{ \sqrt{u_{i_1} \dots u_{i_\ell} u_{ i_{\ell+1}} \dots u_{ i_d}} \sqrt{|R|} } \right)^2 
%&= \left(\sqrt{B} + C_u \frac{t}{|R_{s_1, \ldots, s_d }|^{1/2}} \right)^2,
\end{align}
where $C_u$ is some constant depending on $d$ and the penultimate inequality holds because $1- u_{ i_j} \geq 1/2 \geq u_{ i_j}$. Observe that $u_{i_j}$ can be equal $0$ (when $i_j=0$) in which the right hand side is $+\infty$. 

Again by Lemma \ref{bds}  and using $1-u_{i_j} \geq 1/2 \geq u_{i_j}$,
\begin{align}\label{alphabound}
\alpha &= L(R_{s_1,\ldots, s_d}) = \inf_{p \in \smud: h(p,p_0) \leq t} \inf_{x \in R_{s_1,\dots,s_d}} p(x) \nonumber  \\
%&\geq \inf_{p \in \smud: h(p,p_0) \leq t} p \left(a_1 + u_{i_1+1}(b_1 -
%    a_1), \dots, a_\ell + u_{i_\ell+1}(b_\ell -
%    a_\ell), b_{\ell+1}-u_{ i_{\ell-1}} 
%    (b_{\ell+1}-a_{\ell+1}), \dots b_{d}-u_{ i_d} (b_{d}-a_{d})) \right) \\
&\geq \left(\sqrt{B} -  \frac{t}{\sqrt{\prod_{j \in \mathcal{H}} (1-u_{i_j})(b_j-a_j)\prod_{j \notin \mathcal{H}} u_{i_j} (b_j-a_j)}})  \right)_+^2\nonumber \\
&\geq \left(\sqrt{B} - C_l\frac{t}{ \sqrt{u_{i_1} \dots u_{i_\ell} u_{ i_{\ell+1}} \dots u_{ i_d}} \sqrt{|R|} }  \right)_+^2
%&= \left(\sqrt{B} -  C_l\frac{t}{|R_{s_1,  \ldots, s_d }|^{1/2}} \right)_+^2,
\end{align}
where $C_l$ is some constant depending on $d$.
Note that the final bound for $\alpha$ and $\beta$ does not depend on
the choice of $\mathcal{H}$. Thus without loss of generality, we just
let $s_j = i_j$ for all $j \in \{1,\ldots, d\}$. Thus  
\begin{align*}
H(t, R) &:= \mathbb{E} \sup_{p \in \smud: h(p, p_0) \leq t} \int
          \frac{p_0-p}{p_0+p} \one(R) d(P_0-P_n)  \\ &\leq 2^d \sum_{i_j  \in \{0,1,\ldots, I\}} H_{i_1, \ldots, i_d} (t)
\end{align*}
where
\[
 H_{i_1, \ldots, i_d} (t) := \mathbb{E} \sup_{p \in \smud: h(p,p_0) \leq t} \int \frac{p_0-p}{p_0+p} \one(R_{i_1, \ldots, i_d}) d(P_0-P_n). 
\]

By \eqref{alphabetaclaim}, the bound on $ H_{i_1,\ldots, i_d}$ is given by
\begin{align}\label{firbo2}
  \begin{split}
& H_{i_1,\ldots, i_d}\\ &\leq C_{d, \qh} \sqrt{\frac{t}{n}} \sqrt{\beta - \alpha}
      B^{-\frac{1}{4}} |R_{i_1,\ldots,i_d}|^{\frac{1}{4}} \left[\log
    \left(e + \frac{2e (\beta - \alpha) B^{-\frac{1}{2}} |R_{i_1,\ldots,i_d}|^{\frac{1}{2}}}{t \sqrt{2}} 
               \right)\right]^{d-1} \\ &+ 
  \frac{C_{d, \qh}}{nt} (\beta - \alpha) B^{-\frac{1}{2}} |R_{i_1,\ldots,i_d}|^{\frac{1}{2}}
                                         \left[\log \left(e + \frac{2
                                         e (\beta 
        - \alpha) B^{-\frac{1}{2}} |R_{i_1,\ldots,i_d}|^{\frac{1}{2}}}{t
                                         \sqrt{2}} 
    \right) \right]^{2(d-1)} 
    \end{split}
\end{align}
since 
\begin{align*}
|R_{i_1,\ldots,i_d}|^{1/(4\ph)} \|p_0^{-1}\|^{1/4}_{L_\qh(R)} &= |R_{i_1,\ldots,i_d}|^{1/(4\ph)} \left(\int_{R_{i_1,\ldots,i_d}} p_0^{-\qh}\right)^{1/(4\qh)} \\
&= |R_{i_1,\ldots,i_d}|^{1/(4\ph)} B^{-1/4} |R_{i_1,\ldots,i_d}|^{1/(4\qh)}\\
&= B^{-1/4} |R_{i_1,\ldots,i_d}|^{1/4}
\end{align*}
where the penultimate equality follows since $p_0$ is a constant $B$ on the interior of $R$.

Observe that when one of the $i_j$'s equals zero, the bound
\eqref{firbo2} above becomes infinite since $\beta = \infty$ and
$u_0=0$. For such cases, we use the following simpler upper bound 
\begin{equation}\label{trib2}
  \begin{split}
    H_{i_1,\ldots,i_d}(t) &\leq 2P_0(R_{i_1,\ldots,i_d}) \\ &\leq 2B
    |R_{i_1,\ldots, i_d}| = 2B|R|(u_{i_1+1}-u_{i1}) \ldots
    (u_{i_d+1}-u_{i_d}).
  \end{split}    
\end{equation}
In fact, we use such bound for the case where $|R_{i_1,\ldots,i_d}| = u_{i_1} \ldots u_{i_d} |R| \leq n^{-1/2}B^{-1} t$ as well as the case where one of the $i_j's$ equals zero.

Assume there are no $i_j's$ which equal zeroes.
From equations \eqref{betabound} and \eqref{alphabound}, we consider two different cases (i) $|R_{s_1,\ldots, s_d}|^{1/2} \leq C_l t B^{-1/2}$ (that is, $u_1 \ldots u_d \leq C_l^2 t^2/(B |R|)$) and (ii) $|R_{s_1,\ldots, s_d}|^{1/2} > C_l t B^{-1/2}$. For the first case (i), $\alpha=0$ so that
\begin{align}\label{bound_u1}
(\beta-\alpha) &|R_{s_1, \ldots, s_d}|^{1/2} \nonumber \\
&\leq C\left(\sqrt{B} + C_u \frac{t}{ \sqrt{u_{i_1} \dots u_{i_\ell} u_{ i_{\ell+1}} \dots u_{ i_d}} \sqrt{|R|} } \right)^2 |R_{s_1,\ldots, s_d}|^{1/2} \nonumber\\
&\leq C\left( B + \frac{t^2}{|R_{s_1, \ldots, s_d}|}\right) |R_{s_1, \ldots, s_d}|^{1/2} \nonumber\\
&\leq C \left(B^{1/2}t + \frac{t^2}{|R_{s_1, \ldots, s_d}|^{1/2}} \right).
%&\leq C  \left( t p_j^{1/2} + n^{1/4} t^{3/2}p_j^{1/2} \right),
\end{align}
%where the penultimate inequality follows since $p_j = p_j^{1/2} p_j^{1/2} \leq C_l p_j^{1/2}  \frac{t}{|R_{s_1,\ldots,d_s}|^{1/2}}$ and the last inequality follows by \eqref{secondregion}.

For the second case (ii), we have 
\[
|\beta-\alpha| \leq C\frac{B^{1/2} t}{|R_{s_1,\ldots, s_d}|^{1/2}} + C'\frac{t^2}{|R_{s_1,\ldots, s_d}|},
\]
hence
\begin{align}\label{bound_u2}
(\beta-\alpha) |R_{s_1, \ldots, s_d}|^{1/2} &\leq C\left(\frac{B^{1/2}t}{|R_{s_1,\ldots, s_d}|^{1/2}} + \frac{t^2}{|R_{s_1,\ldots, s_d}|}\right)|R_{s_1, \ldots, s_d}|^{1/2} \nonumber \\
%&\leq C \left(\frac{B^{1/2}t }{|R_{s_1, \ldots, s_d}|^{1/2}} + \frac{t^2}{|R_{s_1, \ldots, s_d}|} \right)|R_{s_1, \ldots, s_d}|^{1/2} \\
&\leq C \left(B^{1/2}t + \frac{t^2}{|R_{s_1,\ldots, s_d}|^{1/2}} \right) \leq C B^{1/2} t.
\end{align}

%This implies that we need to focus on the first case only.

Now we fix 
\[
\eta = \frac{t}{n^{1/2} B |R|}
\]
and write
\[
 H(t,R) \leq 2^d\sum_{i_j \in \{0,1,\ldots, I\}} H_{i_1, \ldots, i_d} (t) =2^d \left( A(t,\eta) +  B(t,\eta) +  C(t,\eta) +  D(t,\eta)\right).
\]
where 
\begin{align*}
  A(t, \eta) &:= \sum_{\substack{i_1, \dots, i_d: 
    u_{i_1} \dots u_{i_d} \geq \eta \\ u_{i_1} \dots u_{i_d} \leq \frac{C_l^2t^2}{B|R|}}} 
               H_{i_1,\dots, i_d}(t)  \\
        B(t, \eta) &:=
  \sum_{\substack{i_1, \dots, i_d: 
    u_{i_1} \dots u_{i_d} \geq \eta \\ u_{i_1} \dots u_{i_d} \geq \frac{C_l^2t^2}{B|R|}}} 
               H_{i_1,\dots, i_d}(t)
\end{align*}
and
\begin{align*}
%A(t,\eta):= \sum_{i_j,\ldots,i_d: u_{i_1} \ldots u_{i_d}>\eta} \tilde H_{i_1, \ldots, i_d} (t) 
  C(t, \eta) &:= \sum_{i_1, \dots, i_d: i_j=0 \text{ for some $j$}} 
       \tilde  H_{i_1,\dots, i_d}(t)  \\
       D(t, \eta) &:= \sum_{\substack{i_1, \dots, i_d: 
    i_j \geq 1 \text{ for all } j \\ u_{i_1} \dots u_{i_d} \leq \eta}} 
               H_{i_1,\dots, i_d}(t)
\end{align*}
From the above bounds \eqref{bound_u1} and \eqref{bound_u2}, we have
\[
 B(t,\eta) \leq  A(t,\eta). 
\]
Then it suffices to bound $ A(t,\eta)$. Since we assume $u_{i_1} \ldots u_{i_d} \geq \frac{t}{n^{1/2} B|R|}$, we can further bound \eqref{bound_u1} as follows
\[
(\beta-\alpha)|R_{s_1,\ldots, s_d}|^{1/2} \leq C \left(B^{1/2}t + B^{1/2} n^{1/4}t^{3/2} \right).
\]
Plugging the above in \eqref{firbo2}, we have
\begin{align*}
    \begin{split}
 A(t,\eta)&\leq C_{d,\qh} (I+1)^d \sqrt{\frac{t}{n}} (t^{1/2}+n^{1/8}t^{3/4}) \left[\log
    \left(e + \frac{2e (t+n^{1/4}t^{3/2})}{t \sqrt{2}} 
               \right)\right]^{d-1} \\ &+ 
 (I+1)^d \frac{C_{d, \qh}}{nt} (t+n^{1/4}t^{3/2})
                                         \left[\log \left(e + \frac{2
                                         e (t+n^{1/4}t^{3/2})}{t
                                         \sqrt{2}} 
    \right) \right]^{2(d-1)} \\ &\leq
    C_{d,\qh} (I+1)^d \left( n^{-1/2}t + n^{-3/8}t^{5/4} \right) \left[\log
    \left(e + \frac{2e (1+n^{1/4}t^{1/2})}{\sqrt{2}} 
               \right)\right]^{d-1} \\ &+ 
 C_{d, \qh}(I+1)^d (n^{-1}+n^{-3/4}t^{1/2})
                                         \left[\log \left(e + \frac{2
                                         e (1+n^{1/4}t^{1/2})}{
                                         \sqrt{2}} 
    \right) \right]^{2(d-1)}.
    \end{split} 
\end{align*}

Using the same idea in the proof of Proposition \ref{mainhrbound}, we have 
\[
 C(t,\eta) \leq 2B|R|d n^{-1/2} t \leq 2dn^{-1/2}t
\]
since $p_0$ is a density so that $1 = \int p_0  \geq \int_R p_0 = B|R|$,
and
\[
 D(t,\eta) \leq 2(I+1)^d n^{-1/2} t.
\]

Finally, since 
\[
I \leq \log_2 \frac{1}{n^{-1/2}t},
\] combining these four terms $ A(t,\eta),  B(t,\eta),  C(t,\eta)$ and $ D(t,\eta)$, the claim is proved.
\end{proof}
\color{black}

Now we are ready to prove Theorem \ref{adaptation}.

\begin{proof}[Proof of Theorem \ref{adaptation}]
Without loss of generality, we let $m \leq n$. Otherwise, there is nothing to prove. The main task is to bound
\[
\mathbb{E}G(t) = 2\mathbb{E} \sup_{p \in \mathcal{P}_{\text{SMU}(d)}: h(p,p_0) \leq t} \int \frac{p_0-p}{p_0+p} d(P_0-P_n).
\]
The strategy for controlling the above will be different from that of the proof of Theorem 4.2 in the main paper. Let $\mathcal{L}$ denote the class of all vectors $\ell := (\ell_1, \ldots, \ell_m)$ where each $\ell_j$ is an integer with $1 \leq \ell_j \leq m$ and such that $\sum_{j=1}^m \ell_j \leq 2m$. Because the number of $m$-tuples of positive integers whose sum is equal to $p$ equals ${p-m \choose m-1}$, it is easy to see that $\mathcal{L}$ is a finite set whose cardinality $|\mathcal{L}|$ is bounded as 
\begin{align*}
|\mathcal{L}| &\leq \sum_{p=m}^{2m} {p-1 \choose m-1} = \sum_{q = m-1}^{2m-1} {q \choose m-1} \\
&=\sum_{q=m-1}^{2m-1} {q \choose q-(m-1)} \leq \sum_{q=m-1}^{2m-1} {2m-1 \choose q-(m-1)} \leq 2^{2m-1} \leq 4^m.
\end{align*}
Now for each $\ell \in \mathcal{L}$, let
\[
\mathcal{P}(\ell) := \left\{p \in \smud: h(p,p_0) \leq t, \int_{R_j} (\sqrt{p}-\sqrt{p_0})^2 \leq \frac{\ell_j t^2}{m} \text{$\forall$ $j=1, \ldots, m$} \right\}.
\]
We then claim that 
\begin{align}\label{claimsubset}
\{p \in \smud: h(p,p_0) \leq t\} \subseteq \bigcup_{ \ell \in \mathcal{L}} \mathcal{P}(\ell).
\end{align}

To prove \eqref{claimsubset}, take $p \in \smud$ with $h(p, p_0) \leq
t$. For each $j = 1, \dots, m$, let $\ell_j$ be the smallest positive
integer such that
\begin{align*}
  \int_{R_j} \left(\sqrt{p} - \sqrt{p}_0 \right)^2 \le
  \frac{\ell_jt^2}{m}. 
\end{align*}
Because $\ell_j$  is the smallest positive integer satisfying this, we
would have
\begin{align*}
\frac{(\ell_j - 1) t^2}{m} \leq  \int_{R_j} \left(\sqrt{p} -
  \sqrt{p}_0 \right)^2 \le   \frac{\ell_jt^2}{m} 
\end{align*}
which implies that
\begin{align*}
  \sum_{j=1}^m \frac{(\ell_j - 1) t^2}{m} \leq \sum_{j=1}^m \int_{R_j} \left(\sqrt{p} -
  \sqrt{p}_0 \right)^2 \leq \int \left( \sqrt{p} - \sqrt{p_0} \right)^2
  \leq t^2
\end{align*}
or equivalently $\sum_{j=1}^m \ell_j \leq 2m$. Thus $(\ell_1, \dots,
\ell_m) \in \lin$ which proves \eqref{claimsubset}. With this, we
control $\mathbb{E}G(t)$ as
\begin{align*}
 \mathbb{E} G(t) &=  2\E \sup_{p \in \smud : h(p, p_0) \leq t} \int \frac{p_0 - p}{p_0 + p}
         d(P_0 - P_n) \\ &\leq 2\E \max_{\ell \in \lin} \sup_{p \in
                           \Ps(\ell)} \int \frac{p_0 - p}{p_0 + p} 
         d(P_0 - P_n). 
\end{align*}

By Lemma \ref{intchan} (applied with $a = 1$), we obtain 
\begin{align}\label{boundrhs}
 \mathbb{E} G(t) \leq 4 \max_{\ell \in \lin}  \E \sup_{p \in \Ps(\ell)} \int \frac{p_0 - p}{p_0 + p}
         d(P_0 - P_n)  + 8t\sqrt{\frac{\log(e|\lin|)}{n}} +\frac{28}{3}\frac{\log (e |\lin|)}{n}.
\end{align}
Note that $\log (e
|\lin|)$ is of the order $m$. \color{black} The second term is of order $t (m/n)^{1/2}$ and the third term is of order $m/n$. To bound the first term, \color{black} we shall, as
before, split the integral as the sum over $R_j$ for $j = 1, \dots,
m$:
\begin{align*}
  \E \sup_{p \in \Ps(\ell)} \int \frac{p_0 - p}{p_0 + p}
         d(P_0 - P_n) &\leq \E \sup_{p \in \Ps(\ell)} \sum_{j=1}^m \int_{R_j} \frac{p_0 - p}{p_0 + p}
                        d(P_0 - P_n) \\
  &\leq \sum_{j=1}^m \E \sup_{p \in \Ps(\ell)} \int_{R_j} \frac{p_0 - p}{p_0 + p}
                        d(P_0 - P_n).
\end{align*}
Note now that the supremum inside the sum is over $\Ps(\ell)$ which
means that we have the additional condition
\begin{align}\label{additionalcond}
  \int_{R_j} \left(\sqrt{p} - \sqrt{p_0} \right)^2 \leq \frac{\ell_j
  t^2}{m}. 
\end{align}

Fix $j \in \{1,\ldots, m\}$, then by Lemma \ref{adaptiveres} with the additional condition \eqref{additionalcond} so that $t_j = t \sqrt{\frac{\ell_j}{m}}$, we have the following series of bounds:
\begin{align}
\log_2 \left( \frac{1}{n^{-1/2}t_j}\right) &= \log_2\left(\frac{m^{1/2}}{n^{-1/2} \ell_j^{1/2} t  }\right) \leq \log_2 \left(\frac{1}{n^{-1}t}\right) \forall j \label{bound1} \\
\sum_{j=1}^m n^{-1/2}t_j &= n^{-1/2} \sum_{j=1}^m \left( \frac{\ell_j}{m}\right)^{1/2}t \leq 2^{1/2} \left(\frac{m}{n}\right)^{1/2} t \label{bound2} \\
\sum_{j=1}^m n^{-3/8} t_j^{5/4} &= n^{-3/8} \sum_{j=1}^m (\frac{\ell_j}{m})^{5/8} t^{5/4} \leq 2^{5/8}  \left(\frac{m}{n}\right)^{3/8} t^{5/4}\label{bound3}\\
%\log \frac{1}{t_j} &= \log \frac{\sqrt{m}}{\sqrt{\ell_j} t} \leq \log \left(\frac{1}{n^{-1/2}t} \right) \forall j\label{bound4}\\
\sum_{j=1}^m n^{-3/4}t_j^{1/2} &= n^{-3/4} t^{1/2} \sum_{j=1}^m \left(\frac{\ell_j}{m}\right)^{1/4} \leq 2^{1/4} \left(\frac{m}{n}\right)^{3/4} t^{1/2}\label{bound5}\\
\sum_{j=1}^m \frac{1}{n} &= \frac{m}{n} \label{bound6}
\end{align}
where \eqref{bound2}, \eqref{bound3} and \eqref{bound5} follows since by H\"older's inequality, for every $p>1$, we have
\[
\sum_{j=1}^m \ell_j^{1/p} \leq \left(\sum_{j=1}^m \ell_j \right)^{1/p} m^{1/q}
\]
where $q:=\frac{p}{p-1}$. This, and the fact that $\sum_{j=1}^m \ell_j \leq 2m$, allow us to deduce
\[
\sum_{j=1}^m \ell_j^{1/p} \leq 2^{1/p} m^{1/p} m^{1/q} = 2^{1/p}m.
\]
Combining above series of bounds \eqref{bound1}-\eqref{bound6}, we
have 
\begin{align*}
& \sum_{j=1}^m \mathbb{E} \sup_{p \in \mathcal{P}(\ell)} \int_{R_j}
  \frac{p_0-p}{p_0+p} d(P_0-P_n) \\ &\leq  
C_{d}\left\{ \left(\log \frac{1}{n^{-1}t}\right)^{2d-1} \left(\left(\frac{m}{n}\right)^{1/2}t  + \left(\frac{m}{n} \right)^{3/8}t^{5/4} \right)+\right. \\
&+ \left. \left(\log \frac{1}{n^{-1}t}\right)^{3d-2} \left(\left(\frac{m}{n} \right)^{3/4}t^{1/2} + \frac{m}{n} \right) \right\}.
\end{align*}
\color{black} Combining all three terms in \eqref{boundrhs}, 
\color{black}
we thus obtain
\begin{align*}
\mathbb{E}G(t) &\leq \bar G(t) := C_{d}\left\{ \left(\log
                                    \frac{1}{n^{-1}t}\right)^{2d-1}
                                    \left(\left(\frac{m}{n}\right)^{1/2}t
                                    + \left(\frac{m}{n}
                                    \right)^{3/8}t^{5/4}
                                    \right) \right. \\ & \left.+\left(\log
                                    \frac{1}{n^{-1}t}\right)^{3d-2}\left(\left(\frac{m}{n}
                                    \right)^{3/4}t^{1/2} + \frac{m}{n}
                                    \right)\right\}. 
\end{align*}
Note that 
\begin{align*}
& \bar G(m^{1/2}n^{-1/2} (\log n)^\alpha) \\ &\leq C_{d,B} (\frac{m}{n}) \left(
                                          (\log n)^{2d-1} \left((\log
                                          n)^\alpha + (\log
                                          n)^{\frac{5\alpha}{4}} \right)+
                                          (\log n)^{3d-2} \left((\log
                                          n)^{\frac{\alpha}{2}} +
                                          1\right)\right) \\ 
&\leq (\frac{m}{n}) \left((\log n)^{2d-1+5\alpha/4} + (\log n)^{3d-2+\alpha/2} \right)
\end{align*}
thus we take $\alpha=(4/3)(2d-1)$. Also $\bar G(t)/t^{5/4}$ is non-increasing. Thus the equations 
 \eqref{hellexp.eq} and \eqref{hellexp.exp} hold with $t_0 =m^{1/2}n^{-1/2}(\log n)^{(4/3)(2d-1)}$ and $\eta=3/4$.
\end{proof}

\subsection{Proof of Theorem \ref{lowerbound}}
In the proof of Theorem \ref{lowerbound}, we use Legendre polynomials
and their properties. Let us first recall basic definitions and
properties of Legendre polynomials (for proofs of these facts and more
details, see e.g. \cite{lima}).  
\begin{definition}[Legendre and Shifted Legendre Polynomials] For $u \in [-1,1]$, the Legendre
  Polynomial of order $\ell$ is defined to be 
\begin{equation}\label{legendre_eq}
\tilde\lf_\ell(u) = \frac{1}{2^\ell} \sum_{k=0}^{\lfloor \ell/2 \rfloor} (-1)^k {\ell \choose k} {2\ell-2k \choose \ell} u^{\ell-2k}.
\end{equation}

For $u \in [0,1]$,
  the shifted Legendre Polynomial of order $\ell$ is defined as  
\[
\lf_\ell(u) = \tilde \lf_\ell(2u-1).
\]
\end{definition}
The first few shifted Legendre polynomials are $\lf_0(u)=1$, $\lf_1(u) = 2u-1$, $\lf_2(u) = 6u^2-6u+1$, and $\lf_3 (u) = 20u^3-30u^2+12u-1$.

\begin{lemma}[Orthogonal property]\label{orthogonal_prop} The polynomials $\tilde \lf_\ell(u)$ and $\lf_\ell(u)$ are  orthogonal over $[-1,1]$ and $[0,1]$ respectively.
\begin{equation}\label{orthogonal1}
\int_{-1}^1 \tilde \lf_\ell(u) \tilde \lf_{\ell'}(u) du = \frac{2}{2\ell+1}\one\{ \ell \neq \ell' \}
\end{equation}
\begin{equation}\label{orthogonal2}
\int_{0}^1 \lf_\ell(u) \lf_{\ell'}(u) du = \frac{1}{2\ell+1} \one\{ \ell \neq \ell' \}.
\end{equation}
\end{lemma}

\begin{lemma}[Recurrence relation]\label{recurrence_prop}
For $u \in [-1,1]$ 
\begin{equation}\label{recurrence1}
\tilde \lf_{\ell+1}(u) = \frac{2\ell+1}{\ell+1} \tilde \lf_{\ell}(u) - \frac{\ell}{\ell+1} \tilde \lf_{\ell-1}(u).
\end{equation}
For $u \in [0,1]$,
\begin{equation}\label{recurrence2}
 \lf_{\ell+1}(u) = \frac{2\ell+1}{\ell+1}  \lf_{\ell}(u) - \frac{\ell}{\ell+1}  \lf_{\ell-1}(u).
\end{equation}
\end{lemma}

\begin{lemma}[Integration of Legendre Polynomials]
\begin{equation}\label{integration1}
\int {\tilde \lf}_\ell(u) du = \frac{{\tilde \lf}_{\ell+1}(u)-{\tilde \lf}_{\ell-1}(u)}{2\ell+1}+C
\end{equation}
\begin{equation}\label{integration2}
\int { \lf}_\ell(u) du = \frac{{ \lf}_{\ell+1}(u)-{ \lf}_{\ell-1}(u)}{2(2\ell+1)}+C
\end{equation}
\end{lemma}

The following Lemma \ref{legendrelem} contains useful properties to
prove Theorem \ref{lowerbound}.  \color{black}

\begin{lemma}\label{legendrelem}
Let $\lf_\ell$ be the shifted Legendre polynomials of order $\ell$
defined on $[0,1]$.  Consider
$\lf_2(2^m u-i)$ the location scale family of Legendre Polynomials for
$i=0, \ldots, 2^m-1$. We define  
\begin{align}
s_{m,i}(u) &:= \lf_2(2^m u-i)  \label{def_s} \\
A_{m,i}(x) &= \int_x^{(i+1)2^{-m}} s_{m,i}(u)du. \label{def_a}
\end{align}
Then
\begin{enumerate}
\item $\int s_{m,i}(u)du=0$ and $\int u s_{m,i} (u) du=0$ for $i=0, \ldots, 2^m-1$.
\item $\int A_{m,i}(x) dx = 0$,  $\int A_{m,i}(x)^2 dx = \frac{1}{210} 2^{-3m}$  for $i=0, \ldots, 2^m-1$, and for $i \neq j$, we have $\int A_{m,i}(x) A_{m,j}(x) dx = 0$. 
\item $|A_{m,i}(x)| \leq 2^{-m}$. 
\end{enumerate}
\end{lemma}

\begin{proof}[Proof of Lemma \ref{legendrelem}]
Note that 
\[
s_{m,i}(u) = \left[\frac{3}{2}\left( 2^{m+1}(u-i2^{-m})-1 \right)^2 - \frac{1}{2} \right] \one\{ i2^{-m} \leq u \leq (i+1)2^{-m} \},
\] thus $-1/2\leq s_{m,i}(u) \leq 1$. Let $I_{m,i} = [i2^{-m}, (i+1)2^{-m}]$.
By the orthogonal property of $\lf_2(t)$ with $1$ and $t$ via Lemma \ref{orthogonal_prop}, we have
\begin{align*}
\int s_{m,i}(u)du &= \int_0^1 2^{-m} \lf_2(t)dt = 0,  \\
\int u s_{m,i} (u) du &= \int_0^1 2^{-m} (t+i)\lf_2(t)dt = 0
\end{align*}
for all $i=0, \ldots, 2^m-1$. 

For the second claim, note that  when $x < i2^{-m}$, $A_{m,i}(x) = \int s_{m,i}(u)du =0$ and when $x>(i+1)2^{-m}$, $A_{m,i}(x) = 0$ since $s_{m,i}$ is supported on $I_{m,i}$.
For $x \in I_{m,i}$,  
\begin{align*}
A_{m,i}(x) &= \int_x^{(i+1)2^{-m}} \lf_2(2^m u - i )du 
= 2^{-m} \int_{2^m x-i}^1 \lf_2(t) dt\\
&=\frac{2^{-m}}{10}  \left. \left[\lf_3(t)-\lf_1(t)\right] \right|_{2^m x-i}^1  
= \frac{2^{-m}}{10} \left[ -\lf_3(2^m x-i) + \lf_1(2^m x-i)\right]
%&= \int_{x-i2^{-m}}^{2^{-m}} \left( \frac{3}{2}(2^{m+1}t-1)^2 - \frac{1}{2} \right) dt \\
%&= \int_{2^m x-i}^1 2^{-m}\left(\frac{3}{2}(2v-1)^2-\frac{1}{2}\right) dv \\
%&= 2^{-m} \left[-2(2^m x-i)^3 + 3(2^m x-i)^2 -(2^m x -i) \right]\\
%&= 2^{-m} \frac{1}{10} \left[ -L_3(2^m x-i) + L_1(2^m x-i)\right]
\end{align*}
where the penultimate equality follows since
\begin{equation}\label{legendre}
\int \lf_2(x)dx =\frac{ \lf_3(x)-\lf_1(x)}{10} + C
\end{equation}
from recurrence relations of Legendre polynomials, and the last equality follows since $\lf_3(1)=\lf_1(1)=1$. 
Using \eqref{legendre}, we have
\[
\int A_{m,i}(x) dx = \frac{2^{-2m}}{10} \int (-\lf_3(x)+\lf_1(x)) dx = 0.
\]
Also for $i \neq j$, $\int A_{m,i}(x) A_{m,j}(x) dx = 0$. Indeed, if $x \in I_{m,i}$, then $A_{m,j}(x) = 0$ and similarly if $x \in I_{m,j}$ then $A_{m,i}(x)= 0$. Lastly, 
\begin{align}
    \int A_{m,i}(x)^2 dx &= \frac{2^{-2m}}{100} \int_{i2^{-m}}^{(i+1)2^{-m}} \left[-\lf_3(2^m x-i) + \lf_1(2^m x-i) \right]^2 dx \nonumber\\
    &= \frac{2^{-3m}}{100} \int_0^1 [\lf_3(u)-\lf_1(u)]^2 du = \frac{2^{-3m}}{100} \int_0^1 [\lf_3^2(u)+\lf_1^2(u)]du  \nonumber \\
    &=\frac{C}{100} 2^{-3m}, \label{Asquare}
\end{align}
where the third equality follows since $\int_0^1 \lf_3(u)\lf_1(u)du=0$ and $C$ in \eqref{Asquare} is defined such as 
\begin{equation}\label{defc}
C:= \int_0^1 (\lf_3^2(u)+\lf_1^2(u)) du = \frac{1}{7}+\frac{1}{3} = \frac{10}{21}.
\end{equation}
For the third claim, it is clear that $|s_{m,i}(u)| \leq 1$ and since $A_{m,i}(x)$ is nonzero only if $x \in I_{m,i}$.
\end{proof}

We are now ready to present the proof of Theorem \ref{lowerbound}. 

\begin{proof}[Proof of Theorem \ref{lowerbound}]
Without loss of generality, we let $M=1$, $b = 1/2$, and let $B =
3/2$. Indeed, the construction of $f_\alpha(\bm x)$ for $\bm x \in
[0,1]^d$ and $b \leq f_\alpha \leq B$ below can be modified by
considering $\tilde f_\alpha(\bm x) =M^{-d} f_\alpha(\bm x/M)$ where
$\bm x \in [0,M]^d$  and $M^{-d} b\leq \tilde f_\alpha(\bm x) \leq
M^{-d}B$.  

We let $\bm{1} = (1,1, \ldots, 1) \in \mathbb{R}^d$, and
$G_\alpha$ be a mixture of discrete and continuous distribution where
$G_\alpha(\bm 1)=1/2$ and for $\alpha_{M,I} \in \{0,1\}$ and $\bm
\theta \in [0,1)^d$, using the definition \eqref{def_s}, for the set
$\mathcal{O} \subseteq [0,1]^d$,  
%\[
%dG_\alpha(\theta_1,\theta_2) = \frac{1}{2} \theta_1 \theta_2 \left(
%4+ \frac{1}{|\mathcal{M}_m|}\sum_{M \in \mathcal{M}_n} \sum_{I \in
%\mathcal{I}_M} \alpha_{M,I} s_{m_1,i_1}(\theta_1)
%s_{m_2,i_2}(\theta_2)  \right) 
%\]
\begin{align*}
&G_\alpha(\mathcal{O}) =\\
&\frac{1}{2} \one\{ \bm{1} \in \mathcal{O}\} +  \frac{1}{2} \int_{\mathcal{O}} \left(\prod_{j=1}^d \theta_j\right) \left( 1+
  \frac{1}{|\mathcal{M}_n|}\sum_{M \in \mathcal{M}_m} \sum_{I \in
    \mathcal{I}_M} \alpha_{M,I} \prod_{j=1}^d s_{m_j,i_j}(\theta_j)
\right) d\bm \theta, 
\end{align*}
where
$\mathcal{M}_m = \{(m_1, \ldots, m_d) \in \mathbb{N}^d: m_1+\ldots+m_d
= m, m_j = c_d k_j, \  1\leq j\leq d \}$ where $c_d = 2d$ is a
universal constant only depending on $d$, $k=\sum_{j=1}^d k_j$, and
$\mathcal{I}_{M} = \{(i_1,\ldots, i_d) \in \mathbb{N}^d: i_j \leq
2^{m_j}, 1\leq j\leq d\}$. 

Clearly, $\int_{[0,1]^d} dG_\alpha(\bm \theta)=1$ and for $\bm \theta =(\theta_1,\ldots, \theta_d) \in [0,1]^d$, we have
\[
\left|\frac{1}{|\mathcal{M}_m|} \sum_{M \in \mathcal{M}_m} \sum_{I \in
    \mathcal{I}_M} \prod_{j=1}^d s_{m_j,i_j}(\theta_j)\right| \leq \frac{1}{|\mathcal{M}_m|} \sum_{M \in \mathcal{M}_m} \prod_{j=1}^d \left| s_{m_j,i_j^*}(\theta_j)\right|\leq 1,
\]
where the first inequality holds since for any $(\theta_1,\ldots, \theta_d)$, per each $M$, there exists a unique index set $(i_1^*,\ldots, i_d^*)$ where each $s_{m_j,i_j^*}$ is nonzero, and the last inequality holds since each $|s_{m_j,i_j^*}|$ is upper bounded by $1$.

Then when $0\leq x_j \leq 1$ for $j=1,\ldots, d$, we explicitly represent %$f_\alpha(\underline x) = \int \frac{\one\{x_1\leq \theta_1,\ldots,x_d\leq \theta_d \}}{\prod_{j=1}^d \theta_j} dG_\alpha(\underline \theta)$
\begin{align*}
f_\alpha(\bm x) &=\int \frac{\one\{x_1\leq \theta_1,\ldots,x_d\leq \theta_d \}}{\prod_{j=1}^d \theta_j} dG_\alpha(\bm \theta) \\
&= \frac{1}{2} + \frac{1}{2}\prod_{j=1}^d(1-x_j) + \frac{1}{2|\mathcal{M}_m|} \sum_{M \in \mathcal{M}_m} \sum_{I \in
    \mathcal{I}_M} \alpha_{M,I} \prod_{j=1}^d A_{m_j,i_j}(x_j).
\end{align*}
Note that $f_\alpha(\bm x) \leq f_\alpha(\bm 0) \leq 1+1/2 = 3/2$.

Using the Varshamov-Gilbert Lemma (see e.g. Lemma 2.9 of \cite{tsybakovbook}), there exists at least $\exp(C_d |\mathcal{M}_m|2^m)$ with $|\mathcal{M}_m| \sim m^{d-1}/(d-1)!$ possible scale mixtures of uniform densities such that \begin{equation}\label{VG}
c 2^m |\mathcal{M}_m| \leq \sum_{M \in \mathcal{M}_m}  \sum_I (\alpha_{M,I}-\beta_{M,I})^2 \leq 2^m |\mathcal{M}_m|
\end{equation}
is satisfied for some constant $c \in (0,1)$.

Then
\begin{align*}
    \int &(f_\alpha-f_\beta)^2 = 
    \frac{1}{4|\mathcal{M}_m|^2}\sum_{M \in \mathcal{M}_m}  \int \left(\sum_I (\alpha_{M,I}-\beta_{M,I})  \prod_{j=1}^d A_{m_j,i_j}(x_j)  \right)^2d \bm x \\
      &+ \frac{1}{4|\mathcal{M}_m|^2}  \sum_{M \neq \tilde M  \in \mathcal{M}_m}  \int \left(\sum_I (\alpha_{M,I}-\beta_{M,I})  \prod_{j=1}^d A_{m_j,i_j}(x_j)  \right) \times \\
      &\left(\sum_{\tilde I} (\alpha_{\tilde M,\tilde I}-\beta_{\tilde M,\tilde I})  \prod_{j=1}^d A_{\tilde m_j,\tilde i_j}(x_j)  \right)\\
      &=:\frac{1}{4}\left[ (*)+(**) \right],
\end{align*}
where the first term above is bounded as follows.
\begin{align*}
    (*) %&= \frac{1}{|\mathcal{M}_m|^2}\sum_M \int \left(\sum_I (\alpha_{M,I}-\beta_{M,I})  \prod_{j=1}^d A_{m_j,i_j}(x_j) \right)^2 d \underline x\\
    &= \frac{1}{|\mathcal{M}_m|^2}\sum_M  \sum_I (\alpha_{M,I}-\beta_{M,I})^2 \prod_{j=1}^d \left(\int A_{m_j,i_j}^2 \right)\\
    &=\frac{1}{|\mathcal{M}_m|^2}\sum_M \sum_I (\alpha_{M,I}-\beta_{M,I})^2 \left( \frac{1}{210} \right)^d 2^{-3 \sum_{j=1}^d m_j} \\
    &= c_d  \frac{2^{-2m}}{|\mathcal{M}_m|}   \sim m^{-(d-1)}2^{-2m},
\end{align*}
where the first equality follows since for any $j$, $\int A_{m_j,i_j} A_{m_j, \tilde i_j} = 0$ for $i_j \neq \tilde i_j$, the second equality follows by the second assertion of Lemma \ref{legendrelem}, and the last equality follows by \eqref{VG}.
In addition, by Lemma \ref{crossprod},  $|(**)| \leq \frac{2}{3} (*)$ .

Also since $f_\alpha>1/2$, we know that 
\[
KL(f_\alpha, f_\beta) \leq \int \frac{(f_\alpha-f_\beta)^2}{f_\alpha} \leq 2 L_2^2(f_\alpha, f_\beta).
\]
Moreover, since $f_\alpha<3/2$, we have that
\[
h^2(f_\alpha, f_\beta) = \int \frac{(f_\alpha-f_\beta)^2}{(f_\alpha+f_\beta)^2} \geq (1/9) L_2^2(f_\alpha,f_\beta).
\]

Applying Fano's method (see e.g. Lemma 3 of \cite{Yu97lecam}), we obtain the minimax lower bound
\[
C_1 2^{-2m}m^{-(d-1)} \left(1-C_2 \frac{ n 2^{-2m}m^{-(d-1)}}{m^{d-1} 2^m} \right),
\]
where $C_1$ and $C_2$ are universal constants depending only on $d$.
We take 
\[
2^{3m}m^{2(d-1)} \sim n,
\] that is, $2^{-2m}m^{-4(d-1)/3} \sim n^{-2/3}$ which implies that the lower bound is of order $n^{-2/3}(\log n)^{(d-1)/3}$. This completes the proof of Theorem \ref{lowerbound}.
\end{proof}
%\color{black}
\begin{lemma}\label{crossprod}
Using the same notation in Lemma \ref{legendrelem} and the proof of Theorem \ref{lowerbound}, we consider 
\begin{align*}
(**) &= \frac{1}{|\mathcal{M}_m|^2}  \sum_{M \neq \tilde M} \int \left(\sum_I (\alpha_{M,I}-\beta_{M,I})  \prod_{j=1}^d A_{m_j,i_j}(x_j)  \right) \times \\
&\left(\sum_{\tilde I} (\alpha_{\tilde M,\tilde I}-\beta_{\tilde M,\tilde I})  \prod_{j=1}^d A_{\tilde m_j,\tilde i_j}(x_j)  \right). 
\end{align*}
We claim the following.
\begin{align}\label{twostar}
(**) \leq \frac{2}{3} \left(\frac{1}{210}\right)^d \frac{2^{-2m}}{|\mathcal{M}_m|}. 
\end{align}
\end{lemma}

\begin{proof}[Proof of Lemma \ref{crossprod}]
First, note that
\[
(**) \leq \frac{1}{|\mathcal{M}_m|^2}\sum_{M \neq \tilde M} \sum_I \sum_{\tilde I} \left| \prod_{j=1}^d \int A_{m_j,i_j}A_{\tilde m_j, \tilde i_j}  \right|.
\]
%Now we  prove the following claim regarding \eqref{twostar}:
%({\color{black}\textbf{please consider isolating the proof of the claim
%    \eqref{twostar} in a separate lemma}})
%\begin{align*}
%    |(**)| 
%    %&\leq \frac{1}{m^2}\sum_{M\neq \tilde M} \int  (\sum_I A_{m_1,i_1}A_{m_2,i_2})(\sum_{\tilde I} A_{\tilde m_1,\tilde i_1}A_{\tilde m_2,\tilde i_2}) \\
%    &\leq  \frac{1}{|\mathcal{M}_m|^2}\sum_{M \neq \tilde M} \sum_I \sum_{\tilde I} \left| \prod_{j=1}^d \int A_{m_j,i_j}A_{\tilde m_j, \tilde i_j}  \right|
%    \leq \frac{2}{3} \left(\frac{C}{100}\right)^d \frac{2^{-2m}}{|\mathcal{M}_m|}  .
%\end{align*}

We first consider the case where $m_1 = \tilde m_1+c_d$ and $m_2 =
\tilde m_2-c_d$ where $c_d = 2d$ and $m_j = \tilde m_j$ for $j=3,\ldots,d$. Note that
for any $I_{m_1,i_1}$, there exists only one $\tilde I_{\tilde m_1,
  \tilde i_1}$ which includes $I_{m_1,i_1}$. If these two intervals
are disjoint, then the integral $\int A_{m_1,i_1}A_{\tilde m_1,\tilde
  i_1}$ becomes zero. Also we can check that $i_1 2^{-m_1}
<x<(i_1+1)2^{-m_1}$ is equivalent to $(i_1/2^{c_d})2^{-\tilde m_1} <
x< ((i_1+1)/2^{c_d})2^{-\tilde m_1}$. Thus the corresponding $\tilde
i_1$ can be taken as $\lfloor i_1/2^c_d \rfloor$. Suppose for now that
$i_1$ is divisible by $2^{c_d}$ with a remainder of $\ell$. Then
$\ell$ can take values from $\{0,\ldots, 2^{c_d}-1\}$, which leads to 
\begin{align*}
\int &A_{m_1,i_1}A_{\tilde m_1, \tilde i_1}
 = \int_{I_{m_1,i_1}} A_{m_1,i_1}A_{\tilde m_1, \tilde i_1} \\
 &= \frac{2^{-m_1-\tilde m_1}}{100} \int_{I_{m_1,i_1}} \left(\lf_3(2^{m_1}x_1-i_1)-\lf_1(2^{m_1}x_1-i_1) \right) \times \\
 &\left( \lf_3(2^{\tilde m_1}x_1-\lfloor i_1/2^{c_d} \rfloor)-\lf_1(2^{\tilde m_1}x_1- \lfloor i_1/2^{c_d} \rfloor) \right) \\
 &=\frac{ 2^{-m_1-\tilde m_1}2^{-m_1}}{100}
\left[\int_0^1 (\lf_3(u)-\lf_1(u))\left( \lf_3(\frac{u+\ell}{2^{c_d}})-\lf_1(\frac{u+\ell}{2^{c_d}}) \right) du\right] \\
&= \frac{2^{-2m_1-\tilde m_1}}{100}\left[ \int_0^1 \lf_3(u)\lf_3(\frac{u+\ell}{2^{c_d}}) + \lf_1(u)\lf_1(\frac{u+\ell}{2^{c_d}}) -\lf_1(u)\lf_3(\frac{u+\ell}{2^{c_d}}) \right] \\
&=\frac{2^{-2m_1-\tilde m_1}}{100} \left[ 2^{-3{c_d}}\frac{1}{7} + 2^{-{c_d}}\frac{1}{3}  - \int_0^1 \lf_1(u)\lf_3(\frac{u+\ell}{2^{c_d}}) \right]
\end{align*}
when $i_1$ is divisible by $2^{c_d}$ with a remainder of $\ell$ (with $0\leq \ell \leq 2^{{c_d}}-1$ and $\ell \in \mathbb{N}$).

With some tedious calculations, we can show 
\begin{align}\label{tedious}
\int_0^1 \lf_1(u)\lf_3(\frac{u+\ell}{2^{c_d}}) du = 2^{-3{c_d}}(10\ell^2+10\ell+3) + 2^{-2{c_d}}(-10\ell-5)+2^{-{c_d}}(2) 
\end{align}
Depending on the value of $\ell$, the above expression can take a
negative value. Solving the second order equation of $\ell$,
\eqref{tedious} is minimized at $\ell^* = (2^{c_d}-1)/2$, which gives
the minimum value $-2^{-{c_d}-1}(1-2^{-2{c_d}}) \geq
-2^{-{c_d}-1}$. Similarly, the maximum will be achieved at $\ell=0$ or
$\ell = 2^{{c_d}-1}$, which gives the maximum value $2^{-{c_d}}(2-3
\times 2^{-2{c_d}})(1-2^{-2{c_d}}) \leq 2^{-{c_d}+1}$. This shows 
\begin{align*}
\left|2^{-3{c_d}}\frac{1}{7} + 2^{-j}\frac{1}{3}  - \int_0^1
  \lf_1(u)\lf_3(\frac{u+\ell}{2^{c_d}})\right|  &\leq
                                                  2^{-{c_d}}\frac{5}{3}-5\times
                                                  2^{-2{c_d}}+\frac{20}{7}2^{-3{c_d}}
                                                  \\ &\leq 5 \left(\frac{1}{3}2^{-{c_d}}+\frac{1}{7}2^{-3{c_d}} \right)\leq 2^{-{c_d}} (5C),
\end{align*}
where $C = \frac{10}{21}$ as in \eqref{defc}.

By repeating the similar calculation for the case $m_2 = \tilde
m_2-c_d$, we have when $m_1 = \tilde m_1+c_d$ and $m_2 = \tilde
m_2-c_d$ 
\begin{align*}
& \left| \int A_{m_1,i_1}A_{\tilde m_1, \tilde i_1} \int A_{m_2,i_2}
  A_{\tilde m_2, \tilde i_2} \right| \\ &\leq 2^{-m_1-\tilde
  m_1}2^{-m_1}2^{-m_2-\tilde m_2}2^{-\tilde
  m_2}\Big(\frac{1}{100}\Big)^2 2^{-2{c_d}}(5C)^2. 
\end{align*}
Note that for each interval $I_{m_1,i_1}$, there exists a unique
corresponding interval $\tilde I_{\tilde m_1, \tilde i_1}$ which is
not disjoint with each other and also for each interval $I_{\tilde
  m_2, \tilde i_2}$, there exists a unique corresponding interval
$I_{m_2, i_2}$ which is not disjoint with each other. Thus 
\begin{align*}
& \sum_I \sum_{\tilde I} \left| \prod_{j=1}^d A_{m_j,i_j} A_{\tilde
  m_1, \tilde i_j}\right| \\ &\leq 2^{-m_1-\tilde m_1}2^{-m_2-\tilde
                            m_2}\Big(\frac{5C}{100}\Big)^2 2^{-2{c_d}}
                            2^{\sum_{j=3}^d m_j} 2^{-\sum_{j=3}^d
                            3m_j}\Big(\frac{5C}{100}\Big)^{m-2} \\
&=\Big(\frac{5C}{100}\Big)^d  2^{-2m}2^{-2c_d}. 
\end{align*}

Using the above ideas, let us consider more general case. For $j_1,\ldots, j_{d-1}$ (whose value is among $c_d\{0, \pm 1, \pm 2, \ldots, \pm (k-1)\}$, we consider the case $m_1 = \tilde m_1+j_1, m_2 = \tilde m_2+j_2, \ldots, m_{d-1} = \tilde m_{d-1}+j_{d-1}$ and $m_d = \tilde m_d+j_d$  where  $\sum_\ell j_\ell = 0$ and $m_d = m-\sum_{j=1}^{d-1}m_j$. We suppose $(m_1,\ldots,m_d) \neq (\tilde m_1,\ldots, \tilde m_d)$.
We know that $\sum_{\ell_1}^d |j_\ell|$ is among $\{2c_d,4c_d,\ldots \}$ and let us suppose $\sum_{\ell_1}^d |j_\ell| = 2c_d$ for now. 
There exist at most ${d \choose 2}2 = d(d-1) < 5^d$ possible pair $\tilde M$ for each $M$.
Indeed, we pick $2$ dimensions where we put plus sign on the first dimension (and the minus sign for the other dimension) or vice versa.
For this case, our previous calculations give 
\begin{align*}
    \sum_I \sum_{\tilde I} \left|\prod_{j=1}^d \int A_{m_j, i_j} A_{\tilde m_j, \tilde i_j}\right| &\leq \left( \frac{5C}{100}\right)^d 2^{-2m} 2^{-2c_d}. 
\end{align*}

More generally, when $\sum_{\ell_1}^d |j_\ell| = Jc_d$, there exist at most $\sum_{j=2}^d {d \choose j} (2J)^j \leq (2J+1)^d$ possible pair $\tilde M$ for each $M$ (pick $j$ dimensions, sign choices for each such dimension, and counting of splitting $J$ with $j-1$ pieces).
For this case, 
\begin{align*}
    \sum_I \sum_{\tilde I} \left|\prod_{j=1}^d \int A_{m_j, i_j} A_{\tilde m_j, \tilde i_j}\right| &\leq
    \left( \frac{5C}{100}\right)^d 2^{-2m} 2^{-2Jc_d}. 
\end{align*}

Thus we  bound
\begin{align*}
|(**)| &\leq \frac{1}{|\mathcal{M}_m|} \sum_{J\geq 2} (2J+1)^d \left( \frac{5C}{100}\right)^d  2^{-2m} 2^{-2Jc_d} \\
&\leq \left( \frac{5C}{100}\right)^d \frac{2^{-2m}}{|\mathcal{M}_m|}
  \sum_{J\geq 2} (10J+5)^d 2^{-2J c_d} \\ &\leq \frac{2}{3}\left( \frac{C}{100}\right)^d \frac{2^{-2m}}{|\mathcal{M}_m|} = \frac{2}{3}\left( \frac{1}{210}\right)^d \frac{2^{-2m}}{|\mathcal{M}_m|}
\end{align*}
by the choice of $c_d=2d$.
The claim in Lemma \ref{crossprod} is proved.
\end{proof}
\color{black}

\subsection{Proof of Proposition
  \ref{prodsmugen}}

The main idea is to
construct a decomposition of $[0, M]^d$ into rectangles $\{R_j\}$
which satisfy the conditions of Theorem \ref{jrec}. This will allow us
to deduce Proposition \ref{prodsmugen} as a consequence of Theorem
\ref{jrec}. For the decomposition, we use the following
univariate result. 

\begin{lemma}\label{decomp1d}
  Let $P_0$ be a probability measure on $[0, M]$ having a right
  continuous nonincreasing density $p_0$ on $[0, M]$. Assume that
  $p_0$ is bounded from above on $[0, M]$ by $B = p_0(0) < 
  \infty$. For every 
  $\delta \in (0, 1)$, there exist points $0 = x_0 < x_1 < \dots < x_K
  \leq M$ with 
  \begin{equation}\label{decomp1d.K} 
    K \leq \lceil \log \log \frac{4B}{\delta} \rceil
  \end{equation}
  such that
  \begin{equation}\label{decomp1d.eq1}
    \max_{1 \leq k \leq J} \frac{p_0(x_{k-1})}{\sqrt{p_0(x_k -)}} \leq
    2\sqrt{B}
  \end{equation}
  where $p_0(x_k-)$ above denotes the left limit of $p_0$ at $x_k$,
  and
  \begin{equation}\label{decomp1d.eq2}
   P_0[x_K, M] \leq \delta M. 
  \end{equation}
\end{lemma}

\begin{proof}[Proof of Lemma \ref{decomp1d}]
  We take $x_0 = 0$ and define
  \begin{equation*}
    x_{k} = \sup \left\{u \in [x_{k-1}, M] :
      \frac{p_0(x_{k-1})}{\sqrt{p_0(u)}} \leq 2 \sqrt{B} \right\}
  \end{equation*}
  for $k = 1, 2, \dots, J$ where $K$ is the smallest integer for which
  either $x_K = M$ or $p_0(x_K) \leq \delta$. This immediately ensures
  that $P_0[x_K, M] \leq \delta M$ (this is obvious if $x_K = M$ as
  then $[x_K, M]$ will be the singleton $\{M\}$ which has zero $P_0$
  measure; if $x_K < M$, then $P_0[x_k, M] \leq p_0(x_K) (M - x_K)
  \leq \delta M$). 

  The inequality inside
  the supremum above will hold for points slightly smaller than $x_k$,
  and thus by taking the left limit, we obtain
  \eqref{decomp1d.eq1}. As long as $x_k < M$, the inequality in the
  supremum of the 
  definition of $x_k$ will be violated for points slightly larger than
  $x_k$. Thus by taking the right limit and using the assumed
  right-continuity of $p_0$, we get
  \begin{equation*}
    p_0(x_k) \leq \frac{p_0^2(x_{k-1})}{4B} \qt{provided $x_k < M$}. 
  \end{equation*}
  Using this recursively for $k \geq 1$ along with $p_0(x_0) = B$, we
  obtain 
  \begin{equation*}
    p_0(x_k) \leq \frac{4B}{4^{2^k}} \qt{provided $x_k < M$}. 
  \end{equation*}
  One can check that for $\frac{4B}{4^{2^k}} \leq \delta$  when  $k$
  equals the right hand side of \eqref{decomp1d.K}. This completes the
  proof of Lemma \ref{decomp1d}.   
\end{proof}

We are now ready to prove Proposition \ref{prodsmugen}.

\begin{proof}[Proof of Proposition \ref{prodsmugen}]
  We use Lemma \ref{decomp1d} with $\delta := n^{-2/3}/(AMd)$ for each
  univariate density $p_{0j}, 1 \leq j \leq d$. For each $p_{0j}$,
  this gives points $x_{0, j} = 0 < 
  x_{1, j} < \dots < x_{K_j, j} \leq M$ satisfying the conditions of Lemma
  \ref{decomp1d} for $p_{0j}$. We decompose $[0, M]^d$ (which is the
  full domain of $p_0$) into rectangles
  \begin{equation*}
    R(k_1, \dots, k_d) := \prod_{j=1}^d [x_{k_j, j}, x_{k_j + 1, j}]. 
  \end{equation*}
  as each $k_j$ ranges in $0, 1, \dots, K_j - 1$. These rectangles
  clearly have disjoint interior. They do not cover the whole of $[0,
  M]^d$ though because $K_j$ can be strictly smaller than $M$. But
  their union has probability
  \begin{align*}
    & P_0 \left( \cup \left\{R(k_1, \dots, k_d) : 0 \leq k_j < K_j, 1
    \leq j \leq d \right\} \right) \\ &= P_0 \left( \prod_{j=1}^d [0,
                                     x_{K_j}] \right) \geq 1 - A
                                        \sum_{j=1}^d P_{0j}[x_{K_j,
                                        j}, M] \geq 1 - A \delta M 
      d = 1 - n^{-2/3} 
  \end{align*}
  where $P_{0j}$ is the probability measure having density $p_{0j}$.
  For a fixed $\qh \in (1, \infty)$, we now bound $W(R, p_0, \qh)$ for
  each rectangle $R = R(k_1, \dots, k_d)$ in order to apply Theorem
  \ref{jrec}. Observe first that
  \begin{align*}
    \|p_0^{-1}\|_{L_{\qh}(R)}^{\qh} &= \int_R p_0^{-\qh} \leq a^{-\qh}
                                      \int_R p_{01}^{-\qh} \dots
                                      p_{0d}^{\qh} = a^{-\qh}
                                      \prod_{j=1}^d \int_{x_{k_j,
                                      j}}^{x_{k_j+1, j}}
                                      p_{0j}^{-\qh}. 
  \end{align*}
  Because $p_{0j}$ is a nonincreasing density, we can write $p_{0j}(x)
  \geq p_0(x_{k_j + 1, j}-)$ for $x$ in the interior of $[x_{k_j, j},
  x_{k_j + 1, j}]$. We thus get
  \begin{align*}
    \|p_0^{-1}\|_{L_{\qh}(R)}^{1/4} \leq |R|^{1/(4 \qh)} a^{-1/4}
    \prod_{j=1}^d \left(p_{0j}(x_{k_j + 1, j} - ) \right)^{-1/4}. 
  \end{align*}
  Also
  \begin{align*}
    \max_{x \in R} p_0(x) \leq A \max_{x \in R} p_{01}(x_1) \dots
    p_{0d}(x_d) \leq A \prod_{j=1}^d p_{0j}(x_{k_j, j}). 
  \end{align*}
  As a result
  \begin{align*}
    |R|^{1/(4\ph)}     \|p_0^{-1}\|_{L_{\qh}(R)}^{1/4} \sqrt{
    \max_{x \in R} p_0(x) } &\leq |R|^{1/4} a^{-1/4} \sqrt{A}
                              \prod_{j=1}^d \left(\frac{p_{0j}(x_{k_j,
                              j})}{\sqrt{p_{0j}(x_{k_j + 1, j} -)}} \right)^{1/2}
  \end{align*}
  Using $|R| \leq M^d$ and then Lemma \ref{decomp1d} to control the
  terms in the product above, we obtain
  \begin{align*}
    |R|^{1/(4\ph)}     \|p_0^{-1}\|_{L_{\qh}(R)}^{1/4} \sqrt{
    \max_{x \in R} p_0(x) } \leq M^{1/4} a^{-1/4} \left(\sqrt{2}
    B^{1/4} \right)^d. 
  \end{align*}
  Thus $W(R, p_0, \qh) \leq C_{a, M, B}$. The number of rectangles
  here is
  \begin{align*}
    \prod_{j = 1}^dK_j \leq \lceil \log \log \frac{4B}{\delta}
    \rceil^d = \lceil \log \log \left(4 B A M d n^{2/3} \right) 
    \rceil^d \leq C_{A, B, M, d} \lceil \log \log n \rceil. 
  \end{align*}
  Proposition \ref{prodsmugen} now follows from Theorem \ref{jrec}
  (note that $\lceil \log \log n \rceil \leq C \log \log n$ for $n
  \geq 3$). 
\end{proof}

\subsection{Proofs of Proposition \ref{1dimproved} and Proposition
  \ref{1d_vandegeer}}\label{proofs_1d}

For the proof of Proposition \ref{1dimproved}, we use the
following lemma which can be seen as a univariate analogue of
Proposition \ref{mainhrbound}. It only applies to $d = 1$ but gives a
better bound (in terms of logarithmic factors) compared to Proposition
\ref{mainhrbound}. 

\begin{lemma}\label{hrb_1d}
  Fix $d = 1$, $n \geq 2$ and let $R = [a, b]$ be contained in the
  support of $p_0$. Then there exists a positive constant $C$ such
  that $H(t, R)$ (for the definition, see \eqref{secupbo}) 
  satisfies the following bound for every $t > 0$: 
\begin{align}\label{hrb_1d.eq}
  H(t, R) \leq C \gamma^{1/3} b^{1/6} t n^{-1/3} + C
  \sqrt{\gamma} b^{1/4} 
  \sqrt{\frac{t}{n}} +  \frac{C
  \gamma}{nt} \sqrt{b} + \frac{C b^{1/3} \gamma^{2/3}}{n^{2/3}}
\end{align}
where $\gamma := p_0(a)/\sqrt{p_0(b-)}$ with $p_0(b-)$ denoting the left
limit of $p_0$ at $b$.  
\end{lemma}

\begin{proof}[Proof of Lemma \ref{hrb_1d}]
  For every $a < y < b$ (where $R = [a, b]$), we have
  \begin{equation*}
    H(t, [a, b]) \leq H(t, [a, y]) + H(t, [y, b]). 
  \end{equation*}
  For $H(t, [a, y])$, we use the trivial bound \eqref{trivialb} to get
  \begin{equation*}
    H(t, [a, y]) \leq 2 p_0(a) (y - a). 
  \end{equation*}
  For $H(t, [y, b])$, we use Lemma \ref{ESupResult} with $\alpha = 0$ and
  \begin{equation*}
    \|p_0^{-1}\|_{L_{\mathfrak{q}([y, b])}} = \left(\int_y^b
      p_0^{-\mathfrak{q}} \right)^{1/\mathfrak{q}} \leq \frac{(b -
    y)^{1/\mathfrak{q}}}{p_0(y-)}. 
\end{equation*}
This gives
\begin{align*}
  H(t, [y, b]) &\leq C \sqrt{\frac{t}{n}} \sqrt{\beta}
  \left(\frac{b - y}{p_0(y-)} \right)^{1/4} + \frac{C \beta}{nt}
                 \left(\frac{b - y}{p_0(y-)} \right)^{1/2} \\
  &\leq C \sqrt{\frac{t}{n}} \sqrt{\beta}
  \left(\frac{b}{p_0(b-)} \right)^{1/4} + \frac{C \beta}{nt}
                 \left(\frac{b}{p_0(b-)} \right)^{1/2}. 
\end{align*}
For $\beta$, we use Lemma \ref{bds} to get
\begin{equation*}
 \sqrt{\beta} \leq \sqrt{p_0(a)} + \frac{t}{\sqrt{y-a}}. 
\end{equation*}
Putting these inequalities together, we obtain the following upper bound on $H(t, [a, b])$:
\begin{equation}\label{bound_y}
  \begin{split}
&  2 p_0(a) (y - a) + C \sqrt{\frac{t}{n}} \left( \sqrt{p_0(a)} +
  \frac{t}{\sqrt{y-a}}\right)   \left(\frac{b}{p_0(b-)}
  \right)^{1/4} \\ &+ \frac{C}{nt} \left(p_0(a) + \frac{t^2}{y-a}
                     \right)   \left(\frac{b}{p_0(b-)} \right)^{1/2}.
  \end{split}                     
\end{equation}
This bound is true for every $a < y < b$. We now make the choice
\begin{align*}
  y = a + \frac{t}{n^{1/3}} \left(\frac{b}{p_0(b-)} \right)^{1/6}
  \frac{1}{(p_0(a))^{2/3}}. 
\end{align*}
This will be a valid choice for $y$ if  $a < y < b$. If $y > b$, then
\eqref{bound_y} will exceed $2 p_0(a) (b-a)$ which is clearly larger
than $H(t, R)$ (see \eqref{trivialb}). We can therefore plug in this
value of $y$ in \eqref{bound_y} and it is trivial to verify that
\eqref{bound_y}  then leads to \eqref{hrb_1d.eq}. 
\end{proof}

\begin{proof}[Proof of Proposition \ref{1dimproved}]
  We use Lemma \ref{decomp1d} with $\delta = n^{-2/3}/M$ to obtain
  points $x_0 = 0 < x_1 < \dots < x_K \leq M$ satisfying the
  conditions of Lemma \ref{decomp1d} for $p_0$. Then using
  \eqref{altexpEGt}, \eqref{eledeco} and \eqref{trivialb}, we can
  write 
\begin{equation*}
  \E G(t) \leq 2 \sum_{i=1}^K H(t, [x_{i-1}, x_i]) + 4 P_0 [x_k, M]. 
\end{equation*}
Using the bound given in Lemma \ref{hrb_1d} for $H(t, [x_{i-1},
x_i])$ and the bound given in Lemma \ref{decomp1d} for $P_0[x_k, M]$,
we obtain 
\begin{align}\label{1d_bo}
  \E G(t) \leq 2 K C_{B, M} \left(t n^{-1/3} + \sqrt{\frac{t}{n}} +
  \frac{1}{nt} + \frac{1}{n^{2/3}} \right) + \delta M. 
\end{align}
Note that we have replaced $\gamma$ appearing in \eqref{hrb_1d.eq} by
$2 \sqrt{B}$ and $b$ appearing in \eqref{hrb_1d.eq} by $M$ (these
constants are absorbed in $C_{B, M}$). Theorem \ref{hellexp} now
completes the proof (this argument is similar to the one used in
Theorem \ref{jrec}). The $\log \log n$ factor comes from the presence
of $K$ on the right hand side of \eqref{1d_bo}. 
\end{proof}

\begin{proof}[Proof of Proposition \ref{1d_vandegeer}]
  We use Theorem \ref{hellexp} so the goal is to bound $\E G(t)$ from above. Using \eqref{altexpEGt} and \eqref{temp.jb}, we get
  \begin{align*}
      \mathbb{E} G(t) &= 2\E \sup_{p \in \smud : h(p_0, p) \leq t} \int 
                \frac{4p_0}{p_0 + p} d(P_0 - P_n) \\
                &= 2\E \sup_{p \in \smud : h(p_0, p) \leq t} \int 
                \left[\frac{p_0 - p}{p_0 + p} \right] d(P_0 - P_n) \\
                &\leq \frac{C}{\sqrt{n}} J(t \sqrt{2}) + \frac{C}{nt^2} J^2(t \sqrt{2}) 
  \end{align*}
  where $J(\delta)$ is as defined in \eqref{braentint} with
  \begin{align*}
      \mathcal{F} := \left\{\frac{p_0 - p}{p_0 + p} : p \in \smud, h(p_0, p) \leq t \right\}
  \end{align*}
  Let $\mathcal{F}_{\downarrow}$ denote the class of all non-increasing functions $f$ on $[0, M]$ (note that $[0, M]$ is the support of $p_0$) for which $0 \leq f(x) \leq \sqrt{B}$ for all $x \in [0, M]$. The main claim is that
  \begin{align}\label{vandy_ent}
      N_{[]}(\epsilon, \mathcal{F}, L_2(P_0)) \leq N_{[]}(\epsilon/4,
    \mathcal{F}_{\downarrow}, L_2([0, M]))
      \end{align}
  To see \eqref{vandy_ent}, fix $p \in \smud$ with $h(p_0, p) \leq t$. As explained in the text immediately after Proposition \ref{1dimproved}, the function $\sqrt{p p_0/(p + p_0)}$ is non-increasing on $[0, M]$, and clearly 
  \begin{align*}
      \sqrt{\frac{p(x) p_0(x)}{p(x) + p_0(x)}} \leq \sqrt{p_0(x)} \leq \sqrt{B}
  \end{align*}
  for $0 \leq x \leq M$. Therefore $\sqrt{p p_0/(p + p_0)} \in \mathcal{F}_{\downarrow}$. Further it is easy to check that if $f_L, f_U$ are two functions such that $f_U \leq \sqrt{p_0}$ (note $\sqrt{p p_0/(p + p_0)} \leq \sqrt{p_0}$), then  
  \begin{align*}
      f_L \leq \sqrt{\frac{p p_0}{p + p_0}} \leq f_U \implies 1 - 2 \frac{f_U^2}{p_0} \leq \frac{p_0 - p}{p_0 + p} \leq 1 - 2\frac{f_L^2}{p_0}
  \end{align*}
  and also
  \begin{align*}
      \int \left(\left(1 - 2 \frac{f_L^2}{p_0} \right) - \left(1 - 2
    \frac{f_U^2}{p_0} \right) \right)^2 p_0 &= 4 \int \frac{(f_U -
                                              f_L)^2 (f_U +
                                              f_L)^2}{p_0} \\ &\leq 16
                                                                \int_0^M
                                                                (f_U -
                                                                f_L)^2
                                                                ,
  \end{align*}
  where, in the last inequality, we used $f_L, f_U \leq
  \sqrt{p_0}$. This proves \eqref{vandy_ent} which, together with a
  standard result on the bracketing 
  numbers of  non-increasing functions (see e.g., \cite[Theorem
  2.7.5]{vaartwellner96book} applied with $Q$ being the uniform
  measure on $[0, M]$), allows us to deduce
  \begin{align*}
      \log N_{[]}(\epsilon, \mathcal{F}, L_2(P_0)) \leq \frac{C_{B,
    M}}{\epsilon} \qt{for all $\epsilon > 0$}. 
  \end{align*}
  From here, it immediately follows that $J(\delta) \leq C_{B, M}
  \sqrt{\delta}$ so that
  \begin{equation*}
    \E G(t) \leq C_{B, M} \left( \sqrt{\frac{t}{n}} + \frac{1}{nt} \right)
  \end{equation*}
  and then \eqref{1dimproved_eq} directly follows from Theorem
  \ref{hellexp}. 
\end{proof}

\section{Additional technical results and proofs}\label{additechres}

In this section, we provide the proofs of Lemmas \ref{trivial},
\ref{bds} and \ref{brackd}. Then we provide proofs for the claim in Remark \ref{domain}. We also state and prove
two technical results: Lemma \ref{intchan} and Lemma
\ref{simpleintegral} which were 
used in the proofs of Theorem \ref{adaptation} and Lemma
\ref{ESupResult} respectively.

\begin{proof}[Proof of Lemma \ref{trivial}]
First note that $\frac{p_0 - p}{p_0 + p}$  is decreasing in $p$ (for
fixed $p_0$)  so that
\begin{equation}\label{sieq}
  \frac{p_0 - p_U}{p_0 + p_U} \leq  \frac{p_0 - p}{p_0 + p} \leq
  \frac{p_0 - p_L}{p_0 + p_L} 
\end{equation}
whenever $p_L \leq p \leq p_U$. Combining this with
\begin{align*}
\int_{R} \left( \frac{p_0-p_L}{p_0+p_L}-\frac{p_0-p_U}{p_0+p_U}
  \right)^2 p_0 &= 4\int_{R} \frac{p_0^3
                  (p_U-p_L)^2}{(p_0+p_L)^2(p_0+p_U)^2} 
  \\
  &\leq 4\int_{R} \frac{(p_U-p_L)^2}{p_0}  \\  &\leq 4 \left\{\int_R \left(p_U -
      p_L \right)^{2\ph} \right\}^{1/\ph}  \left\{\int_R
    \frac{1}{p_0^{\qh}} \right\}^{1/\qh} \\ &= 4 \left\{\int_R \left(p_U -
      p_L \right)^{2\ph} \right\}^{1/\ph} \|p_0^{-1}\|_{L_{\qh}(R)} 
\end{align*}
the proof of \eqref{brack.lqbound} is completed. 
%(and \eqref{brack.l2minbound} is just the special case of
%\eqref{brack.lqbound} for $\qh = \infty$). 
\end{proof} 

\begin{proof}[Proof of Lemma \ref{bds}]
Let $\alpha \leq x$ (that is, $\alpha_1 \leq x_1, \ldots, \alpha_d \leq x_d$). Without loss of generality, we assume $p(x) \geq p_0(\alpha)$ since otherwise there is nothing to prove.
\begin{align*}
    \delta^2 \geq \int_{R} (\sqrt{p}-\sqrt{p_0})^2 & \geq \int_\alpha^x (\sqrt{p}-\sqrt{p_0})^2 \\
    &\geq (x_1-\alpha_1) \ldots (x_d-\alpha_d) (\sqrt{p(x)}-\sqrt{p_0(\alpha)})^2
\end{align*}
where the last inequality follows since $p$ and $p_0$ are coordinatewise non-increasing densities.
This gives 
\[
p(x) \leq \left(\sqrt{p_0(\alpha)} + \frac{\delta}{\sqrt{(x_1-\alpha_1) \ldots (x_d-\alpha_d)}}\right)^2
\]
and we can take the infimum over $0\leq \alpha \leq x$ since this relation holds for any such $\alpha$.
For the second bound, we assume $p(x) \leq p_0(\beta)$ and note that 
\begin{align*}
    \delta^2 \geq \int_{R} (\sqrt{p}-\sqrt{p_0})^2 &\geq \int_{x}^\beta (\sqrt{p_0}-\sqrt{p})^2 \\
    &\geq (\beta_1-x_1) \ldots(\beta_d-x_d) (\sqrt{p_0(\beta)}-\sqrt{p(x)})^2.
\end{align*}
This gives 
  \[
  p(x) \geq \left(\sqrt{p_0(\beta)} - \frac{\delta}{\sqrt{(\beta_1-x_1) \ldots (\beta_d-x_d)}} \right)_+^2.
  \]
  This relation holds for any such $\beta$, thus we take the supremum over $\beta \geq x$. The proof is complete.
\end{proof}

\begin{proof}[Proof of Lemma \ref{brackd}]
Fix $p \in \smud$. By \eqref{altrep}, we can write
  \begin{equation*}
    p(x_1, \dots, x_d) := \tilde G\left([x_1, \infty) \times \dots
    \times [x_d, \infty) \right)
\end{equation*}
for some measure $\tilde G$ on $[0, \infty)^d$. We first claim that there
exists a measure $G'$ supported on $R := [a_1, b_1] \times \dots
\times [a_d, b_d]$ such that, for every $x \in R$,   
\begin{align}
p(x_1, \ldots, x_d) &=G'\left([x_1, \infty) \times \dots
    \times [x_d, \infty) \right)  = G'\left([x_1, b_1] \times \dots
    \times [x_d, b_d] \right). \label{c1}
\end{align}
To prove \eqref{c1}, just take $G^{(1)}$ to be the restriction of $\tilde G$ to the
set $[a_1, \infty) \times \dots \times [a_d, \infty)$ and then define
$G'$ as the image measure of $G^{(1)}$ under the transformation
\begin{equation*}
  (u_1, \dots, u_d) \mapsto \left(\min(u_1, b_1), \dots, \min(u_d, b_d)
  \right). 
\end{equation*}
This proves the first equality in \eqref{c1}. The second equality
simply follows from the fact that $G'$ is supported on $R$.

Now let $\mu$ be the measure defined by
\begin{equation*}
  \mu(A) := \frac{1}{\beta-\alpha} \left( G' \left\{u : \left(\frac{b_1 - u_1}{b_1 - a_1}, \dots,
    \frac{b_d - u_d}{b_d - a_d} \right)  \in A \right\} -\alpha \right). 
\end{equation*}
As $G'$ is supported on $R$, it is clear then that $\mu$ is supported
on $[0, 1]^d$. Further
\begin{align*}
  \mu([0, 1]^d) &= \frac{1}{\beta-\alpha} \left(G'(R)-\alpha \right) \\ &= \frac{1}{\beta-\alpha} \left( G' \left([a_1, b_1]
  \times \dots \times [a_d, b_d] \right)-\alpha \right) \\ & = \frac{1}{\beta-\alpha} \left(p(a_1,
\dots, a_d)-\alpha \right) \leq \frac{\sup_{x \in R} p(x)-\alpha}{\beta-\alpha}. 
\end{align*}
Thus $\mu$ is a subprobability measure on $[0, 1]^d$ (subprobability
measure means $\mu[0, 1]^d \leq 1$) when $p$ lies in
the set $\left\{p \in \smud : \sup_{x \in R} p(x) \leq
  \beta\right\}$. Further the distribution function of $\mu$:  
\begin{equation*}
  F_{\mu}(x) := \mu \left([0, x_1] \times \dots \times [0, x_d]
  \right) 
\end{equation*}
is related to $p$ via
\begin{equation*}
  p(x_1, \dots, x_d)-\alpha = (\beta-\alpha) F_{\mu} \left(\frac{b_1 - x_1}{b_1 -
      a_1}, \dots, \frac{b_d - x_d}{b_d - a_d} \right). 
\end{equation*}
Now to prove \eqref{L2ent.eq}, note that if $F_L$ and $F_U$ are
functions on $[0, 1]^d$ such that $F_L \leq F_{\mu} \leq F_U$ and such
that 
\begin{equation*}
  \int_{[0, 1]^d} \left|{F_U} - {F_L} \right|^r \leq \eta^r, 
\end{equation*}
then
\begin{equation*}
  \begin{split}
p_L(x_1, \dots, x_d) &:=\alpha+  (\beta-\alpha) F_{L} \left(\frac{b_1 - x_1}{b_1 -
      a_1}, \dots, \frac{b_d - x_d}{b_d - a_d} \right) \\ &\leq p(x_1,
  \dots, x_d) \\ &\leq p_{U}(x_1, \dots, x_d) := \alpha+ (\beta-\alpha) F_{U} \left(\frac{b_1 - x_1}{b_1 -
      a_1}, \dots, \frac{b_d - x_d}{b_d - a_d} \right)
  \end{split}
\end{equation*}
and
\begin{equation*}
  \int_R \left|{p_U} - {p_L} \right|^r = (\beta-\alpha)^r |R| \int_{[0,
    1]^d} \left|F_U-F_L \right|^r \leq (\beta-\alpha)^r |R|
  \eta^r. 
\end{equation*}
This implies that
\begin{equation*}
  N_{[]}(\epsilon, \F(R, \alpha, \beta), L_r(R)) \leq N_{[]}(\eta,
  \mathcal{A}_d, L_r([0, 1]^d) ) \qt{for $\epsilon = (\beta - \alpha)
    \eta |R|^{1/r}$}, 
\end{equation*}
and inequality \eqref{L2ent.eq} then follows from Theorem \ref{gaoL2}. 
\end{proof}

The following result was used in the proof of Theorem \ref{adaptation}.

\begin{lemma}\label{intchan}
  For every positive $a$, we have
  \begin{align*}
  \E \max_{\ell \in \lin} \sup_{p \in \Ps(\ell)} \int \frac{p_0 - p}{p_0 + p}
         d(P_0 - P_n) &\leq \left(1 + a \right)     \max_{\ell \in \lin}  \E \sup_{p \in \Ps(\ell)} \int \frac{p_0 - p}{p_0 + p}
         d(P_0 - P_n) \\ & + 4 t \sqrt{\frac{\log(e |\lin|)}{n}} + \left(\frac{2}{a}+\frac{1}{3} \right)
          \frac{4\log(e |\lin|)}{n}. 
       \end{align*}
\end{lemma}

\begin{proof}[Proof of Lemma \ref{intchan}]
  For every $u \geq 0$, by the union bound
  \begin{align*}
&    \P \left\{ \max_{\ell \in \lin} \sup_{p \in \Ps(\ell)} \int \frac{p_0 - p}{p_0 + p}
         d(P_0 - P_n) \geq \max_{\ell \in \lin} \E \sup_{p \in \Ps(\ell)} \int \frac{p_0 - p}{p_0 + p}
                   d(P_0 - P_n)  + u \right\} \\
    &\leq \sum_{\ell \in \lin} \P \left\{ \sup_{p \in \Ps(\ell)} \int \frac{p_0 - p}{p_0 + p}
         d(P_0 - P_n) \geq \E \sup_{p \in \Ps(\ell)} \int \frac{p_0 - p}{p_0 + p}
                   d(P_0 - P_n)  + u \right\}. 
  \end{align*}
Now (similarly as in \eqref{bouscon}), we use Bousquet's concentration
inequality for the suprema of 
empirical processes in 
the form stated in \cite[Theorem 12.5]{boucheron2013concentration}
applied to $$X_{i, p} := \frac{1}{2}\frac{p(X_i) - p_0(X_i)}{p(X_i) +
  p_0(X_i)} - \E \frac{1}{2}
\frac{p(X_i)-p_0(X_i)}{p(X_i) + p_0(X_i)},$$ to deduce   
\begin{align*}
 \P \left\{ \sup_{p \in \Ps(\ell)} \int \frac{p_0 - p}{p_0 + p}
         d(P_0 - P_n) \geq \E(\ell) + u \right\} \leq 
         \exp \left(\frac{-n u^2}{8 \left(\E(\ell) + \frac{t^2}{2} + \frac{u}{6} \right)} \right)
\end{align*}
where
\begin{equation*}
  \E(\ell) := \E \sup_{p \in \Ps(\ell)} \int \frac{p_0 - p}{p_0 + p}
                 d(P_0 - P_n). 
\end{equation*}
Therefore
\begin{align*}
&     \P \left\{ \max_{\ell \in \lin} \sup_{p \in \Ps(\ell)} \int \frac{p_0 - p}{p_0 + p}
         d(P_0 - P_n) \geq \max_{\ell \in \lin} \E(\ell)  + u \right\} \\
 & \leq \sum_{\ell \in \lin} \exp \left(\frac{-n u^2}{8 \left(\E(\ell)
   + \frac{t^2}{2} + \frac{u}{6} \right)} \right)  \\
  &\leq |\lin|\exp \left(\frac{-n u^2}{8 \left(\max_{\ell \in \lin}\E(\ell)
   + \frac{t^2}{2} + \frac{u}{6} \right)} \right). 
\end{align*}
Integrating both sides of this inequality from $u = 0$ to $u =
\infty$, we obtain
\begin{align*}
&  \E \left(\max_{\ell \in \lin} \sup_{p \in \Ps(\ell)} \int \frac{p_0 - p}{p_0 + p}
         d(P_0 - P_n) - \max_{\ell \in \lin} \E(\ell) \right)_+ \\ &\leq
  \int_0^{\infty} \min \left\{|\lin|\exp \left(\frac{-n u^2}{8 \left(\max_{\ell \in \lin}\E(\ell)
   + \frac{t^2}{2} + \frac{u}{6} \right)} \right), 1 \right\} du
\end{align*}
where $x_+ := \max(x, 0)$. The trivial inequality $a \leq b + (a -
b)_+$ then gives
\begin{align*}
&  \E \max_{\ell \in \lin} \sup_{p \in \Ps(\ell)} \int \frac{p_0 - p}{p_0 + p}
  d(P_0 - P_n) \\ &\leq \max_{\ell \in \lin} \E(\ell) +   \int_0^{\infty} \min \left\{|\lin|\exp \left(\frac{-n u^2}{8 \left(\max_{\ell \in \lin}\E(\ell)
   + \frac{t^2}{2} + \frac{u}{6} \right)} \right), 1 \right\} du. 
\end{align*}
We now complete the proof of Lemma \ref{intchan} by showing that
\begin{align*}
&  \int_0^{\infty} \min \left\{|\lin|\exp \left(\frac{-n u^2}{8
  \left(\max_{\ell \in \lin}\E(\ell) 
   + \frac{t^2}{2} + \frac{u}{6} \right)} \right), 1 \right\} du \\ &\leq
  a \max_{\ell \in \lin} \E(\ell) + b t^2 + \frac{C(a, b)}{n} \log(e |\lin|)
\end{align*}
for every $a, b > 0$. The integral above is bounded by
\begin{align*}
  &\int_0^{a \max_{\ell \in \lin} \E(\ell) + b t^2} 1 du + \\
  & \int_{a
    \max_{\ell \in \lin} \E(\ell) + b t^2}^{\infty} \min \left\{ |\lin| \exp \left(\frac{-n u^2}{8 \left(\max_{\ell \in \lin}\E(\ell)
    + \frac{t^2}{2} + \frac{u}{6} \right)} \right), 1 \right\} du \\
  &\leq a \max_{\ell \in \lin} \E(\ell) + b t^2 + \int_{a
    \max_{\ell \in \lin} \E(\ell) + b t^2}^{\infty} \min \left\{ |\lin| \exp \left(\frac{-n u^2}{8 \left(\frac{u}{a}
    + \frac{u}{2b} + \frac{u}{6} \right)} \right), 1 \right\} du \\
  &= a \max_{\ell \in \lin} \E(\ell) + b t^2 + \int_{a
    \max_{\ell \in \lin} \E(\ell) + b t^2}^{\infty} \min \left\{ |\lin|\exp \left(\frac{-n u}{8 \left(\frac{1}{a}
    + \frac{1}{2b} + \frac{1}{6} \right)} \right), 1 \right\} du \\
  &\leq a \max_{\ell \in \lin} \E(\ell) + b t^2 + \int_{0}^{\infty}
    \min \left\{ |\lin|\exp \left(\frac{-n u}{K(a, b)} \right), 1 \right\} du
\end{align*}
where $K(a, b) := 8(\frac{1}{a} + \frac{1}{2b} +
\frac{1}{6})$. Letting $T := \frac{K(a, b)}{n} \log |\lin|$, we get
\begin{align*}
&  \int_0^{\infty} \min \left\{|\lin|\exp \left(\frac{-n u^2}{8 \left(\max_{\ell \in \lin}\E(\ell)
                 + \frac{t^2}{2} + \frac{u}{6} \right)} \right), 1 \right\} du \\
  &\leq a \max_{\ell \in \lin} \E(\ell) + b t^2 + \int_{0}^{\infty}
   \min \left\{ |\lin|\exp \left(\frac{-n u}{K(a, b)} \right), 1 \right\} du \\
  &= a \max_{\ell \in \lin} \E(\ell) + b t^2 + \int_{0}^{T}
  \min \left\{|\lin|\exp \left(\frac{-n u}{K(a, b)} \right), 1
    \right\} du \\ &+ \int_T^{\infty}\min \left\{|\lin|\exp \left(\frac{-n u}{K(a, b)} \right), 1
    \right\} du \\
  &\leq a \max_{\ell \in \lin} \E(\ell) + b t^2 + T +    |\lin| \int_T^{\infty}
    \exp \left(\frac{-n u}{K(a, b)} \right) du \\
  &=  a \max_{\ell \in \lin} \E(\ell) + b t^2 + T +    |\lin|
    \frac{K(a, b)}{n} \exp \left(-\frac{nT}{K(a, b)} \right) \\
  &= a \max_{\ell \in \lin} \E(\ell) + b t^2 + \frac{K(a, b)}{n} \log
    |\lin| + \frac{K(a, b)}{n} \\ &= a \max_{\ell \in \lin} \E(\ell) + b
    t^2 + \frac{K(a, b)}{n} \log \left(e|\lin| \right). 
\end{align*}
Now we take
\begin{equation*}
  b = \frac{2}{t} \sqrt{\frac{\log (e |\lin|)}{n}} 
\end{equation*}
to finish the proof of Lemma \ref{intchan}. 

\end{proof}

The following result was used in the proof of Lemma \ref{ESupResult}. 

\begin{lemma}\label{simpleintegral}
  For every $q > 0$, there exists a positive constant $C_q$ such that
  for every $0 < s \leq B$, the following inequality holds:
  \begin{equation}\label{simpleintegral.eq}
    \int_0^s  \left(\log \frac{B}{\epsilon}
      \right)^{q/2} \sqrt{\frac{B}{\epsilon}} d\epsilon \leq C_q
      \sqrt{s B} \left(\log \frac{eB}{s} \right)^{q/2}. 
  \end{equation}
\end{lemma}

\begin{proof}[Proof of Lemma \ref{simpleintegral}]
  Let $I$ denote the integral on the left hand side of
  \eqref{simpleintegral.eq}. By the change of variable $y =
  \frac{1}{2} \log   \frac{B}{\epsilon}$, we get
  \begin{equation}\label{changedvar.eq}
I =  B2^{(q/2)+1} \int_{\alpha_0}^{\infty} e^{-y} y^{q/2} dy \qt{where
  $\alpha_0  := \frac{1}{2} \log \frac{B}{s}$}.  
\end{equation}
We separately consider the two cases $\alpha_0 \leq 1$  and
$\alpha_0 > 1$. When $\alpha_0 \leq 1$,
\begin{align*}
  I &\leq B 2^{(q/2) + 1} \int_0^{\infty} e^{-y} y^{q/2} dy \\
  &= B 2^{(q/2) + 1} \left[\int_0^{\infty} e^{-y} y^{q/2} dy \right] e
    e^{-\alpha_0} \\
  &\leq \sqrt{s B} C_q \qt{provided $C_q \geq  2^{(q/2) + 1}
    \left[\int_0^{\infty} e^{-y} y^{q/2} dy \right] e$} \\
  &\leq \sqrt{s B} C_q \left(\log \frac{eB}{s} \right)^{q/2}. 
\end{align*}
When $\alpha > 1$, let $v$ be the smallest positive integer that is
strictly greater than $q/2$. Integration by parts $v$ times in
\eqref{changedvar.eq} gives
\begin{equation*}
  I \leq C_q B \alpha_0^{q/2}  e^{-\alpha_0} + C_q B
  \int_{\alpha_0}^{\infty} e^{-y} y^{(q/2) - v} dy
\end{equation*}
for some constant $C_q$. As $q/2 < v$, the second integral is bounded
from above by $\int_{\alpha_0}^\infty e^{-y} dy = e^{-\alpha_0} \leq
\alpha_0^{q/2} e^{-\alpha_0}$. We thus obtain
\begin{equation*}
  I \leq C_q B \alpha_0^{q/2} e^{-\alpha_0} = C_q \sqrt{sB}
  \left(\frac{1}{2} \log \frac{B}{s} \right)^{q/2} \leq C_q \sqrt{sB}
  \left(\frac{1}{2} \log \frac{eB}{s} \right)^{q/2}
\end{equation*}
which completes the proof of \eqref{simpleintegral.eq}. 
\end{proof}

\subsection{Proofs of Claims in Remark \ref{domain}}\label{domain_claims}
\begin{proof}[Proof for $\tilde p_0$ begin a SMU density in Remark \ref{domain}]
First, we can easily check that $\tilde p_0(u_1, \ldots, u_d)$ is nonnegative and integrates up to 1. Indeed, 
\begin{align*}
\int_0^1 \ldots \int_0^1 &\tilde p_0(u_1, \ldots, u_d) du_1 \ldots du_d \\
&=\int_0^1 \ldots \int_0^1 p_0(u_1M_1, \ldots, u_d M_d) M_1 \ldots M_d du_1 \ldots du_d \\
&= \int_0^{M_1} \ldots \int_0^{M_d} p_0(x_1, \ldots, x_d) dx_1 \ldots dx_d=1.
\end{align*}
Now we show that $\tilde p_0$ has a form of scale mixtures of uniform densities. By definition of $p_0$ being the SMU density, we represent 
\[
p_0(u_1, \ldots, u_d) = \int_0^\infty \ldots \int_0^\infty I(0\leq u_1 \leq \theta_1) \ldots I(0\leq u_d \leq \theta_d) dG(\theta_1, \ldots, \theta_d).
\]
Then, by the definition of $\tilde p_0$ followed by the change of variables (letting $\tau_i = \theta_i/M_i$),
\begin{align*}
&\tilde p_0(u_1, \ldots, u_d) = p_0(u_1M_1, \ldots, u_d M_d) M_1 \ldots M_d\\
&=\int_0^\infty \ldots \int_0^\infty I(0\leq u_1M_1 \leq \theta_1) \ldots I(0\leq u_d M_d \leq \theta_d)  M_1 \ldots M_d dG(\theta_1, \ldots, \theta_d) \\
&=\int_0^\infty \ldots \int_0^\infty I(0\leq u_1 \leq \tau_1) \ldots I(0\leq u_d \leq \tau_d)  M_1 \ldots M_d dG(M\tau_1, \ldots, M\tau_d) \\
&=\int_0^\infty \ldots \int_0^\infty I(0\leq u_1 \leq \tau_1) \ldots I(0\leq u_d \leq \tau_d)  d \tilde G_M(\tau_1, \ldots, \tau_d) 
\end{align*}
by defining a new probability measure $\tilde G_M(\tau_1, \ldots, \tau_d) := G(M_1\tau_1, \ldots, M_d\tau_d)$. This shows the claim.
\end{proof}
    
\begin{proof}[Proof of \eqref{mle_scaling} in Remark \ref{domain}] We denote the class of scale mixture of uniform (SMU) densities by $\mathcal{P}$.
With the original data $X_1, \ldots, X_n$ having a density $p_0 \in \mathcal{P}$, we maximize the log likelihood $\sum_{i=1}^n \log p(X_{i1}, \ldots, X_{id})$ over the SMU class $\mathcal{P}$. 
%Since we assume that the true density $p_0$ is included in $\mathcal{P}$, the maximizer is also included in $\mathcal{P}$. 
Now we consider the scaled data $X_1/M_1, \ldots, X_d/M_d$. Since $X_1, \ldots, X_n$ is from $p_0$, we know that $X_1/M_1, \ldots, X_d/M_d$ are from a density $\tilde p_0 (u_1, \ldots, u_d) = p_0(u_1M_1, \ldots, u_dM_d) M_1 \ldots M_d$ and we showed that $\tilde p_0 \in \mathcal{P}$ above. Since $p_0$ and $\tilde p_0$ have such correspondence, $\mathcal{P}$ is transformed as
\[
\mathcal{Q} = \{ q(u_1, \ldots, u_d) =  p(u_1M_1, \ldots, u_d M_d)M_1\ldots M_d | p \in \mathcal{P}\}.
\]
For every $p \in \mathcal{P}$, by considering 
\begin{align*}
\sum_{i=1}^n \log p(X_i) &= \sum_{i=1}^n \log p(M_1\tilde X_{i1}, \ldots, M_d \tilde X_{id}) \\
&= \sum_{i=1}^n \log q(\tilde X_{i1}, \ldots, \tilde X_{id}) - n\log (M_1\ldots M_d),
\end{align*}
maximizing the likelihood using original data over $p \in \mathcal{P}$ must be equivalent to maximizing $q \in\mathcal{Q}$ using the scaled data. This shows 
\[
\tilde p_{n,d}^{\text{SMU}} (u_1,\ldots, u_d) = \hat p_{n,d}^{\text{SMU}}(u_1M_1, \ldots, u_dM_d) M_1 \ldots M_d.
\]
\end{proof}

%%%%%%%%%%%%%%%%%%%%%%%%%%%%%%%%%%%%%%%%%%%%%%
%% Acknowledgements                         %%
%% should be provided in the                %%
%% Acknowledgements section.                %%
%%%%%%%%%%%%%%%%%%%%%%%%%%%%%%%%%%%%%%%%%%%%%%
\begin{acks}[Acknowledgments]
We are grateful to Frank Fuchang Gao for many helpful discussions about the results in \cite{gao2013book}. We also thank the two anonymous reviewers for their insightful comments, which greatly improved the paper’s content and presentation. 
\end{acks}

%%%%%%%%%%%%%%%%%%%%%%%%%%%%%%%%%%%%%%%%%%%%%%
%% Funding information, if any,             %%
%% should be provided in the                %%
%% funding section.                         %%
%%%%%%%%%%%%%%%%%%%%%%%%%%%%%%%%%%%%%%%%%%%%%%
\begin{funding}
Arlene K. H. Kim was supported by the National Research Foundation of
Korea (NRF) grant funded by the Korea government (MSIT)
(No. RS-2023-00219212 and RS-2023-NR076465). 

\noindent Adityanand Guntuboyina was supported by the National Science
Foundation (NSF) grant DMS-2210504.
\end{funding}

\bibliographystyle{imsart-number} % Style BST file (imsart-number.bst or imsart-nameyear.bst)
\bibliography{AG}       % Bibliography file (usually '*.bib')

%% or include bibliography directly:
%\begin{thebibliography}{9}
%
%\bibitem{r1}
%\textsc{Billingsley, P.} (1999). \textit{Convergence of
%Probability Measures}, 2nd ed.
%Wiley, New York.
%\MR{1700749}

%\bibitem{r2}
%\textsc{Bourbaki, N.}  (1966). \textit{General Topology}  \textbf{1}.
%Addison--Wesley, Reading, MA.

%\bibitem{r3}
%\textsc{Ethier, S. N.} and \textsc{Kurtz, T. G.} (1985).
%\textit{Markov Processes: Characterization and Convergence}.
%Wiley, New York.
%\MR{838085}

%\bibitem{r4}
%\textsc{Prokhorov, Yu.} (1956).
%Convergence of random processes and limit theorems in probability
%theory. \textit{Theory  Probab.  Appl.}
%\textbf{1} 157--214.
%\MR{84896}
%\end{thebibliography}

\end{document}